%% file: geometric-groupoids.tex
\documentclass[12pt]{amsart}
\input{preamble.tex}

\title[Derived smooth and Banach higher groupoids]
{Derived Smooth and Banach Higher Groupoids:\\
 Representability and Descent}
\author{Qingyun Zeng}
\date{September 23, 2026; revision 4}

\begin{document}

\begin{abstract}
We study homotopy theories of higher groupoids in
specified geometric categories, distinguishing geometric
representability from realization and descent.
For ordinary smooth and Banach-open sheaves, finite matching
and germwise Kan calculus give full Brown categories of
fibrant objects. A crossing-axes example obstructs the
unchanged full structure on ordinary geometric higher
groupoids. We prove the empty-compatible ordinary smooth
geometric iCFO directly in the open-covering-family topos,
and its split-Banach version on a fixed chart-compatible
plot site with local kernel-product closure. These results include
all finite unique-horn bounds and their union. A point-only
site counterexample shows why an arbitrary Banach plot site
does not suffice for the geometric assertion.

For represented derived models, finite matching and an
explicit outer-prism filtration give the structural Brown
calculus. Nuiten's finite-geometric realization theorem
then gives the finite derived-smooth CFO and its
localization. We prove the enriched realization base-change
theorem on the original small Kan-enriched site, giving
the outer-unbounded hypercomplete CFO. Its geometric horn
conditions occur in positive degrees, alongside Reedy model
fibrancy in every degree.

For Banach domains, structured charts, valuewise
hypercompletion and open-local algebraic representatives
give geometric finite limits through structured spectra
and strict-open gluing. Coherent split-open
refinement and a separate, set-sized enriched localization
argument establish hypersheaf representability. Combined with
realization base change, they give the represented Brown CFO.
The same hypercompleted weak-equivalence class is used
at every outer bound. A square-zero structured test detects
nonordinary self-intersections of the chosen Banach charts,
including infinite-dimensional charts.

For connective associative affines, homotopy split horns
give pointwise prestack constructions. Pridham's finite
effectivity theorem and the pointwise diagonal-fibration
lemma give the corresponding CFOs for his Artin and
Deligne--Mumford classes. Absolute horns ensure Kan
diagonal objects, while relative horns control fibrations.
The associative theory uses quasi-isomorphisms of algebras
and realization in presheaves.
\end{abstract}

\maketitle
\clearpage
\tableofcontents

\input{introduction.tex}

\input{foundations.tex}
\input{groupoids.tex}
\input{derived-geometry.tex}
\input{represented-criterion.tex}
\input{banach-geometry.tex}
\input{noncommutative.tex}
\input{comparisons.tex}

\appendix
\input{enriched-realization.tex}
\input{split-open-hyperdescent.tex}

\printbibliography
\end{document}

%% file: preamble.tex
\usepackage{mathpazo}
\usepackage{avant}
\usepackage{amsmath,amssymb,tensor}
\usepackage{mathrsfs}
\usepackage{tikz-cd}
\usepackage{booktabs,longtable,array}
\usepackage[backend=biber,style=alphabetic,maxbibnames=99]{biblatex}
\usepackage[hidelinks]{hyperref}
\hypersetup{
  pdftitle={Derived Smooth and Banach Higher Groupoids: Representability and Descent},
  pdfauthor={Qingyun Zeng}
}
\AtBeginBibliography{%
  \setlength{\emergencystretch}{3em}%
  \raggedright
  \interlinepenalty=10000
}

\newtheorem{thm}{Theorem}[section]
\newtheorem{cor}[thm]{Corollary}
\newtheorem{lem}[thm]{Lemma}
\newtheorem{prop}[thm]{Proposition}

\newtheorem{maintheorem}{Theorem}

\theoremstyle{definition}
\newtheorem{defn}[thm]{Definition}
\newtheorem{example}[thm]{Example}

\newtheorem{construction}[thm]{Construction problem}
\theoremstyle{remark}
\newtheorem{rem}[thm]{Remark}
\newtheorem{question}[thm]{Question}
\numberwithin{equation}{section}

\newcommand{\R}{\mathbb R}

\newcommand{\Z}{\mathbb Z}

\newcommand{\CA}{\mathcal A}

\newcommand{\CC}{\mathcal C}

\newcommand{\CM}{\mathcal M}

\newcommand{\CT}{\mathcal T}
\newcommand{\E}{\mathcal E}

\newcommand{\Hsm}{\mathcal H_{\mathrm{sm}}^\wedge}
\newcommand{\KSh}{\mathcal K(\E)}

\newcommand{\Aff}{\mathsf{Aff}}
\newcommand{\dMfd}{\mathsf{dMfd}}
\newcommand{\dLie}{\mathsf{dLie}}
\newcommand{\dSt}{\mathsf{dSt}}
\newcommand{\Sh}{\mathsf{Sh}}
\newcommand{\PSh}{\mathsf{PSh}}
\newcommand{\sSet}{\mathsf{sSet}}
\newcommand{\Set}{\mathsf{Set}}
\newcommand{\LieGpd}{\mathsf{Lie}_{\infty}\mathsf{Grpd}}

\newcommand{\y}{\mathsf y}
\newcommand{\op}{\mathrm{op}}

\DeclareMathOperator{\Map}{Map}
\DeclareMathOperator{\Hom}{Hom}

\DeclareMathOperator{\Spec}{Spec}

\DeclareMathOperator{\Ex}{Ex}
\DeclareMathOperator{\Sd}{Sd}
\DeclareMathOperator{\id}{Id}
\DeclareMathOperator{\ev}{ev}

\DeclareMathOperator{\Tot}{Tot}

\DeclareMathOperator{\Fun}{Fun}
\DeclareMathOperator*{\colim}{colim}

%% file: introduction.tex
\section{Introduction}
\label{sec:introduction}

Higher groupoids provide geometric presentations of
homotopy-coherent equivalence relations. Their geometric
levels are important: they retain smooth parameters,
derived intersections, or associative algebra tests
which an abstract homotopy type alone does not record.
This produces a basic tension. The category of geometric
objects may not contain the pullbacks required by a
homotopy theory, whereas passing immediately to all
sheaves or prestacks loses the assertion that the
levels of a presentation are geometric.

This paper studies that tension in several specified settings.
The common method is
\[
 \text{geometric limits and covers}
 \longrightarrow
 \text{represented horns and paths}
 \longrightarrow
 \text{homotopy theory}.
\]
The resulting homotopy structures depend on both the geometry
and the realization functor. Ordinary smooth geometric groupoids
admit an empty-compatible incomplete Brown structure; the
split-Banach version requires compatibility of the fixed
plot site with geometric atlases and pullbacks. Ordinary
sheaves and the represented models considered below admit
full Brown structures. The finite derived-smooth application
uses Nuiten's realization theorem, while the associative
application combines Pridham's finite effectivity theorem
with a pointwise diagonal argument. Enriched realization
base change and split-open hyperdescent supply the
unbounded hypercomplete derived-smooth and Banach
constructions. This common method separates geometry,
model structure and realization while keeping the
resulting geometries distinct.

\subsection{The geometric obstruction and its repairs}

Write $cM$ for the constant simplicial manifold on $M$.
Every map between constant simplicial manifolds satisfies the
positive relative horn condition. Consequently a full Brown
structure with that unchanged fibration class would require a
geometric pullback of
\[
 c\R^2\xrightarrow{xy}c\R\xleftarrow{0}c*.
\]
The crossing-axes functor is not represented by a
smooth manifold; the same obstruction holds for
the full Banach manifold category
(Proposition~\ref{prop:crossing-axes} and
Proposition~\ref{prop:ordinary-banach-obstruction}).
This explains why the ordinary geometric iCFO
and the independent full sheaf-theoretic CFO
are different results.
The geometric Banach theorem requires atlases from the
already chosen plot site and closure under the relevant
kernel-product pullbacks. Without that compatibility,
the cubic/product cospan on a point-only site obstructs
even the mandatory acyclic pullback
(Proposition~\ref{prop:ordinary-insufficient-site}).

For a derived repair, finite homotopy limits must
exist as actual geometric objects, and the functor
of points must retain them. Even then, a full
ordinary category of homotopy-coherent diagrams
does not automatically satisfy Brown's ordinary
axioms. We use strict Reedy-fibrant presheaf models,
with weak representability retained through
model equivalences. Finite weighted matching
and an explicit prism filtration give their
geometric pullbacks and outer paths.
The final acyclic-pullback axiom is supplied by
the appropriate realization theorem, not by
levelwise right properness alone.

\subsection{Main results}

Theorem~\ref{thm:geometric-icfo} proves the ordinary smooth
iCFO by finite horn constructions, prescribed-point
submersion locality, a geometric acyclic-boundary upgrade,
and represented outer paths. Constant geometric tests
identify every existing categorical fibration pullback
with the constructed levelwise pullback.
Theorem~\ref{thm:ordinary-banach-icfo} verifies the same
inputs for an explicitly chart-compatible ordinary
split-Banach geometry. Both results include every finite
unique-horn bound, including zero, and their union, while
retaining all simplicial degrees. Their geometric fibration
conditions concern positive relative horns.

\begin{maintheorem}[Represented homotopy structures]
\label{main:represented}
Proposition~\ref{prop:geometric-realization-criterion}
gives full Brown categories of represented higher groupoids
from finite-limit representability, stable geometric horn
classes, and its two stated realization hypotheses.
Using Nuiten's finite realization theorem and Pridham's
finite effectivity theorem, it applies to finite
derived-smooth presentations with ordinary realization
and to the specified pointwise associative prestack models.

Theorem~\ref{lem:derived-hypercomplete-base-change}
proves $\mathbf{PB}_\tau(\mathcal T)$ on the original
small Kan-enriched site. It gives the outer-unbounded
derived-smooth CFO with hypercompleted realization
in Theorem~\ref{thm:derived-reedy-hypercomplete}.
The constructions preserve finite outer bounds and the
unions of finite-bound subcategories whenever the
realization hypotheses hold on those categories.
\end{maintheorem}

The finite derived-smooth localization theorem
uses Nuiten's geometric presentation result.
For unbounded presentations, the essential image and
localization form a further comparison problem, separate
from the ordinary-sheaf acyclic-pullback condition
$\mathbf{AP}_\infty$.
For associative affines the geometric classes
are Pridham's homotopy Artin and
Deligne--Mumford classes, with their split
classical affine truncation condition.
Derived sections use the finite Cech effectivity theorem
in addition to classical splitting.

\begin{maintheorem}[Banach geometry and represented CFOs]
\label{main:banach}
For the Banach pregeometry with open embeddings as admissibles,
the structured-spectrum and strict-open-gluing theorems recalled
below give a valuewise-hypercomplete geometric gluing class
with finite limits. Its functor of points is fully faithful
in the ordinary strict-open sheaf category.

Both clauses of $\mathbf{SO}(\mathcal C)$ hold, so these
represented functors are hypersheaves. With generated-smooth
maps admitting target-open homotopy sections as the horn class,
$\mathbf{PB}_\tau(\mathcal C)$ gives the represented full Brown CFO.
Every outer bound uses the same hypercompleted
realization weak equivalences. Geometric horn conditions
are imposed in positive degrees; the ambient Reedy
requirement applies in every degree, including zero.
\end{maintheorem}

The geometric-limit construction is
Theorem~\ref{thm:banach-geometric-limits}.
The hypersheaf and Brown applications are
Theorems~\ref{thm:banach-representability}
and~\ref{thm:banach-cfo}.
The geometric category is generated from the chosen
Banach charts through valuewise hypercompletion, finite
limits and strict-open gluing. Open-local algebraic
representatives of hyp-affine maps supply the local
pullbacks used in this construction.

For a Banach space $H$ included in the chosen test category, let
$\operatorname{Chart}(H)$ be its structured chart
\eqref{eq:banach-chart}. The derived nature is visible in the
self-intersection at the origin,
\[
 D_H=\operatorname{Chart}(*)\times^h_{\operatorname{Chart}(H)}
                         \operatorname{Chart}(*).
\]
Ordinary chart tests see a point, whereas the square-zero structured
test of \eqref{eq:banach-square-zero} sees $H_{\mathrm{add,disc}}$,
the underlying additive group regarded as a discrete space
(Proposition~\ref{prop:banach-derived-intersection}).
In particular this construction is not obtained
by adjoining an ordinary classifying sheaf.

\subsection{A dictionary of conclusions}

Table~\ref{tab:geometry-dictionary} collects the distinct conclusions.

\begin{table}[tbp]
\caption{Homotopy structures in the geometries considered here.}
\label{tab:geometry-dictionary}
\centering
\begin{tabular}{@{}>{\raggedright\arraybackslash}p{0.24\textwidth}
                  >{\raggedright\arraybackslash}p{0.30\textwidth}
                  >{\raggedright\arraybackslash}p{0.40\textwidth}@{}}
\toprule
Setting & Homotopy structure & Essential distinction\\
\midrule
Ordinary smooth or chart-compatible split-Banach geometric groupoids
& iCFO, including finite unique-horn bounds and their union
& Fixed-site atlas and kernel-product hypotheses in the Banach case;
  crossing axes obstruct a full CFO when that cospan is admitted\\
\addlinespace
Locally Kan smooth or Banach-open sheaves
& Full Brown CFO
& Germwise weak maps; geometric representability is not automatic\\
\addlinespace
Finite derived-smooth represented models
& Full CFO and finite-geometric localization via Nuiten's theorem
& Ordinary-sheaf realization; finite outer degree does not truncate
  derived tests\\
\addlinespace
Unbounded derived-smooth represented models
& Full CFO by enriched realization base change
& Hypercompleted realization; the unbounded essential image is a
  further comparison problem\\
\addlinespace
The represented Banach model
& Hypersheaf representation by split-open hyperdescent;
  full CFO by enriched realization base change
& One hypercompleted realization class at every outer bound\\
\addlinespace
Associative Artin/DM prestacks
& Full CFOs using finite effectivity and the pointwise diagonal lemma;
  finite bounds included
& Objectwise realization and absolute horns;
  quasi-isomorphism-based affine tests\\
\bottomrule
\end{tabular}
\end{table}

\subsection{Sources and organization}

The sheaf-theoretic construction follows Brown and the
smooth sheaf presentation of Nikolaus--Schreiber--Stevenson
\cite{Brown1973,NSS2015}.
Rogers--Zhu prove their ordinary geometric iCFO theorem
under their covering hypotheses
\cite{RogersZhu2020}. We adapt their finite-horn,
submersion-locality and path constructions to the
open-covering-family topos, including the empty cover.
The ordinary Banach application additionally uses fixed-site
atlases and split-submersion pullbacks.
Geometric higher-groupoid frameworks also occur
in \cite{BehrendGetzler2017}.
Nuiten's derived-smooth geometry and finite
localization theorem are used with their precise
grading and geometric hypotheses
\cite{NuitenThesis2018}.
The represented-hypergroupoid calculus and the
associative Artin/DM classes follow Pridham
\cite{Pridham2013,PridhamNC2023}.
Structured spectra and effective open gluing
come from Lurie's geometry framework
\cite{LurieDAGV2009}.

The emphasis here is the explicit Banach
construction and the separation of geometric
closure from realization in the represented models.

Section~\ref{sec:foundations} fixes the
presentation conventions. Section~\ref{sec:groupoids}
proves the ordinary geometric iCFOs, their finite-bound
versions, the site and full-CFO obstructions, and the
independent sheaf repair.
Section~\ref{sec:derived} supplies the derived-smooth
case and the matching/path calculations, distilled
in Section~\ref{sec:represented-criterion}.
Sections~\ref{sec:banach} and
\ref{sec:noncommutative-geometry} construct the
Banach and associative cases. Finally,
Section~\ref{sec:comparisons} discusses comparisons
and further applications.
Appendices~\ref{app:enriched-realization}
and~\ref{app:split-open-hyperdescent} give the
enriched lifting, replacement, augmented refinement
and localization proofs.

\subsection*{Acknowledgments}

This paper grew out of part of the author's Ph.D. thesis work.
The author is grateful to Jonathan Block for his advice and support,
to David White for discussions on the homotopy theory of algebras
over operads, and to Sylvain Lavau and Jim Stasheff for discussions
on $L_\infty$-algebroids and singular foliations.

%% file: foundations.tex
\section{Conventions and the homotopy-theoretic problem}
\label{sec:foundations}

\subsection{Geometric objects, presentations, and ambient models}

We distinguish three levels of structure. A geometric category
contains the spaces on which a particular geometry is defined.
A higher groupoid is a simplicial presentation whose levels and
horn maps satisfy geometric conditions. An ambient sheaf or
presheaf model category supplies ordinary objects and morphisms
on which the axioms of homotopy theory can be imposed.
Passing between these levels requires a theorem; the existence
of an ambient model structure does not make a subcategory of
geometric presentations a category of fibrant objects.

Throughout, $\mathcal S$ denotes the infinity-category of spaces.
An infinity-category, its ordinary homotopy category, and an
ordinary category presenting it are not identified.
The notation $\mathcal D[W^{-1}]$ means infinity-localization;
a statement about its mapping spaces is stronger than a
Gabriel--Zisman localization assertion.
All constructions are made in compatible universes. A test
site is small relative to the ambient presheaf universe, and
the specified geometric finite constructions are retained
within the chosen size bounds.

For an infinity-site $(\CT,\tau)$ we distinguish
\[
 \PSh(\CT),\qquad
 \Sh_\tau(\CT),\qquad
 \Sh_\tau(\CT)^\wedge.
\]
These are respectively presheaves, ordinary space-valued
infinity-sheaves, and hypersheaves. In the middle category,
descent means covering-family descent. The final superscript
indicates hypercompletion. The associative examples use the
presheaf category.

Site and envelope symbols are local to their sections.
The ordinary smooth site $\mathcal C$ and sheaf topos $\E$
in Section~\ref{sec:groupoids} are not the Banach test
site and geometric envelope denoted by the same letters
in Section~\ref{sec:banach}.

\subsection{Brown and incomplete Brown structures}

\begin{defn}\label{def:brown-conventions}
A category of fibrant objects consists of an ordinary category
with classes $W$ of weak equivalences and $F$ of fibrations
satisfying the following conditions.
Isomorphisms belong to both classes, $W$ has two-out-of-three,
and $F$ is closed under composition. A terminal object and
finite products exist, and every map to the terminal object
is a fibration. Pullbacks along fibrations exist and are
fibrations; the pullback of a map in $F\cap W$ is again in
$F\cap W$. Finally, each diagonal has a factorization
\[
 X\xrightarrow{s}PX\xrightarrow{(e_0,e_1)}X\times X,
 \qquad s\in W,\quad(e_0,e_1)\in F.
\]
We use ``full Brown CFO'' to emphasize existence of all
pullbacks along fibrations, not closure under every
ordinary finite limit.
\end{defn}

These are Brown's axioms \cite[Section~1]{Brown1973}.
The path objects in this paper are functorial when the
presentation supplies the indicated cotensors.
These axioms specify weak equivalences, fibrations and paths;
a Quillen model structure additionally specifies cofibrations
and their axioms.

In the incomplete Brown convention, finite products and path
objects are still required. Every pullback along an acyclic
fibration must exist, and its projection must again be an
acyclic fibration. Pullbacks along general fibrations need
not exist; whenever such a pullback exists, its projection
must be a fibration \cite[Definition~2.1]{RogersZhu2020}.
These are distinct existence and stability requirements.
Theorem~\ref{thm:geometric-icfo} proves the required
empty-compatible adaptation directly in the ordinary
open-covering-family topos. Its categorical pullback
argument separates stability of existing fibration
pullbacks from existence of all acyclic-fibration
pullbacks. The ordinary split-Banach version,
Theorem~\ref{thm:ordinary-banach-icfo}, requires explicit
compatibility of the fixed plot site with geometric
atlases and the local kernel-product pullback models.

\subsection{Two simplicial directions}

A represented model will be an ordinary strict simplicial
object in a simplicial presheaf model category. Its outer
index records the groupoid presentation. Its inner
simplicial direction records derived mapping spaces.
An outer bound does not bound the inner homotopy groups
of its values on derived tests.

For a finite simplicial set $K$, write $\{K,X\}$ for the
outer weighted limit of such a diagram. For an outer
Reedy-fibrant object, this ordinary limit computes the
homotopy weighted limit. A geometric positive relative
horn of $f:X\to Y$ is the associated map
\[
 X_m\longrightarrow
 Y_m\times^h_{\{\Lambda^i[m],Y\}}
                    \{\Lambda^i[m],X\},\qquad m\ge1.
\]
The required class of geometric maps will be specified
in each setting. Geometric horn conditions begin in degree
one. For fibrations in the represented CFOs, the degree-zero
map $f_0$ is subject only to the ambient Reedy requirement,
which also applies in higher degrees.

The outer path object is
\[
 (PX)_m=\{\Delta[m]\times\Delta[1],X\}.
\]
This outer cotensor differs from a levelwise internal model
cotensor, whose endpoint map can fail the geometric horn
condition, as illustrated after
Proposition~\ref{prop:derived-outer-path}.

\subsection{Grading and algebraic conventions}

Derived smooth functions and connective associative
algebras use nonnegative \emph{homological} grading,
equivalently nonpositive cohomological grading.
A geometric Chevalley--Eilenberg algebra in positive
cohomological degrees is a different object.
Changing signs on the labels without also changing
the differential does not identify these conventions.

The Banach construction uses space-valued local structures.
The associative examples use unital algebras over a
characteristic-zero field $k$. Their quasi-isomorphism
theory is distinct from Morita localization.

%% file: groupoids.tex
\section{Smooth higher groupoids and fibrant presentations}
\label{sec:groupoids}

\subsection{The ordinary smooth site}

Let $\CC$ be a small ordinary site of Euclidean open sets
$U\subseteq\R^d$, $d\ge0$, with all smooth maps and the usual
open-covering families. Include the empty covering family
over the empty object. Set
\[
 \E=\Sh_{\Set}(\CC),\qquad
 s\E=\Fun(\Delta^\op,\E).
\]
This is an ordinary sheaf topos and its category of simplicial
objects. It is not a category of derived charts.

Euclidean opens form a dense smooth site. Smooth Yoneda
$\y M(U)=C^\infty(U,M)$ embeds finite-dimensional smooth
manifolds fully faithfully and preserves existing ordinary
limits. For $r\ge1$ an integer, let $D^d_{1/r}$ be the open
Euclidean ball of radius $1/r$ centered at $0$. The germ points
\[
 p_d^*F=\colim_{r\to\infty}F(D^d_{1/r}),\qquad d\ge0,
\]
are jointly conservative
\cite[Proposition~4.12 and Lemma~4.13, p.~44]{NSS2015}.
Each inverse-image functor preserves finite limits and
colimits. At $d=0$, the functor is evaluation at the point.
The entire family is needed; ordinary point-valued
surjectivity is not the same as local smooth lifting.

The family detects epimorphisms as well as isomorphisms:
apply the points to the image subobject and use
conservativity.
These facts concern the actual covering-family topology.
Sheafification does not change a germ: a local gluing or
equality at a point is represented on one smaller neighborhood.
Filtered colimits of sets commute with finite limits, and this
description also makes germs commute with sheaf colimits.
Germwise equality is local equality of sections, and germwise
existence gives local lifts. Translation and restriction to
balls identify every pointed Euclidean-domain germ with one
of the specified $p_d$; no further point family is needed.

\subsection{Matching objects and local fibrations}

For a finite simplicial set $K$, let $M_KX$ be the internal
object of maps from $K$ to $X\in s\E$.

\begin{lem}[Finite matching]\label{lem:finite-matching}
The object $M_KX$ is a finite limit, and for every point $p$
\[
 p^*M_KX\cong\Hom_{\sSet}(K,p^*X).
\]
\end{lem}

\begin{proof}
Build $K$ by finitely many skeletal cell attachments along
boundaries. Mapping out replaces these pushouts by
pullbacks, beginning with $M_{\Delta[n]}X=X_n$ and
$M_\varnothing X=1$. Only finite limits are used.
Apply the left-exact inverse-image functor.
\end{proof}

\begin{defn}
An object $X$ is locally Kan if
\[
 X_n\longrightarrow M_{\Lambda^i[n]}X
\]
is epi for every $n\ge1$ and $0\le i\le n$. A map $f:X\to Y$
is a local fibration if every relative horn map
\begin{equation}\label{eq:relative-horn}
 X_n\longrightarrow
 M_{\Lambda^i[n]}X
 \times_{M_{\Lambda^i[n]}Y}Y_n
\end{equation}
is epi for these same indices. Let $F$ denote this class of local
fibrations, let $W$ be the stalkwise weak homotopy equivalences,
and let $\KSh$ be the full category of locally Kan simplicial sheaves.
\end{defn}

There is no degree-zero horn condition. By
Lemma~\ref{lem:finite-matching}, local Kan objects have
exactly Kan stalks, and local fibrations have exactly Kan
fibration stalks. The empty simplicial sheaf is locally Kan;
its map to the terminal object is a fibration, but not a
weak equivalence.

\begin{thm}[Brown--NSS]\label{thm:sheaf-cfo}
With the preceding classes, $\KSh$ is a full Brown category
of fibrant objects. Functorial paths are
\[
 PX=X^{\Delta[1]},\qquad
 (X^K)_n=M_{\Delta[n]\times K}X.
\]
\end{thm}

\begin{proof}
The terminal sheaf and finite products have Kan stalks.
Every map $X\to1$ is a fibration. Isomorphisms, composition
of fibrations and two-out-of-three for weak equivalences
are detected at stalks.

For a diagram $X\to Y\leftarrow Z$ in which $X\to Y$ is a
local fibration and the three objects are locally Kan,
form the levelwise sheaf pullback $Q=X\times_Y Z$. Its
stalk is the pullback of a Kan fibration over a Kan
complex. Thus $Q$ is locally Kan and $Q\to Z$ is a local
fibration. If $X\to Y$ was also weak, each stalk is a
trivial Kan fibration, and its pullback is again trivial.
This proves both existence and stability for all
fibration and acyclic-fibration pullbacks.

The weights $\Delta[n]\times\Delta[1]$ are finite, so
\[
 p^*(PX)\cong(p^*X)^{\Delta[1]}.
\]
The standard cotensor theorem for simplicial sets makes
\[
 X\xrightarrow{s}PX\xrightarrow{(e_0,e_1)}X\times X
\]
a stalkwise weak equivalence followed by a local
fibration. The middle object is locally Kan, and the
construction is functorial. These are all Brown axioms.
\end{proof}

This is the established stalkwise construction
\cite[Section~1, pp.~420--421]{Brown1973}, explicitly
\cite[Proposition~4.18, p.~45]{NSS2015} in the smooth
setting. Here ``full'' means that every pullback
\emph{along a fibration} exists; closure of $\KSh$
under all ordinary finite limits is a stronger requirement.

\begin{prop}[Acyclic boundary criterion]\label{prop:boundary-criterion}
A map between locally Kan simplicial sheaves lies in
$F\cap W$ precisely when
\[
 X_n\longrightarrow
 M_{\partial\Delta[n]}X
 \times_{M_{\partial\Delta[n]}Y}Y_n
\]
is epi for every $n\ge0$.
\end{prop}

\begin{proof}
At every stalk this is the characterization of trivial
Kan fibrations by lifting against all boundary inclusions.
Finite matching commutes with stalks, and the points
detect epimorphisms.
\end{proof}

At $n=0$ the condition is $X_0\to Y_0$ epi. These are
internal acyclic Kan maps, also called internal hypercovers.
Their definition is internal to the sheaf topos; geometric
hypercovers additionally require represented levels and
matching maps.

\subsection{Internal Kan replacement}

Define
\[
 (\Ex_\E X)_n=M_{\Sd\Delta[n]}X,\qquad
 RX=\Ex^\infty_\E X=\colim_{r\ge0}\Ex_\E^rX.
\]
The transition maps are induced by the last-vertex maps.
The colimit is taken in the sheaf category.

\begin{prop}\label{prop:internal-ex}
There is a natural weak equivalence $\eta_X:X\to RX$
with locally Kan target. The functor $R$ preserves
stalkwise weak equivalences and finite limits in $s\E$.
\end{prop}

\begin{proof}
Since $\Sd\Delta[n]$ is finite, points commute with
each $\Ex_\E$ stage. They also preserve the colimit, so
\[
 p^*RX\cong\Ex^\infty(p^*X).
\]
The classical simplicial-set Kan replacement theorem
proves the first assertion. Naturality of $\eta$ and
two-out-of-three prove preservation of weak maps.
Each finite stage preserves finite limits, and filtered
colimits in a Grothendieck topos are exact. Thus $R$
preserves finite limits in the ambient category $s\E$.
\end{proof}

Write $\Hsm=\Sh_\infty(\CC)^\wedge$ for the hypercomplete
infinity-topos of space-valued sheaves on the ordinary smooth site
$\CC$ fixed above.
\begin{thm}\label{thm:smooth-localization}
Inclusion of locally Kan simplicial sheaves induces an
equivalence of infinity-localizations
\[
 \KSh[W^{-1}]\simeq(s\E)[W^{-1}]
 \simeq\Hsm.
\]
\end{thm}

\begin{proof}
Inclusion $i$ and $R$ are relative functors. The natural
maps $\id\to iR$ and $\id\to Ri$ have weak components,
so induce inverse equivalences after localization.
The last identification is the standard local-model
presentation theorem
\cite[Observation~3.77 and Proposition~4.18]{NSS2015};
compare \cite[Theorem~6.2]{DHI2004}.
\end{proof}

The operator $R$ supplies local Kan fillers; hyperdescent
fibrant replacement additionally enforces descent for
hypercovers. Locally Kan objects and fibrations
are generally broader than local-model fibrant objects
and model fibrations \cite[Remark~3.56]{NSS2015}.
Hyperdescent uses a homotopy \emph{limit}.

\begin{example}[Kan replacement does not add all descent objects]
Let $G$ be the locally constant sheaf $\Z/2$, and let
$BG$ be its degreewise nerve. Every $BG(U)$ is Kan and
has one vertex. Let $U$ be an annulus and cover it by three
contractible angular sectors with connected pairwise overlaps and empty triple
intersection. A double-cover cocycle on the three cyclic overlaps
can have transition-sign product $-1$, whereas every coboundary
has product $+1$. The Cech homotopy limit therefore has two
components, while $BG(U)$ has one.

The same obstruction remains after $R$. The levels
are locally constant sheaves, and on connected cover
members and intersections $R$ has the corresponding
$\Ex^\infty B(\Z/2)$ values. Its degree-zero sheaf is
still terminal. The missing torsor objects have not
been added.

This example uses a site containing the annulus.
A presentation on only contractible Cartesian test
objects can already be fibrant; its derived value on
an annulus is not degreewise ordinary sheaf extension.
\end{example}

\subsection{Geometric presentations and the incomplete structure}

An ordinary geometric Lie infinity-groupoid is a
simplicial smooth manifold with represented horn
matching objects and surjective-submersion horn maps.
Here manifolds have the usual finite-dimensional Hausdorff
and second-countability conventions, and include the empty
manifold. All simplicial degrees are retained.
Write $\LieGpd$ for their ordinary category with simplicial smooth
maps. Geometric fibrations satisfy the analogous represented
relative horn conditions in positive degrees.
Weak equivalences are detected by the full family of germs
$p_d$, including $d=0$. There is no condition on the
degree-zero map of an arbitrary geometric fibration.

\paragraph{Geometric covers and the open-cover topology.}
A geometric cover in this subsection is a surjective
submersion, not merely an epimorphism in $\E$.
Such covers compose, are stable under finite products,
and have ordinary manifold pullbacks along arbitrary
smooth maps. These pullbacks remain Hausdorff and
second-countable. A submersion has local smooth sections
through every prescribed source point; conversely such
sections make its derivative surjective at each point,
so the submersion theorem applies. Surjectivity together
with this prescribed-point property characterizes a
geometric cover. In particular a geometric cover induces
a sheaf epimorphism. The empty isomorphism is a cover,
but $\emptyset\to *$ is not.

We use covering-family descent in $\E$, with finite
geometric constructions formed from manifold products
and submersion pullbacks. Smooth Yoneda is fully faithful by
gluing the smooth maps supplied on atlas charts.
Moreover, if $\y q:\y W\to\y V$ is epi, lifting the
germ of a target chart gives a local smooth section of
$q$ over $V$. This section need not pass through a
chosen point of $W$. For example, the map
$\R\amalg\R\to\R$ which is the identity on one component
and constant zero on the other has a global section,
but is not a submersion on its second component.

\begin{lem}[Finite geometric horn calculus]
A finite composite of pushouts of positive horn inclusions
$S\hookrightarrow T$ induces a represented surjective
submersion $M_T\y X\to M_S\y X$ whenever $X$ is a
geometric higher groupoid and $M_S\y X$ is represented.
\label{lem:ordinary-geometric-matching}

Let $n\ge1$, let $Y$ be a geometric higher groupoid, and let
$\mathcal X\to\y Y$ be a simplicial sheaf map.
Assume that $\mathcal X_m$ is represented for $m<n$
and that its relative horn maps in dimensions
$1\le m<n$ are represented surjective submersions.
Then
\[
 H_{n,i}(\mathcal X/Y)
 =M_{\Lambda^i[n]}\mathcal X
   \times_{M_{\Lambda^i[n]}\y Y}\y Y_n
\]
is represented for every $0\le i\le n$.
The degree-$n$ source $\mathcal X_n$ may be unrepresented.
\end{lem}

\begin{proof}
For a single attachment $T=S\amalg_{\Lambda^j[m]}\Delta[m]$,
matching in $\E$ gives
\[
 M_T\y X
 =M_S\y X\times_{M_{\Lambda^j[m]}\y X}\y X_m.
\]
The new leg is a represented submersion, so this is an
available manifold pullback. Finite iteration proves the
first assertion. The empty weight gives $M_\varnothing X=1$.
Representability for other finite weights requires the
specified geometric pullbacks.

Every nonempty face inclusion $\Delta[k]\subseteq\Delta[n]$
is a finite positive-horn extension. Choose a vertex $v$
of the initial face. Process the nonempty faces $\sigma$
of the opposite face not already in the initial face in
increasing dimension, attaching $v*\sigma$ along the horn
missing $\sigma$. All other facets are already present;
the attachment adds $\sigma$ and its cone together.
Likewise $\Lambda^i[n]$ is obtained from its vertex $i$
by coning the proper faces of the opposite simplex.
Only horn dimensions below $n$ occur, and for $n=1$
the horn is already that vertex. Thus every face and
vertex map of a geometric higher groupoid is a
surjective submersion. This starts at its represented
degree-zero object, not at a cover of the point.

For the partial-source assertion, put
\[
 R_S^n=M_S\mathcal X\times_{M_S\y Y}\y Y_n
 \qquad(S\subseteq\Delta[n]).
\]
At a vertex $v$ it is represented by
$X_0\times_{Y_0}Y_n$: the available submersion is
$Y_n\to Y_0$, and no condition on $X_0\to Y_0$ is used.
If $S'$ is obtained by a dimension-$m$ horn attachment
with $m<n$, then
\[
 R_{S'}^n
 =R_S^n\times_{H_{m,j}(\mathcal X/Y)}X_m.
\]
The map to the relative horn target records the source
horn and the corresponding face of the fixed $Y_n$-simplex.
Its new leg is a known represented submersion.
The cone filtration of $\Lambda^i[n]$ proves the
assertion using only source levels below $n$.
\end{proof}

\begin{lem}[Open-local submersion descent]
\label{lem:ordinary-submersion-locality}
Let $p:U\to V$ and $q:W\to V$ be smooth maps, with
$\y q$ epi. Suppose there is an epimorphism
\[
 a:\y Z\longrightarrow\y U\times_{\y V}\y W
\]
whose composite $b:Z\to W$ is a surjective submersion.
Then $p$ is a surjective submersion, and the displayed
sheaf pullback is represented by the ordinary manifold
pullback.
\end{lem}

\begin{proof}
Full faithfulness gives the other component $h:Z\to U$,
with $ph=qb$. Epimorphy of $\y q$ gives target-local
smooth sections by lifting chart germs, and also
surjectivity on points. Since $b$ is surjective,
$ph=qb$ makes $p$ surjective.

Fix any $u\in U$ and put $v=p(u)$. Choose a local section
$t:O_v\to W$ of $q$, and let $w=t(v)$. The dimension-zero
germ makes $a$ surjective on point-valued sections, so
the exact pair $(u,w)$ lifts to $z\in Z$ with
$h(z)=u$ and $b(z)=w$. Choose a local section
$s:O_w\to Z$ of $b$ through that $z$. Shrinking $O_v$
so that $t(O_v)\subseteq O_w$, we obtain
\[
 phst=qbst=qt=\id,\qquad hst(v)=u.
\]
Thus $p$ has a section through every prescribed source
point, and is a submersion. Its ordinary pullback now
exists and represents the sheaf pullback. Representability
thus follows from the submersion property just established.
If $V$ is empty, all three sources are empty and the
conclusion is the empty isomorphism, not a cover of a
nonempty target.
\end{proof}

\begin{prop}[Geometric acyclic-fibration criterion]
\label{prop:ordinary-geometric-acyclic}
A map $f:X\to Y$ of geometric higher groupoids is a
geometric fibration and a germwise weak equivalence
if and only if every relative boundary object
\[
 B_k=M_{\partial\Delta[k]}\y X
       \times_{M_{\partial\Delta[k]}\y Y}\y Y_k
\]
is represented and every boundary projection
$q_k:\y X_k\to B_k$ is a represented surjective
submersion, for all $k\ge0$. In particular
$B_0=\y Y_0$ and $q_0=f_0$.
\end{prop}

\begin{proof}
Suppose first that $f$ is a geometric fibration and is weak.
Geometric horn submersions give sheaf epimorphisms, so
$\y X,\y Y$ are locally Kan and $\y f$ is a local
fibration. Its germs are therefore trivial Kan fibrations.
Finite matching and Proposition~\ref{prop:boundary-criterion}
make every $q_k$ an epimorphism, before representability
of $B_k$ is known.

For $k\ge1$, set $H_k=H_{k,0}(X/Y)$, represented by the
fibration hypothesis. Decomposing the boundary into its
zeroth horn and missing face gives the cartesian sheaf square
\[
\begin{tikzcd}
 B_k\arrow[r,"\alpha_k"]\arrow[d,"r_k"']
   &H_k\arrow[d,"\beta_k"]\\
 \y X_{k-1}\arrow[r,"q_{k-1}"']
   &B_{k-1}.
\end{tikzcd}
\]
Here $\alpha_k$ forgets the missing face and $r_k$ records
it. For a horn $(x_1,\ldots,x_k;y)$, with $x_j$ its
$d_j$-face, $\beta_k$ records the boundary of the missing
$d_0$-face and $d_0y$. For $k\ge2$ its source boundary is
$(d_0x_{r+1})_{0\le r\le k-1}$, by
$d_r d_0=d_0d_{r+1}$; for $k=1$ it is just $d_0y\in Y_0$.
In particular
\[
 \alpha_kq_k=\text{the relative zeroth horn map},
 \qquad r_kq_k=d_0,\qquad
 \beta_k\alpha_k=q_{k-1}r_k.
\]
The face map $d_0:X_k\to X_{k-1}$ is a submersion by
Lemma~\ref{lem:ordinary-geometric-matching}, hence epi.
Thus $r_k$ is epi. Since $q_{k-1}$ is already epi,
the last identity makes $\beta_k$ epi as well.
These epimorphism calculations take place in the sheaf topos,
before establishing representability of $B_k$.

Start with $B_0=\y Y_0$. Once $B_{k-1}$ is represented,
apply Lemma~\ref{lem:ordinary-submersion-locality} to
$q_{k-1}$ and $\beta_k$. Their sheaf pullback is $B_k$,
with epi witness $q_k$ and submersion composite
$\alpha_kq_k$. The lemma proves simultaneously that
$q_{k-1}$ is a surjective submersion and that $B_k$
is represented. This establishes every boundary condition.
The degree-$n$ conclusion uses the positive $(n+1)$st
horn; the case $k=1$ derives submersivity of $f_0$.
Thus degree-zero submersivity follows from acyclicity.

Conversely, build relative matching objects over a fixed
$Y_n$-simplex by finite boundary attachments, starting
with the empty weight and the represented object $Y_n$.
Each dimension-$m$ attachment pulls back $q_m$, including
$q_0$ when a vertex is added. This represents the relative
horn. The two further boundary attachments
$\Lambda^i[n]\hookrightarrow\partial\Delta[n]
\hookrightarrow\Delta[n]$ attach the missing
$(n-1)$-face and the $n$-simplex, making the relative
horn projection a composite of submersions.
At every germ all relative boundary maps are surjective,
so the simplicial-set boundary criterion makes $f$ a
trivial Kan fibration there, and in particular weak.
\end{proof}

\begin{thm}[Empty-compatible ordinary geometric iCFO]
\label{thm:geometric-icfo}
The ordinary geometric Lie infinity-groupoids, including
the everywhere-empty diagram, with the specified germwise
weak equivalences and positive relative-horn fibrations,
form an incomplete category of fibrant objects in the
sense of \cite[Definition~2.1, p.~1134]{RogersZhu2020}.
Its acyclic fibrations are exactly the geometric
hypercovers of Proposition~\ref{prop:ordinary-geometric-acyclic}.
Functorial outer paths are represented levelwise by
\[
 (PX)_n=M_{\Delta[n]\times\Delta[1]}\y X.
\]
\end{thm}

\begin{proof}
\emph{Products and elementary axioms.}
The terminal object is $c*$ and finite products are
levelwise manifold products. Matching commutes with
products, so their horn maps are products of submersions.
Isomorphisms are acyclic fibrations, and weak equivalences
have two-out-of-three at each prescribed germ.
Every object maps fibrationally to $c*$ by its defining
positive horn conditions. This includes the empty diagram
and does not require $X_0\to *$ to be a geometric cover.

\emph{Composition.}
For $X\xrightarrow{f}Y\xrightarrow{g}Z$, induct on
the relative horn degree. Lemma~\ref{lem:ordinary-geometric-matching}
represents $H_{n,i}(gf)$ from the lower horn conditions.
There is a cartesian sheaf square
\[
\begin{tikzcd}
 H_{n,i}(f)\arrow[r]\arrow[d]
   &Y_n\arrow[d]\\
 H_{n,i}(gf)\arrow[r]
   &H_{n,i}(g).
\end{tikzcd}
\]
The left map is therefore a base change of the relative
horn submersion for $g$. Composing it with
$X_n\to H_{n,i}(f)$ proves the next condition for $gf$.
This argument needs only the final target to be a
higher groupoid; it also applies to a simplicial manifold
not yet known to be an object, provided the two arrows
have the stipulated relative horns.

\emph{Pullbacks from degree zero.}
Let $X\xrightarrow{f}Y\leftarrow Z$ have a fibration leg.
First suppose the ordinary manifold pullback
$X_0\times_{Y_0}Z_0$ exists. Form the levelwise sheaf
pullback $\mathcal P=\y X\times_{\y Y}\y Z$.
Its degree zero is represented. Assuming the lower
levels and lower relative horns have been constructed,
the partial-source assertion of
Lemma~\ref{lem:ordinary-geometric-matching} represents
$H_{n,i}(\mathcal P/Z)$ without assuming that
$\mathcal P_n$ is represented. Sheaf pullback identities give
\[
\begin{tikzcd}
 \mathcal P_n\arrow[r]\arrow[d]
   &\y X_n\arrow[d]\\
 H_{n,i}(\mathcal P/Z)\arrow[r]
   &H_{n,i}(f),
\end{tikzcd}
\qquad
 H_{n,i}(\mathcal P/Z)
 =M_{\Lambda^i[n]}\y X
    \times_{M_{\Lambda^i[n]}\y Y}\y Z_n .
\]
The right map is a represented submersion. Its manifold
pullback represents $\mathcal P_n$ and the left map.
One horn index constructs this level; the same square for
every index proves all relative horn conditions.
Full faithfulness supplies the face and degeneracy maps
and their identities. Thus $P\to Z$ is a geometric
fibration, and the composition argument applied to
$P\to Z\to c*$ makes $P$ a higher groupoid.
Its represented sheaf universal property gives its
categorical pullback property.

\emph{Every existing categorical fibration pullback.}
Suppose instead that a pullback $Q$ is initially given
only in the geometric higher-groupoid category.
For every admitted manifold $V$, the constant diagram
$cV$ is an object: positive horns are nonempty and
connected, with identity horn maps. Moreover
\[
 \Hom(cV,Q)=\Hom(V,Q_0).
\]
Indeed the unique maps $[n]\to[0]$ in $\Delta$ force
the higher components to be total degeneracies of the
degree-zero component, and those components define
a simplicial map. Testing the categorical universal
property against every $cV$ shows that $Q_0$ represents
the ordinary degree-zero manifold pullback.
The preceding construction now produces a levelwise
represented pullback $P$. Uniqueness gives $Q\cong P$,
so its projection is a fibration and its Yoneda image
is the sheaf pullback. The constant tests therefore
establish the categorical stability axiom for $Q$.

\emph{All acyclic-fibration pullbacks.}
If $f$ is acyclic,
Proposition~\ref{prop:ordinary-geometric-acyclic} makes
$f_0$ a surjective submersion. Every degree-zero base
change therefore exists, and the preceding induction
gives every required categorical pullback.
For each prescribed germ $p$,
\[
 p^*P=p^*X\times_{p^*Y}p^*Z.
\]
Its projection is a pullback of a trivial Kan fibration,
hence is trivial again. This proves acyclicity as well
as geometric existence.

\emph{Represented outer paths.}
Write $a^\epsilon=(a,\epsilon)$. The staircase simplices
of $\Delta[n]\times\Delta[1]$ are
\[
 P_j=[0^0,\ldots,j^0,j^1,\ldots,n^1],
 \qquad 0\le j\le n.
\]
For $1\le j\le n$, the intersection with the preceding
simplices is
$D_j=[0^0,\ldots,(j-1)^0,j^1,\ldots,n^1]$.
It is the $d_j$-face of both $P_{j-1}$ and $P_j$:
one deletes $(j-1)^1$ in the first and $j^0$ in the second.
Consequently the path level is the iterated manifold
pullback described by
\[
 \left\{(a_0,\ldots,a_n)\in X_{n+1}^{\,n+1}:
 d_ja_{j-1}=d_ja_j,\quad 1\le j\le n\right\}.
\]
Each new leg is the face submersion
$d_j:X_{n+1}\to X_n$ from
Lemma~\ref{lem:ordinary-geometric-matching}. Thus these
levels exist; at $n=0$ the level is $X_1$.
Full faithfulness gives the whole simplicial manifold.

Write $e=(\operatorname{ev}_0,\operatorname{ev}_1):
PX\to X\times X$. Induct on its positive relative
horns, using Lemma~\ref{lem:ordinary-geometric-matching}
and the target higher groupoid $X\times X$.
The relative horn target identifies with $M_{U_{n,i}}\y X$,
where
\[
 U_{n,i}=(\Lambda^i[n]\times\Delta[1])
          \cup(\Delta[n]\times\partial\Delta[1]).
\]
Explicitly it is the pullback of matching on
$\Lambda^i[n]\times\Delta[1]$ and on
$\Delta[n]\times\partial\Delta[1]$, over matching on
$\Lambda^i[n]\times\partial\Delta[1]$.

The filtration of
Lemma~\ref{lem:derived-prism-filtration} uses
$n-1$ horn attachments in dimension $n$ and
$n+1$ in dimension $n+1$ and is purely simplicial.
For completeness, put $D_0$ equal to the top simplex
and $D_{n+1}$ to the bottom simplex. Let $C_s$ be
$D_s$ with its unique base-$i$ vertex removed, and
let $B_j$, for $j\ne i$, be $P_j$ with its unique
base-$i$ vertex removed. Set $d=i$ if $i>0$, and
$d=1$ otherwise. First attach the $D_s$ with
$1\le s\le n$, $s\ne d$, missing $C_s$.
Their other facets omit a base coordinate other than
$i$ and lie in $U_{n,i}$. The missing facets are new
and distinct: the sole internal coincidence is
$C_i=C_{i+1}$, while for $i=0$ the excluded $C_1$
is already top and for $i=n$ the excluded $C_n$
is already bottom.
Next attach $P_i$, missing $D_d$, and then each
$P_j$, $j\ne i$, missing $B_j$.
At each step all other facets are present and the
missing facet is new. A nondegenerate chain outside
$U_{n,i}$ uses both time labels, has base image
$[n]$ or $[n]\setminus\{i\}$, and repeats at most
one base coordinate. These are exactly the
$P_j,D_s,B_j,C_s$, proving exhaustion by actual
horn pushouts. For $n=1$ the first stage is empty
and the remaining two attachments are the triangles
of the square. For the absolute prism inclusion,
first attach the top and bottom simplices along
their $i$th horns.

The finite geometric horn calculus therefore makes
$(PX)_n\to H_{n,i}(e)$ a surjective submersion.
Thus $e$ is a fibration, and composition with the
terminal map makes $PX$ an object. Its degree-zero
component is
\[
 e^{(0)}=(d_1,d_0):X_1\longrightarrow X_0\times X_0.
\]
The map $e^{(0)}$ can be a diagonal, as on a constant
diagram; the fibration condition concerns positive
relative horns.
The notation $\operatorname{ev}_0:PX\to X$
instead denotes evaluation at the zero endpoint.

Finally, finite matching gives
$p^*(PX)\cong(p^*X)^{\Delta[1]}$.
The constant-path inclusion $j:X\to PX$ has
$\operatorname{ev}_0$ as a retraction. The monotone map
$\min:[1]\times[1]\to[1]$ supplies the simplicial
homotopy $j\operatorname{ev}_0\Rightarrow\id$.
Thus $j$ is weak at every prescribed germ.
The representing universal properties make the path
construction functorial. This proves all seven iCFO axioms.
\end{proof}

\begin{cor}[Finite unique-horn bounds]
\label{cor:ordinary-geometric-finite}
For each finite $N\ge0$, the full subcategory whose
absolute horn maps are isomorphisms of represented
manifolds in every degree $m>N$ is an iCFO with the
restricted classes. The union over finite $N$ is
also an iCFO. All simplicial degrees remain present;
the union does not include individually outer-unbounded
objects.
\end{cor}

\begin{proof}
Between two objects of bound $N$, relative horns above
$N$ are automatically isomorphisms: identify simplices
with their absolute horn objects. Products preserve
the bound. The pullback construction does also, because
its relative horns and the absolute horns of its target
are isomorphisms above $N$. For a degree-$m$ endpoint
horn the prism uses only dimensions $m,m+1$; if $m>N$,
all its restriction maps are isomorphisms. Composing
with the target horn then gives the same bound for $PX$.
Constant test manifolds belong to every finite-bound
subcategory, so the argument for an initially given
categorical pullback applies within it as well.

At $N=0$, every vertex restriction $X_n\to X_0$ is
an isomorphism by the positive-horn cone filtration.
Thus $cX_0\to X$ is an isomorphism. A weak map between
constants is a diffeomorphism, by germ conservativity
and full faithfulness of smooth Yoneda, while its
endpoint-pair path map is the permitted diagonal.
This includes the zero bound directly.
In the union, finitely many input objects have a common
maximum bound; the constructions stay within that bound.
An initially given categorical pullback in the union
is identified with this one by the same constant tests
and uniqueness.
\end{proof}

Whenever the admitted geometry contains the finite-dimensional
crossing-axes cospan, these finite-bound categories are not
full CFOs for the unchanged positive-horn fibration class.
This includes the full smooth category and the all-Banach
choice below, by the crossing-axes propositions.

\paragraph{Relation to the published construction.}
The construction adapts Rogers--Zhu to the
open-covering-family topos, including the empty cover.
Their finite horn calculus
\cite[Lemmas~3.6--3.7, Remark~3.8 and Lemma~3.9,
pp.~1149--1150]{RogersZhu2020}
and submersion locality
\cite[Propositions~6.7 and~6.12, pp.~1161--1166]{RogersZhu2020}
provide the local steps, which we combine with
covering-family epimorphisms. The construction from a
degree-zero pullback and the acyclic existence argument
parallel the distinct arguments in
\cite[Propositions~7.10 and~7.12,
pp.~1173--1176]{RogersZhu2020}.
The path construction is related to
\cite[Proposition~7.1, p.~1167; Proposition~7.2,
pp.~1168--1169; Proposition~7.13 and Lemma~7.14,
pp.~1176--1178; Appendix~A, pp.~1209--1212]{RogersZhu2020}.
Proposition~7.1 concerns stalkwise paths; it is distinct
from the iCFO Theorem~7.1 on pp.~1170--1171.
The adaptation replaces the terminal-cover hypothesis
of their Definition~3.1 by the open-covering-family
arguments above.

\begin{rem}[Singleton descent and family descent]
\label{rem:ordinary-family-topology}
The distinction is visible without a geometric limit
calculation. On the manifold category the constant
presheaf with value $S=\{0,1\}$ satisfies the literal
singleton surjective-submersion equalizer condition,
but fails the empty-cover condition. Even the
empty-normalized presheaf
\[
 F_+(M)=
 \begin{cases}
 S,&M\ne\emptyset,\\
 \{*\},&M=\emptyset,
 \end{cases}
\]
with identity restrictions between nonempty objects,
satisfies that singleton condition: over a nonempty
target the cover and its self-pullback are nonempty.
Yet descent for the two component opens of
$(-2,-1)\amalg(1,2)\subset\R$ would require the
nonbijective diagonal $S\to S^2$ to be a bijection.
The representable discrete two-point manifold instead
allows independent labels on the two components.
Thus empty normalization alone does not identify
singleton descent with the family topology used here.
\end{rem}

\begin{prop}\label{prop:geometric-yoneda}
Levelwise smooth Yoneda
\[
 \y:\LieGpd\longrightarrow\KSh
\]
is fully faithful on underlying categories and
preserves the terminal object, finite products,
specified weak equivalences, geometric fibrations,
and existing fibration pullbacks.
\end{prop}

\begin{proof}
Atlas charts prove full faithfulness, also for
simplicial diagrams. Yoneda preserves represented
matching limits, the terminal object and finite products.
Surjective
submersions have local smooth sections and hence
give sheaf epimorphisms. The weak maps are tested
at the same germs.
For an existing categorical pullback along a geometric
fibration, the constant-test argument in
Theorem~\ref{thm:geometric-icfo} first identifies its
degree-zero manifold pullback and then identifies
the entire pullback with the levelwise represented
sheaf pullback. Thus Yoneda preserves every existing
categorical fibration pullback.
\end{proof}

These preservation statements give exactness for the iCFO of
Theorem~\ref{thm:geometric-icfo} on underlying categories.
Reflection of geometric fibrations, and the mapping spaces
and essential image after localization, are further
comparison questions.
The finite-limit alternative of
\cite{BehrendGetzler2017} has its own ordinary
descent-category hypotheses.

\begin{prop}[The unchanged manifold category is not a full CFO]
\label{prop:crossing-axes}
The preceding geometric fibration class cannot make
the full category of geometric Lie infinity-groupoids
a Brown CFO with all fibration pullbacks.
\end{prop}

\begin{proof}
Every map between constant simplicial manifolds is a
geometric horn fibration: a positive-dimensional horn
is nonempty and connected, so the relative horn maps
are identities. A full CFO would have to contain the
pullback of
\[
 c\R^2\xrightarrow{xy}c\R\xleftarrow{0}c*.
\]
If a simplicial manifold $Q$ represented this pullback,
maps $cV\to Q$ would be exactly maps $V\to Q_0$.
Thus $Q_0$ would represent the crossing-axes functor.

The two axis inclusions would lift to smooth curves
through the same point of $Q_0$. The derivative of
$Q_0\to\R^2$ would surject onto the two coordinate
directions there. The submersion theorem would force
its image to contain a neighborhood of the origin,
contradicting $xy=0$.
\end{proof}

The full sheaf-theoretic repair changes the object
category. Replacing ``ordinary limit'' by ``homotopy
limit'' inside an ordinary Brown axiom does not repair
the unchanged geometric category. Also, an internal
Kan presentation is not automatically a Segal groupoid
object in an infinity-category; those are distinct
presentation conventions.

\subsection{The ordinary Banach-open site}
\label{subsec:ordinary-banach}

The preceding ordinary-site construction also has
a Banach version, separate from the derived
geometry of Section~\ref{sec:banach}.
Choose a small Banach-open site, closed under the
open sets and finite products in use, with smooth
maps and open-covering families. Include the empty
object and its empty cover. For a pointed plot
$(U,u)$ define
\[
 p_{U,u}^*F=
 \operatorname*{colim}_{u\in V\subseteq U}F(V),
\]
over open neighborhoods, with restriction maps
as the neighborhoods shrink.

\begin{prop}\label{prop:ordinary-banach-site}
These germ functors preserve finite limits and
colimits and are jointly conservative. Locally
Kan simplicial sheaves on this site form a full
Brown CFO for germwise weak equivalences and
local positive-horn fibrations. Its localization
is the hypercomplete Banach-open infinity-topos.
\end{prop}

\begin{proof}
Finite sheaf limits are sectionwise, and filtered
colimits of sets are exact, proving left exactness.
Open-cover sheafification does not change a germ:
a local gluing or equality at $u$ is represented
on a smaller neighborhood. Therefore the germs
of a sheaf colimit are the colimits of the germs,
which proves the colimit assertion.

The family detects both equality and local
existence of sections. Germwise preimages give
preimages on an open cover; germwise uniqueness
makes them agree after restriction, so they glue.
This proves conservativity and detection of
epimorphisms.

Finite simplicial matching objects commute with
the germ functors. Local Kan objects and local
fibrations consequently have Kan and
Kan-fibration germs. The proof of
Theorem~\ref{thm:sheaf-cfo} applies in this
ordinary sheaf topos, including existence of
all fibration and acyclic-fibration pullbacks
and the path object $X^{\Delta[1]}$.

The internal $\Ex^\infty$ construction uses
finite matching and exact filtered sheaf
colimits as before. It gives germwise
equivalent locally Kan models. The local
stalkwise/hypercover presentation theorem
\cite{DHI2004,NSS2015} then identifies the
localization with the hypercomplete target.
\end{proof}

This sheaf-theoretic construction applies to any chosen
plot site with the stated properties. Its geometric
counterpart additionally requires represented matching
objects and compatibility with manifold charts.

\begin{rem}[Fixed-site chart compatibility]
\label{rem:ordinary-banach-chart-scope}
For the ordinary geometric theorem, require the following
compatibility of the already chosen plot site and the
geometric category. The point and empty manifold are admitted;
the plot domains are geometric objects; every admitted
geometric manifold has an atlas by domains in that site,
with all smooth maps between its domains and the open
restrictions used by the atlases available. The geometric
category is closed under finite products and ordinary
pullbacks along split submersions. Those pullbacks again
have admitted atlases, including their local
complemented-kernel times source-chart models.

One concrete choice fixes universes $\mathcal U\in\mathcal V$,
takes all $\mathcal U$-small Hausdorff spaces with compatible
local charts in $\mathcal U$-small real Banach spaces as the
geometric manifolds, and a $\mathcal V$-small full presentation
of all $\mathcal U$-small Banach domains as the plot site.
Model spaces may vary between components. No global
second-countability, separability or paracompactness
condition is imposed in this realization. A restricted
model-space class may be used if it contains zero and is
closed, up to bounded linear isomorphism, under finite
products and complemented closed subspaces, with the
site supplying the corresponding kernel-product chart
domains and every geometric object's atlas.
For any additional global manifold convention, closure
under the particular split-submersion pullbacks remains
an explicit requirement.

The site and geometric category remain fixed throughout
the construction: enlarging the plot site changes the
germwise weak-equivalence tests. Chart compatibility is not implied
by smallness and opens/products closure alone, and is
needed only for the geometric theorem, not for the
arbitrary-site sheaf CFO of
Proposition~\ref{prop:ordinary-banach-site}.
\end{rem}

\begin{thm}[Ordinary split-Banach geometric iCFO]
\label{thm:ordinary-banach-icfo}
For the chart-compatible ordinary Banach geometry of
Remark~\ref{rem:ordinary-banach-chart-scope}, simplicial
Banach manifolds with represented surjective split-submersion
positive horns form an iCFO. Fibrations are specified by the
corresponding positive relative horns.
Weak equivalences are detected by all pointed plot germs
of Proposition~\ref{prop:ordinary-banach-site} on that fixed
site. Acyclic fibrations are exactly the represented
split-submersion hypercovers in every degree, including zero.
The outer path formula of Theorem~\ref{thm:geometric-icfo},
every finite unique-horn bound $N\ge0$, and their finite-bound
union have the same conclusions.
\end{thm}

\begin{proof}
We verify the geometric inputs of the smooth construction.
A natural transformation of restricted functors of points
gives smooth maps on atlas domains. Naturality makes them
agree on overlaps; they glue, and factorization of any plot
locally through atlas charts identifies the transformation
with the resulting smooth map. Thus restricted Yoneda is
fully faithful. Lifting a target chart germ through an
epimorphism of represented sheaves gives local smooth
target sections. Both assertions use the stipulated atlas
compatibility, not just conservativity of the germ family.

A smooth local section of $f$ through $x$ differentiates
to a bounded right inverse of $Df_x$. Conversely, in
Banach charts let $L=Df_x:E\to F$ and $LR=1$ with $R$
bounded. Put $K=\ker L$ and $\Pi=1-RL:E\to K$. Then
\[
 (\Pi,L):E\longrightarrow K\oplus F,\qquad
 (k,v)\longmapsto k+Rv
\]
are bounded inverse linear maps. The chart map
$u\mapsto(\Pi(u-x),f(u)-f(x))$ has this invertible
derivative. The Banach inverse-function theorem gives
local projection coordinates, and fixing the kernel
coordinate gives a section through $x$.
This is the split convention of
\cite[Definition~1.4.1 and Lemma~1.4.3,
pp.~28--29]{Schmeding2021}, agreeing with the
prescribed-point convention of
\cite[p.~1146]{RogersZhu2020}.

For a map $g:Z\to N$ and a split submersion to $N$,
these projection coordinates describe the pullback locally as
\[
 \{(k,z):(k,g(z))\text{ lies in the projection chart}\}.
\]
It is an open subset of the kernel times a source chart.
These charts prove the ordinary categorical universal
property and the split-submersion property of the
projection to $Z$. The pullback has the subspace topology
in a product of Hausdorff spaces. Its models remain
admitted by the explicit closure hypothesis.
Surjective split submersions compose, have stable finite
products, and induce epimorphisms of plot sheaves.

The prescribed-point proof of
Lemma~\ref{lem:ordinary-submersion-locality} now applies:
chart germs give $t$, the point germ lifts the exact pair
to $z$, and the split submersion supplies $s$ through $z$.
The section $hst$ proves split submersivity without
assuming the sheaf fibre product representable.
Proposition~\ref{prop:ordinary-banach-site} supplies the
finite-limit, colimit and conservative germ properties.
Consequently the finite horn and partial-source inductions,
the geometric boundary upgrade, the categorical
constant-test pullback argument, and the represented
staircase paths of Theorem~\ref{thm:geometric-icfo} use
exactly these geometric properties.
They prove all seven axioms, including every acyclic
pullback. The isomorphism versions of the same horn
and prism constructions prove
Corollary~\ref{cor:ordinary-geometric-finite} here as well.

In the all-$\mathcal U$-small example, each pullback is
a subset of a finite product of $\mathcal U$-sets.
Its topology, atlas and complemented-kernel models remain
$\mathcal U$-small, as does the countable simplicial
diagram. Only finitely many geometric operations occur
in each degree; atlas covers may be arbitrary small families.
\end{proof}

The local-model convention matters. On $H=\ell^2$, the map
$H\amalg H\to H$ equal to the identity on one component
and the left shift on the other is a split submersion.
Its zero fibre is $*\amalg\R$. This is a valid locally
Banach manifold with componentwise models, but not one
globally modeled on a single Banach space. Thus a
globally uniform-model convention does not automatically
have the pullback closure used in the theorem.

\begin{prop}[An insufficient ordinary plot site]
\label{prop:ordinary-insufficient-site}
Smallness and closure of the ordinary plot site under
opens and finite products do not suffice for the
geometric iCFO assertion on all Banach manifolds.
The point-only site gives a failure of mandatory
acyclic-fibration pullback existence, already for
constant diagrams of unique-horn bound zero.
\end{prop}

\begin{proof}
Take $\mathcal C_0=\{\emptyset,*\}$ with its open-family
topology. Its sheaf topos is $\Set$, and its only
nonempty pointed germ is evaluation at the point.
Keep all ordinary Banach manifolds as geometric objects.
Every map between constant simplicial manifolds has
canonical identity positive relative horns. Hence
\[
 f:c\R\longrightarrow c\R,\quad t\longmapsto t^3
\]
is a geometric fibration and is weak for these tests,
since the cubic is bijective on points.
Base-change it along
\[
 g:c\R^2\longrightarrow c\R,\quad (x,y)\longmapsto xy.
\]
If a categorical geometric pullback $Q$ existed,
testing against every constant geometric object $cV$,
not merely the site's weak-equivalence tests, would
make $Q_0$ represent the ordinary smooth pullback
functor $t^3=xy$. Write its cone maps as
$a:Q_0\to\R$ and $r:Q_0\to\R^2$, so $a^3=r_1r_2$.
Point tests give a unique point $q_0$ above $(0,0,0)$.
The compatible maps $s\mapsto(0,s,0)$ and
$s\mapsto(0,0,s)$ lift by the categorical universal
property to smooth curves $\gamma_1,\gamma_2$ through
that same $q_0$. Their derivatives satisfy
\[
 Dr_{q_0}\gamma_1'(0)=(1,0),\qquad
 Dr_{q_0}\gamma_2'(0)=(0,1).
\]
Thus $Dr_{q_0}$ is surjective. It has a bounded right
inverse even for Banach $Q_0$, by sending a basis
of the finite-dimensional target to these two vectors.
The submersion theorem supplies a local smooth section
$\sigma:O\subseteq\R^2\to Q_0$ through $q_0$.
Then $h=a\sigma$ is smooth, $h(0)=0$, and $h^3=xy$.
But $D^2(h^3)_0=0$, whereas the mixed second
derivative of $xy$ at zero is $1$, a contradiction.
The contradiction uses only the categorical universal
property of the hypothetical pullback.

On the full Euclidean site the cubic is not weak:
the identity germ at zero has no smooth cube-root
lift. Differentiating $u(s)^3=s$ at zero would give
$0=1$. Thus this failure for an insufficient site
does not contradict the smooth iCFO theorem.
\end{proof}

The following crossing-axes obstruction concerns the full
Banach choice. For a smaller chart class the same
non-fullness conclusion requires admission of that
finite-dimensional cospan; a purely discrete geometry
instead gives the ordinary full Kan CFO.

\begin{prop}\label{prop:ordinary-banach-obstruction}
The unchanged positive geometric horn-fibration
class does not give a full CFO on all such
ordinary geometric Banach higher groupoids.
\end{prop}

\begin{proof}
Use the constant crossing-axes cospan of
Proposition~\ref{prop:crossing-axes}.
If a degree-zero Banach manifold represented
its pullback, the two lifted axes would make
the derivative of its map to $\R^2$ surjective
at their common point.
A bounded surjection to a finite-dimensional
space has a bounded linear right inverse:
choose lifts of a basis and extend linearly.
The Banach submersion theorem therefore
forces the image to contain an open
neighborhood of the origin, contradicting
its containment in $xy=0$.
\end{proof}

%% file: derived-geometry.tex
\section{Derived groupoids and ordinary represented models}
\label{sec:derived}

\subsection{Complete derived smooth affines}

We use Nuiten's connective
homological dg $C^\infty$-rings, equivalently function algebras in
nonpositive cohomological degrees
\cite[Definition~2.2.1 and Proposition~2.2.4, p.~22]{NuitenThesis2018}.
Weak maps are quasi-isomorphisms; algebra-model fibrations surject in
strictly positive homological degrees. In cohomological notation,
a zero-locus chart is
\[
 A=C^\infty(\R^r)[\epsilon_1,\ldots,\epsilon_s],
 \qquad |\epsilon_i|=-1,\qquad d\epsilon_i=f_i.
\]
Reversing the labels on a positive CE algebra does not produce this
category: it would also change the degree of its differential.

Let $\CA$ be the infinity-localization of these rings. An algebra
is complete if $A\to\Gamma\Spec A$ is an equivalence.
Nuiten identifies affine derived manifolds with
$(\CA^{\mathrm{complete}})^\op$
\cite[Corollary~5.1.28, p.~130]{NuitenThesis2018};
the spectrum on all raw algebras is not fully faithful.
His broad derived manifolds carry no a priori finiteness conditions
\cite[p.~121]{NuitenThesis2018}. In particular we impose neither
quasi-smoothness nor finite presentation on all objects of this section.

Choose a compatible size cutoff $\Aff_\kappa$, closed under the
basic opens and finite constructions in use, with the basic-open
covering-family topology. Smooth localization satisfies
\[
 \pi_j(A\{a^{-1}\})
 \cong \pi_0(A)\{a^{-1}\}\otimes_{\pi_0(A)}\pi_j(A).
\]
Open immersions and etale maps retain the entire pulled-back derived
structure sheaf. Smooth maps are locally projections
$\R^r\times V\to V$, and smooth surjections admit local sections
\cite[Definition~5.2.8 and Lemma~5.2.10,
pp.~135--137]{NuitenThesis2018}.
Neither topological surjectivity nor an algebra-model fibration
is a substitute for geometric smoothness.

\subsection{Ambient models and finite geometric presentations}

Write
\[
 \mathcal X=\Sh_\tau(\Aff_\kappa),
 \qquad \widehat{\mathcal X}=\mathcal X^\wedge.
\]
Here $\mathcal X$ means ordinary space-valued infinity-sheaves,
defined by covering-family descent, as in
\cite[Definition~5.1.1 and Section~5.2.1]{NuitenThesis2018}.
The functor of points
\[
 h:\dMfd\longrightarrow\mathcal X,\qquad
 h(M)(U)=\Map_{\dMfd}(U,M),
\]
is fully faithful and preserves derived pullbacks
\cite[Proposition~5.2.3, proof on p.~135]{NuitenThesis2018}.

\begin{thm}[Ambient derived-stack models]\label{thm:derived-ambient}
There are ordinary simplicial presheaf model presentations
$\CM_{\mathrm{sh}}$ and $\CM_{\mathrm{hyp}}$ of
$\mathcal X$ and $\widehat{\mathcal X}$. The fibrant objects of
either, with inherited model weak equivalences and fibrations,
form a full Brown CFO.
\end{thm}

\begin{proof}
Choose a small Kan-enriched category $\CT$ presenting $\Aff_\kappa$.
The global injective model on
$\Fun_{\sSet}(\CT^\op,\sSet)$ has objectwise monomorphism
cofibrations \cite[Proposition~3.6.1]{HAGI2005}
and presents
$\mathcal P=\Fun(\CT_\infty^\op,\mathcal S)$.
Here $\CT_\infty$ is the infinity-category presented by $\CT$.
Realize ordinary sheafification
$a_{\mathrm{sh}}:\mathcal P\to\mathcal X$
as an accessible left Bousfield localization of this model.
Hypercompletion is a further accessible left-exact localization
\cite[pp.~668--669]{LurieHTT2009}.
Choose $\CM_{\mathrm{hyp}}$ as the corresponding further
Bousfield localization of $\CM_{\mathrm{sh}}$. Their cofibrations
are the same, and the identity
$\CM_{\mathrm{sh}}\to\CM_{\mathrm{hyp}}$ is left Quillen
with derived functor $L^\wedge$.
Apply Brown's construction to fibrant objects.
\end{proof}

\begin{lem}[Hypercomplete-model dictionary]
\label{lem:derived-model-dictionary}
The model $\CM_{\mathrm{hyp}}$ can be taken to be the local
injective $\tau$-hypersheaf model of HAG~I on $\CT$.
Its underlying infinity-category is $\widehat{\mathcal X}$
\emph{under} $\mathcal P$, with localization functor
$a_{\mathrm{hyp}}=L^\wedge a_{\mathrm{sh}}$.
For every strict outer simplicial enriched presheaf $Z_\bullet$,
naturally in the entire diagram,
\begin{equation}\label{eq:derived-diagonal-realization}
\begin{split}
 a_{\mathrm{hyp}}(\operatorname{diag}Z_\bullet)
 &\simeq
 \mathop{\operatorname{colim}}\nolimits_{\Delta^\op}^{\widehat{\mathcal X}}
 L^\wedge a_{\mathrm{sh}}(Z_k)\\
 &\simeq
 L^\wedge
 \mathop{\operatorname{colim}}\nolimits_{\Delta^\op}^{\mathcal X}
 a_{\mathrm{sh}}(Z_k).
\end{split}
\end{equation}
\end{lem}

\begin{proof}
The basic-open topology is a Grothendieck topology on
$\operatorname{Ho}(\CT)$, hence an $S$-topology by
\cite[Definition~3.1.1]{HAGI2005}.
Basic opens are invariant under equivalence and stable under
derived base change. This uses the enriched site, not a model
category of affines or a discrete replacement of its mapping spaces.

The localization $L^\wedge a_{\mathrm{sh}}$ is left exact and
$t$-complete: the latter term is hypercompleteness in
\cite[Remark~6.5.2.11 and Lemma~6.5.2.12, p.~669]{LurieHTT2009}.
It induces exactly $\tau$ on $0$-truncated objects.
Indeed these objects of $\mathcal P$ are presheaves of sets on
$\operatorname{Ho}(\CT)$; the $0$-truncated objects of $\mathcal X$ are its
$\tau$-sheaves of sets, and hypercompletion leaves them unchanged
\cite[Lemma~6.5.2.9, p.~669]{LurieHTT2009}.
The uniqueness in \cite[Theorem~3.8.3]{HAGI2005},
whose inverse construction recovers the topology on precisely
these objects, therefore identifies this Bousfield localization
with the $\tau$-hypercover localization.
Proposition~3.6.1 of that source gives the injective representative
with the same local equivalences. The identification is under
$\mathcal P$: it identifies the local objects and weak maps,
not merely abstract homotopy categories.

For the realization formula, work first in the \emph{global}
injective presheaf model. Every outer simplicial presheaf is
Reedy cofibrant there: its latching maps are objectwise the
inclusions of degenerate outer simplices. Its realization
with the standard cosimplicial simplex therefore computes
the outer homotopy colimit in $\mathcal P$, and the co-Yoneda
identity gives
\[
 \int^{[k]\in\Delta} Z_k\times\Delta[k]
 \cong \operatorname{diag}Z_\bullet.
\]
Apply the colimit-preserving functor
$a_{\mathrm{hyp}}=L^\wedge a_{\mathrm{sh}}$.
This proves the first equivalence in
\eqref{eq:derived-diagonal-realization}; preservation of
colimits by $L^\wedge$ proves the second. All identifications
are natural, so they also identify the pullback comparison maps
used below.
\end{proof}

These ambient models retain the enriched mapping spaces and
apply without an ordinary conservative point family.
Representability of their levels is an additional condition.

\begin{defn}[Nuiten]\label{def:derived-geometric-groupoids}
For finite $n$, a derived Lie $n$-groupoid is a coherent simplicial
diagram in $\dMfd$ with smooth-surjective homotopy horn maps,
which are equivalences in degrees $k>n$.
A Kan fibration is defined by smooth-surjective relative homotopy
horn maps in positive degrees.
It is a \emph{smooth} Kan fibration if its degree-zero map is
additionally smooth, not necessarily surjective.
A smooth hypercover has smooth-surjective relative homotopy
boundary maps in every degree, including zero.
\end{defn}

\begin{thm}[Nuiten's finite-$n$ presentation]\label{thm:nuiten-derived}
Realization into $\mathcal X$ exhibits the infinity-localization
\[
 \dLie_n[H^{-1}]\simeq\dSt_n\subseteq\mathcal X,
\]
where $H$ denotes smooth hypercovers and $\dSt_n$ is the full
subcategory of derived geometric $n$-stacks.
Levelwise pullbacks along general Kan fibrations remain derived
Lie $n$-groupoids and realize to pullbacks in $\mathcal X$.
Every map $|X|\to|Y|$ is represented by a homotopy-commuting triangle
$X\leftarrow U\to Y$, with $U\to X$ a smooth hypercover and
$U\to Y$ Kan, retaining the chosen target $Y$.
\end{thm}

Nuiten proves these assertions in
\cite[Definition~5.2.14, Proposition~5.2.17,
and Remark~5.2.18, pp.~138--139]{NuitenThesis2018}.
The pullback assertion uses Kan fibrations, whose geometric
conditions concern positive relative horns. The universal
infinity-localization identifies mapping spaces as well as
objects. Nuiten's proof specializes Pridham's hypergroupoid
methods. A finite outer bound imposes no bound on the homotopy groups
of values on derived tests.

\begin{lem}[The required finite derived limits]
\label{lem:derived-finite-limits}
Broad $\dMfd$ has finite homotopy limits, preserved by $h$.
\end{lem}

\begin{proof}
Every coherent constant diagram $cM$ is a derived Lie
$n$-groupoid. Every map $cM\to cN$ is Kan: a nonempty horn is
contractible, so both its absolute and relative homotopy horn
maps are equivalences. Apply Theorem~\ref{thm:nuiten-derived}
to $cM\to cN\leftarrow cP$. Its levelwise represented pullback
supplies $M\times_N^hP$, and its realization assertion gives
\[
 h(M\times_N^hP)\simeq hM\times_{hN}hP.
\]
Alternatively, the categorical pullback already suffices:
$[0]$ is initial in $\Delta^\op$, giving
\[
 \Map(cT,Q_\bullet)\simeq\Map(T,Q_0).
\]
Testing against every $cT$ identifies $Q_0$ with the desired pullback.
The terminal object is the derived point. Terminal objects and
pullbacks give finite limits.
\end{proof}

This deduction uses a finite constant presentation of each
individual cospan. It therefore also supplies the levelwise
limits needed later for outer-unbounded diagrams.

\begin{thm}[Geometric-stack-model CFO]
\label{thm:derived-geometric-cfo}
Let $\mathcal B_n$ be the full ordinary subcategory of
$(\CM_{\mathrm{sh}})_{\mathrm{fib}}$ whose associated
infinity-sheaves lie in $\dSt_n$. With inherited model weak
equivalences and fibrations it is a full Brown CFO, and
\[
 \mathcal B_n[W_{\mathrm{model}}^{-1}]\simeq\dSt_n.
\]
\end{thm}

\begin{proof}
Theorem~\ref{thm:nuiten-derived} gives finite homotopy-limit
closure of $\dSt_n$: represent a cospan, use common-target
source refinements to represent its arrows with a Kan leg,
and apply the realization pullback theorem.
The terminal stack is represented by the point.

An ordinary model pullback along a fibration between fibrant
objects computes this homotopy pullback and is fibrant.
Acyclic pullbacks are acyclic by the ambient model axioms.
A model path factorization has fibrant middle object weakly
equivalent to its geometric source, hence still in $\mathcal B_n$.
The remaining Brown axioms are inherited.

A full weak-equivalence-replete subcategory of fibrant models
presents the corresponding full infinity-subcategory.
Indeed an acyclic-fibration cofibrant replacement $QX\to X$
is still fibrant and weakly in that subcategory; contractible
simplex framings stay weakly there as well. Thus all derived
mapping spaces are the ambient ones, not merely their
degree-zero components. This proves the displayed localization.
\end{proof}

\subsection{Ordinary models with represented levels}

The represented-model viewpoint already occurs in
\cite[Definitions~3.1--3.2, Remark~3.3, Lemma~3.4 and
Remark~3.28, pp.~16--21]{Pridham2013}, including the CFO formed by taking the
union over finite bounds.
There the weak maps are specified by $Q$-mapping-space tests;
their realization comparison is
\cite[Theorem~4.10, p.~26]{Pridham2013}.
We give an ordinary presheaf-model version in Nuiten's
unhypercompleted smooth setting $\mathcal X$, with
covering-family descent and broad derived manifolds.
Strong quasi-compactness belongs to the affine hypergroupoid
setting; here the geometric requirement is finite-limit closure.

For a finite simplicial set $K$ and a strict simplicial object
$X$ of $\CM_{\mathrm{sh}}$, write $\{K,X\}$ for its ordinary
outer weighted limit.

\begin{lem}\label{lem:derived-weighted-matching}
If $X$ is Reedy fibrant, $\{K,X\}$ is fibrant and computes
the homotopy weighted matching object. If every $X_k$ is
weakly represented by a derived manifold, so is $\{K,X\}$.
\end{lem}

\begin{proof}
Build $K$ by finitely many boundary-cell attachments.
Weighted limits turn each attachment into a pullback along
a Reedy matching fibration of $X$. All objects involved are
fibrant, so these ordinary pullbacks compute homotopy pullbacks.
Induct from the empty weight. Representability follows from
Lemma~\ref{lem:derived-finite-limits}.
\end{proof}

\begin{rem}\label{rem:generic-weighted-matching}
The cell-induction proof works unchanged for a fully faithful,
finite-limit-preserving functor $h:\mathcal D\to\mathcal X$
from any geometric infinity-category with finite limits.
Replace ``derived manifold'' by ``object of $\mathcal D$''.
This is the form used in the Banach and associative applications;
the existence of the requisite geometric limits and functor of
points is proved separately in each setting.
\end{rem}

\begin{defn}\label{def:derived-reedy-models}
Let $\mathscr R_\infty$ be the full ordinary category of
Reedy-fibrant strict simplicial objects $X$ in
$\CM_{\mathrm{sh}}$ such that each $X_k$ is model weakly
equivalent to $h(M_k)$, and the associated coherent homotopy
horn maps are geometrically smooth surjections.
Let $\mathscr R_n$ impose additionally that these horn maps
are equivalences above finite $n$.

Morphisms are strict natural transformations. A fibration
is a Reedy model fibration whose associated relative
homotopy horn maps are smooth surjections in every positive
degree. The Reedy model-fibration condition applies in every
degree, including zero. Write $W_{\mathrm{sh}}$ for equivalences
under ordinary realization $|X|$ in $\mathcal X$.
\end{defn}

Geometric properties are tested through the fully faithful
image of $h$, before hypercompletion. They are independent
of the choices of weakly representing manifolds.
These presheaf diagrams are realized directly in $\mathcal X$.

\begin{lem}[Strictification]\label{lem:derived-strictification}
Localizing $\mathscr R_n$ at levelwise model equivalences
recovers the coherent infinity-category $\dLie_n$.
The analogous assertion without a finite bound recovers
the full coherent category of outer-unbounded geometric
Kan diagrams.
\end{lem}

\begin{proof}
Strictification for suitably small diagrams in a combinatorial
simplicial model category is
\cite[Proposition~4.2.4.4, p.~258]{LurieHTT2009}.
Reedy fibrant replacement is levelwise a model equivalence;
it retains weak representability and the geometric properties
of homotopy horn maps. Thus the indicated coherent diagrams
have ordinary models here.

The subcategory is full and replete among Reedy-fibrant
objects for levelwise equivalences. Cofibrant replacements
and simplex framings compute the same derived mapping
spaces while remaining weakly in that subcategory, as in
the proof of Theorem~\ref{thm:derived-geometric-cfo}.
Finally the fully faithful $h$ induces a fully faithful
functor on coherent diagram categories.
\end{proof}

\begin{lem}[Fibration calculus]\label{lem:derived-reedy-calculus}
The categories $\mathscr R_\infty$ and $\mathscr R_n$
have terminal objects and finite products. Their fibrations
compose, and ordinary pullbacks along them exist in the same
category and are fibrations.
\end{lem}

\begin{proof}
The terminal diagram is represented by the point.
Products preserve Reedy fibrancy and geometric horn maps.
For $p:X\to Y$, set
\[
 R_\Lambda(p)=Y_k\times_{\{\Lambda,Y\}}\{\Lambda,X\},
 \qquad \Lambda=\Lambda^i[k].
\]
For $q:Y\to Z$, write $\lambda_p:X_k\to R_\Lambda(p)$ and
$\lambda_q:Y_k\to R_\Lambda(q)$ for the relative horn maps.
The relative horn map of $qp$ factors as $\lambda_p$ followed
by the base change of $\lambda_q$, using the identity
\[
 R_\Lambda(qp)
 =\{\Lambda,X\}\times_{\{\Lambda,Y\}}R_\Lambda(q).
\]
Smooth surjections and Reedy fibrations are stable under
composition and base change.

For a fibration $p:X\to Y$ and any $Z\to Y$, the ordinary
Reedy pullback $X\times_Y Z$ is Reedy fibrant and computes
the levelwise homotopy pullback. Its levels are weakly
represented by Lemma~\ref{lem:derived-finite-limits}.
Weighted matching commutes with these limits; hence its
relative horn maps over $Z$ are the base changes of those
of $p$. Composition with the absolute horn maps of $Z$
proves the absolute Kan condition. Above finite $n$, the
absolute horn maps of $X,Y,Z$ are equivalences, so the
relative and pullback horn maps are equivalences too.
\end{proof}

\subsection{Outer paths with explicit finite horn control}

\begin{lem}[Prism filtration]\label{lem:derived-prism-filtration}
For $m\ge1$ and $0\le i\le m$, the inclusion
\[
 U=(\Lambda^i[m]\times\Delta[1])
   \cup(\Delta[m]\times\partial\Delta[1])
   \longrightarrow\Delta[m]\times\Delta[1]
\]
is a composite of $m-1$ horn attachments of dimension $m$
and $m+1$ of dimension $m+1$.
The absolute inclusion
$\Lambda^i[m]\times\Delta[1]\hookrightarrow\Delta[m]\times\Delta[1]$
has a filtration with two further horn attachments of dimension $m$.
\end{lem}

\begin{proof}
Write $a^e=(a,e)$. The top prism simplices and diagonal
facets are
\[
 P_j=[0^0,\ldots,j^0,j^1,\ldots,m^1],\qquad
 D_s=[0^0,\ldots,(s-1)^0,s^1,\ldots,m^1].
\]
Here $0\le j\le m$, $1\le s\le m$; put
$D_0=\Delta[m]\times\{1\}$ and
$D_{m+1}=\Delta[m]\times\{0\}$.
For $1\le s\le m$, let $C_s$ be the facet obtained by deleting
the unique vertex of $D_s$ with base coordinate $i$.
For $0\le j\le m$ with $j\ne i$, let $B_j$ be the facet obtained
by deleting the unique vertex of $P_j$ with base coordinate $i$.
Set $d=i$ if $i>0$, and $d=1$ if $i=0$.

First attach $D_s$ for $1\le s\le m$ and $s\ne d$, along the
horn consisting of all its facets except $C_s$.
All other facets omit a base coordinate different from
$i$ and lie in $U$. These $m-1$ missing facets are new
and distinct: for internal $i$ the sole coincidence is
$C_i=C_{i+1}$; for $i=0$ the excluded $C_1$ is already
top, and for $i=m$ the excluded $C_m$ is already bottom.
These are horns of dimension $m$, with index $i$.

Next attach $P_i$, missing $D_d$. The missing vertex
has index $i$ for $i>0$, or index $1$ for $i=0$.
Its other diagonal facet is present; the remaining
facets lie in $U$.
Finally attach every $P_j$, $j\ne i$, missing $B_j$.
The horn index is $i$ for $i<j$ and $i+1$ for $i>j$.
Both diagonal facets are now present, and the other
facets lie in $U$. Each missing $B_j$ is new.

A nondegenerate chain outside $U$ uses both time labels,
has base image $[m]$ or $[m]\setminus\{i\}$, and repeats
at most one base vertex. The four possibilities are
exactly $P_j,D_s,B_j,C_s$. This proves exhaustion by the
stated horn pushouts and gives their counts.
For the absolute inclusion, first attach the bottom
and top simplices along their $i$th horns, obtaining $U$.
\end{proof}

\begin{prop}[Represented outer paths]\label{prop:derived-outer-path}
For $X\in\mathscr R_\infty$, set
\[
 (PX)_k=\{\Delta[k]\times\Delta[1],X\}.
\]
Then $PX$ is an object of $\mathscr R_\infty$,
the endpoint map $PX\to X\times X$ is a fibration,
and the constant-path map $X\to PX$ lies in
$W_{\mathrm{sh}}$. If $X\in\mathscr R_n$, then
$PX\in\mathscr R_n$.
\end{prop}

\begin{proof}
Lemma~\ref{lem:derived-weighted-matching} supplies
weakly represented levels.
The Reedy endpoint matching map is induced by
\[
 (\partial\Delta[k]\times\Delta[1])
 \cup(\Delta[k]\times\partial\Delta[1])
 \longrightarrow\Delta[k]\times\Delta[1].
\]
A finite boundary-cell filtration expresses it as
a composite of pullbacks of Reedy matching fibrations.
The same argument without the endpoint faces proves
Reedy fibrancy of $PX$.

Lemma~\ref{lem:derived-prism-filtration} expresses the
geometric relative and absolute horn maps as finite
composites of base changes of horn maps of $X$.
Thus they are smooth surjections. In outer degree
$k>n$, all attachment dimensions are $>n$, proving
the finite bound.

Let $j:X\to PX$ be the constant-path map and let $e_0:PX\to X$
be evaluation at zero. Then $e_0j=\id$. The monotone
map $\min:[1]\times[1]\to[1]$ supplies an outer
simplicial homotopy $je_0\Rightarrow\id$.
For simplicial presheaves, the objectwise bisimplicial
diagonal computes realization before sheafification;
this homotopy gives an ordinary simplicial homotopy
on those diagonals. Hence $|j|$ is an equivalence.
It need not be a levelwise model equivalence.
\end{proof}

This is the outer path construction also used in
\cite[Definition~1.20 and Remark~3.28]{Pridham2013}.
It is not a levelwise internal model cotensor.
For example, the internal path diagonal of the pair
groupoid of $\R$ has relative source-horn map
\[
 \R^2\longrightarrow\R^3,\qquad(s,t)\longmapsto(s,t,t),
\]
which is neither surjective nor a submersion.
Moreover, even a coherent constant line gives the
degree-zero endpoint map $\R\to\R^2$.
Thus geometric degree-zero smoothness cannot be added
to our fibration class while retaining these paths.

\subsection{Finite ordinary realization and unbounded hypercomplete realization}

\begin{thm}[Finite represented CFOs]\label{thm:derived-reedy-finite}
For every finite $n$, $(\mathscr R_n,W_{\mathrm{sh}},F)$
is a full Brown CFO and
\[
 \mathscr R_n[W_{\mathrm{sh}}^{-1}]\simeq\dSt_n.
\]
The union $\mathscr R_{\mathrm{fin}}=\bigcup_{n<\infty}
\mathscr R_n$ is also a full Brown CFO and presents
the full union $\bigcup_{n<\infty}\dSt_n$.
\end{thm}

\begin{proof}
Lemmas~\ref{lem:derived-reedy-calculus} and
\ref{lem:derived-prism-filtration} and
Proposition~\ref{prop:derived-outer-path}
give all structural Brown axioms and path objects.
Two-out-of-three holds for realization equivalences.
For an acyclic fibration, the ordinary model pullback
realizes to the pullback of its realization by
Theorem~\ref{thm:nuiten-derived}; hence it remains
a weak equivalence. Ambient right properness alone
would not prove this, since $W_{\mathrm{sh}}$ is not
the levelwise weak-equivalence class.

For localization, first apply
Lemma~\ref{lem:derived-strictification}.
Any coherent arrow between these diagrams is represented
by a strict map $QX\to Y$ after an acyclic Reedy
\emph{model} fibration $QX\to X$; the cofibrant
replacement remains Reedy fibrant and weakly geometric.
Consequently the strict realization class saturates
to the coherent realization class.
Localizing in stages and using Nuiten's universal
infinity-localization at smooth hypercovers gives
the displayed equivalence, including mapping spaces.

The inclusions $\mathscr R_n\subseteq\mathscr R_{n+1}$
are automatic, and the classes and path formulas
are independent of $n$. A finite diagram has a common
finite maximum, so the Brown constructions remain
in the union. Localization preserves this filtered
union of relative categories; the finite-stage
localized inclusions are the full inclusions of
the geometric-stack subcategories.
\end{proof}

The union has no fixed global bound, but every one
of its objects still has finite outer amplitude.
It is not $\mathscr R_\infty$.

\begin{thm}[Enriched realization base change]
Let $\mathcal T$ be the small Kan-enriched site under
consideration, and let $a_{\mathrm{hyp}}$ denote its
hypersheaf localization.
\label{lem:derived-hypercomplete-base-change}
A \emph{local surjection} means
a map inducing an epimorphism on the sheaf of connected
components, equivalently an effective epimorphism in the
associated infinity-topos.

Use the global injective enriched-presheaf model or its
ordinary-sheaf or hypersheaf localization. Let $X,Y,Z$
be Reedy-fibrant strict outer diagrams, let $X\to Y$ be
a Reedy model fibration whose canonical relative homotopy
horn maps are local surjections in every positive outer
degree, and let $Z\to Y$ be any map. Then the canonical
comparison, natural in this cospan,
\[
 a_{\mathrm{hyp}}\operatorname{diag}(X\times_Y Z)
 \longrightarrow
 a_{\mathrm{hyp}}\operatorname{diag}X
 \times_{a_{\mathrm{hyp}}\operatorname{diag}Y}
 a_{\mathrm{hyp}}\operatorname{diag}Z
\]
is an equivalence. The pullback on the left is the ordinary
Reedy-model pullback and therefore computes the levelwise
homotopy pullback. Outer amplitude is unrestricted.
The geometric lifting hypotheses concern positive relative
horns; the ambient Reedy-fibration condition applies in
every degree, including zero. Absolute horns of the
diagrams are unrestricted, and the site may lack enough
points. Write $\mathbf{PB}_\tau(\mathcal T)$ for this
base-change property.
\end{thm}

\begin{proof}
Lemma~\ref{lem:enriched-relative-matching} identifies the
actual relative matching maps with the stated homotopy
ones and promotes local component lifts to strict lifts
of each finite inner simplex. The explicit diagonal-horn
factorization and finite-refinement argument of
Lemma~\ref{lem:enriched-diagonal-local-lifting} give local
strict lifting for $\operatorname{diag}(X\to Y)$
against every finite anodyne inclusion.
Diagonal commutes with the ordinary pullback, so
Lemma~\ref{lem:enriched-local-homotopy-pullback} gives
exactly the displayed comparison. Its proof constructs
the enriched $\Ex^\infty$ action and a boundary-fixed
strict-to-path comparison, and applies the enriched
local Whitehead criterion before hypersheaf left
exactness. These complete arguments are in
Appendix~\ref{app:enriched-realization}.
\end{proof}

The lifting method follows the $\varepsilon$-argument of
\cite[Property~1.7(3) and Proposition~1.19]{Pridham2013}
in its enriched HAG context. The appendix carries out
matching, lifting and local-pullback constructions on
the original small enriched site. Its mapping spaces
retain information beyond $\operatorname{Ho}(\mathcal T)$.
Localization at the final Whitehead step gives
hypersheaf base change. For prestacks,
Lemma~\ref{lem:enriched-pointwise-diagonal} provides
the pointwise diagonal comparison used below.

\begin{thm}[Unbounded hypercomplete represented CFO]
Keep the objects and fibrations of $\mathscr R_\infty$
in Definition~\ref{def:derived-reedy-models}.
\label{thm:derived-reedy-hypercomplete}
Set
\[
 \widehat{\mathsf{Real}}(X)=L^\wedge|X|,
 \qquad
 W_{\mathrm{hyp}}
 =\{f:L^\wedge|f|\text{ is an equivalence}\}.
\]
Then $(\mathscr R_\infty,W_{\mathrm{hyp}},F)$ is a full
Brown CFO, with no bound on an object's outer amplitude.
For a fibration $X\to Y$ and a map $Z\to Y$,
\[
 \widehat{\mathsf{Real}}(X\times_Y Z)
 \simeq
 \widehat{\mathsf{Real}}(X)
 \times_{\widehat{\mathsf{Real}}(Y)}
 \widehat{\mathsf{Real}}(Z).
\]
\end{thm}

\begin{proof}
The finite-matching, fibration and outer-path constructions
already work degree by degree without an outer bound.
A geometric relative horn is a smooth surjection, so its
represented image has local sections and is an effective
epimorphism. Thus the original Reedy diagrams in
$\CM_{\mathrm{sh}}$ satisfy the positive relative
homotopy-horn hypothesis of
Theorem~\ref{lem:derived-hypercomplete-base-change}.

The ordinary Reedy pullback computes the levelwise
homotopy pullback before hypercompletion. Apply
that theorem on the small Kan-enriched presentation
of $\Aff_\kappa$, and then
\eqref{eq:derived-diagonal-realization} to obtain the
displayed realization comparison. It implies stability
of acyclic-fibration pullbacks for $W_{\mathrm{hyp}}$.
The remaining Brown axioms follow from the structural
calculus and the explicit outer-path contraction.
\end{proof}

Geometry and weak representability are tested before
$L^\wedge$, in $\CM_{\mathrm{sh}}$, using the fully
faithful image of $h$. The composite $L^\wedge h$
serves only to define hypercompleted realization.
Realization always induces the canonical functor
\[
 \mathscr R_\infty[W_{\mathrm{hyp}}^{-1}]
 \longrightarrow\widehat{\mathcal X};
\]
every finite-bound category, and their union, also has
this hypercomplete Brown structure: the existing
finite-matching and prism constructions preserve each
bound, and a finite construction in the union has a
common maximum bound. Determining the mapping spaces
and essential image of the realization functors is a
further localization problem, including their relation
to the full image $L^\wedge(\dSt_n)$ at finite bounds.

\begin{question}[The ordinary-sheaf unbounded axiom]
\label{question:derived-unbounded-ordinary}
If a fibration $X\to Y$ in $\mathscr R_\infty$ is
an ordinary realization equivalence, is
$X\times_Y Z\to Z$ an ordinary realization
equivalence for every $Z\to Y$?
\end{question}

This is the precise remaining acyclic-pullback
condition $\mathbf{AP}_\infty$ for
$(\mathscr R_\infty,W_{\mathrm{sh}},F)$.
Full unbounded realization base change in
$\mathcal X$ would suffice, but is stronger than
this bare Brown axiom.
The ordinary and hypercomplete weak-equivalence classes
are defined by their respective realization functors.
Effectivity of all hypercovers characterizes
hypercompleteness
\cite[Theorem~6.5.3.12 and Corollary~6.5.3.13,
p.~680]{LurieHTT2009}, whereas the definition of
$\mathcal X$ requires covering-family descent.

\subsection{Geometric examples and the boundary of the repair}

\begin{example}[Quasi-smooth objects are insufficient]
Let $A=\R[\epsilon_{-1}]$, with zero differential.
A semifree resolution of its augmentation is
$A[\eta_{-2}]$ with $d\eta=\epsilon$.
The negative-degree pairs contract because
$d(\eta^k)=k\epsilon\eta^{k-1}$ in characteristic
zero. Therefore
\[
 \R\otimes_A^{\mathbf L}\R\simeq\R[\eta_{-2}].
\]
Its cotangent degree is $-2$, outside the
quasi-smooth interval $[-1,0]$.
Thus the required finite-limit closure is genuinely
broader than quasi-smooth derived zero loci.
\end{example}

\begin{prop}\label{prop:derived-cech}
Let $p:U\to X$ be a covering surjective etale map
of the specified derived manifolds; a finite actual
open atlas is an example. Then
\[
 C_k=\underbrace{U\times_X^h\cdots\times_X^h U}_{k+1}
\]
is a represented derived etale Lie $1$-groupoid,
its augmentation is an etale hypercover, and
$|h(C_\bullet)|\simeq h(X)$ in $\mathcal X$.
Its hypercomplete realization is $L^\wedge h(X)$.
\end{prop}

\begin{proof}
The derived pullbacks are represented, and their
projections are etale by base change.
In the infinity-slice over $X$, this is the
relative $0$-coskeleton of $U$.
Its relative matching object depends only on the
vertices of the weight. Boundaries in degrees
$k\ge1$ have the same vertices as $\Delta[k]$,
so the relative boundary maps are equivalences;
at zero the map is $p$.

Horns are contractible, so the constant-$X$ horn
matching object is $X$. At degree one the horn
maps are base changes of $p$; in higher degrees
they are equivalences.
The functor $h$ preserves these pullbacks and
makes the cover an effective epimorphism.
Its Cech realization is $h(X)$.
Hypercompletion preserves that colimit.
\end{proof}

The resulting Brown structures live on ordinary Reedy
models with weakly represented levels. Strictification
relates these models to coherent diagrams, while the
geometric and realization conditions specify fibrations
and weak equivalences separately.

%% file: represented-criterion.tex
\section{A reusable represented-horn criterion}
\label{sec:represented-criterion}

The preceding proofs separate a geometric closure problem
from a realization problem. The criterion below packages
the Brown--Reedy construction in terms of a finite-limit
geometric image and a realization functor. Attaching the
weak-equivalence class to the realization functor allows
the structural argument to be reused across geometries.

\begin{defn}[Representability data]
\label{def:representability-data}
Let $\CM$ be a combinatorial simplicial model presentation
of an infinity-category $\mathcal X$, and let
\[
 h:\mathcal D\hookrightarrow\mathcal X
\]
be fully faithful and finite-limit preserving, with
$\mathcal D$ admitting finite limits.
Let $\mathsf C$ be an equivalence-invariant class of
maps of $\mathcal D$, containing equivalences and
stable under composition and homotopy base change.
All model limits below are between fibrant objects
and along model fibrations, so that they compute
the corresponding homotopy limits.

Let $\mathscr R_{\mathsf C}$ consist of ordinary outer
Reedy-fibrant diagrams in $\CM$ with levels weakly in
the image of $h$ and positive homotopy horns in $\mathsf C$.
Its fibrations are Reedy model fibrations with positive
relative homotopy horns in $\mathsf C$.
For finite $n$, impose that the absolute horns above
$n$ are equivalences to obtain $\mathscr R_{\mathsf C,n}$.
\end{defn}

\begin{prop}[Geometric realization criterion]
\label{prop:geometric-realization-criterion}
Suppose the representability data of
Definition~\ref{def:representability-data} are given.
Let $\mathsf{Real}$ be a functor from these strict
diagrams to an infinity-category with finite limits,
and suppose that:
\begin{enumerate}
\item it inverts levelwise model equivalences and the
constant outer-path maps $X\to PX$;
\item for a fibration $X\to Y$ of these diagrams and
any $Z\to Y$, the natural comparison
\[
 \mathsf{Real}(X\times_Y Z)\longrightarrow
 \mathsf{Real}(X)\times_{\mathsf{Real}(Y)}
                          \mathsf{Real}(Z)
\]
is an equivalence.
\end{enumerate}
Then $\mathscr R_{\mathsf C}$ is a full Brown CFO
for these fibrations and the weak maps detected by
$\mathsf{Real}$. Each finite-bound subcategory and the union
over all finite bounds have the same conclusion whenever
the two realization hypotheses hold on the categories in question.
\end{prop}

\begin{proof}
Finite weighted matching objects are computed by the
cell-induction proof of
Lemma~\ref{lem:derived-weighted-matching}.
For these data its geometric input is precisely
finite-limit closure of the full image of $h$.
The induction therefore applies to each geometry below.

The relative-horn factorizations in
Lemma~\ref{lem:derived-reedy-calculus} involve only
composition and base change in $\mathsf C$.
They give terminal objects, finite products,
composition of fibrations, and all their ordinary
pullbacks in $\mathscr R_{\mathsf C}$.
Every object maps fibrationally to the terminal diagram:
the Reedy condition is its object fibrancy, and the
positive relative horns are its absolute horns.
In particular the absolute horn of a pullback
factors through a base change of a relative horn
and then of an absolute horn. Thus the pullback retains
the absolute horn conditions as well as represented levels.

Lemma~\ref{lem:derived-prism-filtration} is purely
simplicial. The proof of
Proposition~\ref{prop:derived-outer-path} therefore
gives a Reedy-fibrant, weakly represented $PX$
and an endpoint fibration for these data.
The first hypothesis makes its constant-path map
a weak equivalence. Above a finite bound $n$,
the prism uses only dimensions $m,m+1>n$,
so paths retain that bound. The same bound is
retained by products and fibration pullbacks.

The weak maps have two-out-of-three because
equivalences in the realization target do.
The second hypothesis makes the pullback of
an acyclic fibration a weak equivalence.
The other Brown axioms were already proved by
the structural and path arguments.
Every finite collection of objects in the union over finite bounds
has a common maximum bound; those same constructions stay within
that bound.
\end{proof}

\subsection{Realization hypotheses}

The second hypothesis concerns realization weak maps,
whereas ambient right properness concerns levelwise
model weak equivalences. The applications use the
following realization comparisons.

For finite derived-smooth presentations, the second hypothesis
is Nuiten's ordinary-sheaf realization theorem.
For the unbounded derived-smooth application it is the
property $\mathbf{PB}_\tau(\mathcal T)$ proved in
Theorem~\ref{lem:derived-hypercomplete-base-change}.
The Banach hypersheaf application uses
$\mathbf{SO}(\mathcal C)$ from
Theorem~\ref{lem:banach-hypercover-refinement}, to supply
its represented hypersheaf geometry, and then
$\mathbf{PB}_\tau(\mathcal C)$ for realization base change.

For the associative prestack applications, geometric maps
have homotopy sections. Evaluation gives pointwise
relative-horn surjectivity, and
Lemma~\ref{lem:enriched-pointwise-diagonal} supplies the
diagonal-fibration conclusion under the stated Reedy hypotheses.
Here both realization and pullbacks in the presheaf
category are computed pointwise.

In each case the first hypothesis follows from
derived realization and the explicit outer
simplicial contraction. The realization target remains
part of the data: ordinary sheaves for the finite
derived-smooth application, hypersheaves for the
unbounded derived-smooth and Banach applications,
and prestacks for the associative applications.

The criterion is a form of the represented-hypergroupoid
method \cite{Pridham2013}. Geometric finite limits are
essential to its hypotheses. Their failure in the
ordinary manifold category leads instead to the iCFO
construction of Theorem~\ref{thm:geometric-icfo}, which
treats existence and stability of pullbacks separately.

%% file: banach-geometry.tex
\section{Banach pregeometry and represented hypercomplete groupoids}
\label{sec:banach}

This section develops a Banach-valued structured geometry
for the represented-horn criterion of
Section~\ref{sec:represented-criterion}.
Classical structured charts and a valuewise hypercompletion
coreflection lead to hyp-affine charts. Open-local algebraic
representatives then give geometric finite limits through
the structured-spectrum and effective strict-open-gluing
theorems of \cite{LurieDAGV2009}.

We prove both clauses of enriched split-open descent
$\mathbf{SO}(\mathcal C)$ and deduce hypersheaf membership
of the represented functors. Combining this with
$\mathbf{PB}_\tau(\mathcal C)$ gives the represented
Brown structure. Geometric finite limits, hyperdescent
and realization base change play distinct roles in
this construction.

The analytic input is the Banach calculus of
\cite{Schmeding2021}. The finite matching, Reedy and prism
arguments are those of Section~\ref{sec:derived}, organized
by Proposition~\ref{prop:geometric-realization-criterion}.
Every outer bound uses the same hypercompleted realization
weak-equivalence class.

We write $\mathrm{Str}_T$ for the infinity-category of
$T$-structured spaces in geometric (space) variance, called
\emph{physical variance} below, and $\mathrm{Str}_T^{\mathrm{vh}}$
for its full subcategory with hypercomplete structure values.
These notions are defined below. Valuewise hypercompleteness
does not require the underlying infinity-topos to be hypercomplete.

\subsection{The Banach pregeometry}
\label{subsec:banach-pregeometry}

Fix a Grothendieck universe and a small category $T$ of Banach
domains: nonempty and empty open subsets of Banach spaces, with the
smooth maps of the Frechet-derivative calculus
\cite{Schmeding2021}. We require $T$ to be closed under
finite products and under passage to arbitrary open subsets; the
empty domain is included.

All Banach domains quantified over below are objects of this
fixed chart category. A Banach space used as a whole-space
chart is required to have that domain in $T$.
When ordinary split Banach submersions are compared with the
generated smooth class, their required local product charts,
including the fibre domains, must belong to the chosen
category. Domains may be infinite-dimensional, and open
covers are arbitrary small families.

\begin{defn}[Admissibles and covers]
\label{defn:banach-pregeometry}
The \emph{admissible} morphisms of $T$ are \emph{exactly} the open
embeddings. A \emph{covering family} of $U\in T$ is an open covering
family $\{U_i\hookrightarrow U\}$, including the empty family when
$U=\emptyset$. Since the pullback of an open embedding
along an arbitrary smooth map is an open subset of the source, the
admissible class is stable under base change and $T$ is a
pregeometry in the sense of
\cite[Definition~3.1.1]{LurieDAGV2009}.
\end{defn}

Noninjective local diffeomorphisms are excluded from the
admissible class: their pullbacks along smooth maps can leave the class of
domains for general Banach manifolds. Locality of structures is
therefore tested only against open embeddings. The smooth product
projections used later for the geometric groupoid model
(Subsection~\ref{subsec:banach-cfo}) form a separate class, not part
of the pregeometry.

\begin{defn}[Physical structured spaces]
\label{defn:banach-structured}
A \emph{$T$-structure} on an infinity-topos $X$ is a functor
$\mathcal O\colon T\to X$ preserving finite products, pullbacks
along admissibles, and effective covering families. The
\emph{infinity-category} $\mathrm{Str}_T$ has objects the pairs
$(X,\mathcal O)$; a morphism $(Y,\mathcal P)\to(X,\mathcal O)$ is a
geometric morphism with inverse image $f^{*}\colon X\to Y$ together
with a natural transformation $f^{*}\mathcal O\to\mathcal P$ that is
\emph{local}, i.e.\ cartesian on every admissible open map. This is
the physical variance of
\cite[Definitions~3.1.4 and~3.1.9,
pp.~82--85]{LurieDAGV2009}.
A \emph{strict open} of $(X,\mathcal O)$ is
$(X/a,\mathcal O|_a)$ for a subterminal object $a$ of $X$, with
$\mathcal O|_a$ the composite of $\mathcal O$ with $X\to X/a$; a
family of strict opens \emph{covers} if the subterminals jointly
cover $1_X$. Write $\mathrm{Str}_T^{\mathrm{vh}}\subset\mathrm{Str}_T$
for the full subcategory of objects whose structure values
$\mathcal O(U)$ are all hypercomplete.
\end{defn}

By \cite[Definition~3.4.1 and Lemma~3.4.3,
pp.~96--97]{LurieDAGV2009} the geometric envelope $j\colon T\to\E$
(Subsection~\ref{subsec:banach-envelope}) preserves products and
admissible pullbacks, and $\E$ has finite limits; a $T$-structure on
$X$ extends canonically to an $\E$-structure, and on each fixed
topos the categories of $T$- and $\E$-structures together with their
local morphisms agree \cite[Proposition~3.4.5,
pp.~97--98]{LurieDAGV2009}. This comparison takes place over a
\emph{fixed} topos. Finite limits of geometric structured spaces
will instead be constructed from spectra and open gluing in
Theorem~\ref{thm:banach-geometric-limits}.

Actual strict-open base change and effective strict-open gluing of
structured spaces are the open special cases of the etale-slice and
gluing results
\cite[Proposition~2.3.5(3)--(5), p.~51 (opening on
p.~50), and Remark~2.3.20, p.~57]{LurieDAGV2009}. Concretely, open
restriction along $f\colon X\to Y$ satisfies
\begin{equation}
\label{eq:banach-open-restriction}
 X\times_Y (Y/b,\mathcal O_Y|_b)
   =(X/f^{*}b,\mathcal O_X|_{f^{*}b}),
\end{equation}
and coherent open-overlap data glue to an actual structured space
whose atlas and overlaps are effective, the gluing being a colimit
in $\mathrm{Str}_T$. These open special cases apply directly
to structured spaces.

\subsection{Classical charts and the chart-evaluation formula}
\label{subsec:banach-charts}

For a Banach domain $U$, let
$\Sh(U)=\Sh_\infty(\operatorname{Opens}(U))$ be the
space-valued localic infinity-topos of its open subsets.
Define the classical chart by
\begin{equation}
\label{eq:banach-chart}
 \operatorname{Chart}(U)=(\Sh(U),\mathcal O_U),
 \qquad
 \mathcal O_U(V)(W)=C^{\infty}(W,V),
\end{equation}
for $V\in T$ and $W$ an open subset of $U$. The universal property
of smooth maps into a product, together with the fact that the
inverse images of an open cover of a target form an open cover
of the source, shows that $\mathcal O_U$ preserves finite
products, admissible pullbacks, and effective covering families, so
$\operatorname{Chart}(U)\in\mathrm{Str}_T$. The \emph{global} functor
$V\mapsto C^{\infty}(U,V)$ need not carry open covers to
epimorphisms of spaces. The chart structure is therefore
sheaf-valued on $U$, with Banach-valued smooth maps as its
basic data.

\begin{lem}[Chart evaluation]
\label{lem:banach-chart-eval}
For $(X,\mathcal O)\in\mathrm{Str}_T$ and $U\in T$ there is a
natural equivalence of mapping spaces
\begin{equation}
\label{eq:banach-chart-eval}
 \Map_{\mathrm{Str}_T}\!\bigl((X,\mathcal O),\operatorname{Chart}(U)\bigr)
 \;\simeq\;
 \Gamma_X\mathcal O(U).
\end{equation}
In particular classical charts are fully faithful and preserve
products and admissible pullbacks.
\end{lem}

\begin{proof}
A section $\alpha$ of $\mathcal O(U)$ assigns to each open
$V\subseteq U$ the subterminal
$1_X\times_{\mathcal O(U)}\mathcal O(V)$, using the monomorphism
$\mathcal O(V)\to\mathcal O(U)$ classifying the admissible
$V\hookrightarrow U$. Admissible pullbacks compute finite
intersections of these subterminals, and the effective-cover axiom
applied to a family covering its union computes arbitrary unions.
Thus $\alpha$ determines a frame homomorphism
$\mathrm{Opens}(U)\to\mathrm{Sub}(1_X)$ and hence a geometric
morphism $X\to\Sh(U)$ with the prescribed inverse image.

The sheaf $\mathcal O_U(V)$ is generated by the local pairs
$(W\subseteq U,\ g\colon W\to V)$ with $g$ smooth. Restricting
$\alpha$ to the subterminal attached to $W$ and applying
$\mathcal O(g)$ produces a section of $\mathcal O(V)$ over that
subterminal; these agree on overlaps and assemble to a natural
transformation $f^{*}\mathcal O_U\to\mathcal O$. For an admissible
$V'\subseteq V$, its preimage under $g$ is an open of $W$, and the
admissible-pullback axiom makes the comparison square cartesian, so
the transformation is local. The construction uses the frame
of opens and the Banach-valued generating pairs.

Conversely, evaluate a local structure map on the identity section
of $U$: locality recovers the frame homomorphism, and naturality on
every generating pair recovers all structure components. Both
assignments are natural in the space of sections $\alpha$, because a
geometric morphism to a localic topos is determined by its frame map
and a local transformation by its values on the generating pairs and
their descent. The two constructions are therefore mutually inverse
with contractible spaces of choices over each $\alpha$, hence give
an equivalence of mapping spaces, not merely a bijection on
$\pi_0$.
\end{proof}

We record a structured test object that separates a genuinely
derived intersection from the ordinary point. On the topos of
spaces define, for an open $U$ of a Banach space $V$,
\begin{equation}
\label{eq:banach-square-zero}
 F(U)=\coprod_{u\in U}K(V_{\mathrm{add}},1),
 \qquad
 F(f)(u,\xi)=(f(u),\,df_u(\xi)),
\end{equation}
where $V_{\mathrm{add}}$ denotes the underlying additive group,
regarded as discrete, and $K(V_{\mathrm{add}},1)$ is its
Eilenberg--Mac Lane space. The coproduct is indexed by the
underlying set of $U$; $df_u(\xi)$ denotes the map of these
spaces induced by the derivative. The chain rule and additivity
of the derivative
\cite[Lemma~1.2.9 and Proposition~1.2.10, p.~17]{Schmeding2021}
make $F$ functorial and
product-preserving. An open embedding is, after an isomorphism of
charts, a literal open inclusion whose derivative identifies the
tangent fibres, so $F$ carries it to the inclusion of the
corresponding components and preserves admissible pullbacks. Open
covers are epimorphic on $\pi_0=U$, and $F(\emptyset)=\emptyset$,
$F(*)=*$. Hence $(\mathcal S,F)$ is a $T$-structured space with
$1$-truncated values. We also write $F$ for this structured space.
The constant and projection maps
between $F$ and the classical point structure are local precisely
because only open embeddings are admissible; were every projection
$U\to*$ declared admissible, locality would force every structure
morphism to be an equivalence on each sort, and $F$ would collapse.

\subsection{Valuewise hypercompletion}
\label{subsec:banach-hypercompletion}

Write $L_X$ for hypercompletion followed by the inclusion of
hypercomplete objects back into $X$. Its unit is infinity-connective
and $L_X$ preserves finite limits
\cite[pp.~668--669]{LurieHTT2009}. Define
\begin{equation}
\label{eq:banach-H}
 H(X,\mathcal O)=(X,L_X\mathcal O),
\end{equation}
leaving the underlying topos unchanged.

That $L_X\mathcal O$ is again a $T$-structure follows sortwise.
Finite products and admissible pullbacks are preserved because
$L_X$ is left exact. Covers are preserved because the effective
epimorphism
$\coprod_i\mathcal O(U_i)\to\mathcal O(U)\to L_X\mathcal O(U)$ factors
through $\coprod_i L_X\mathcal O(U_i)$, so the latter family is
effective epi by right cancellation; preservation of coproducts by
the inclusion is not needed. For an admissible $V\hookrightarrow U$
the map $\mathcal O(V)\to\mathcal O(U)$ is a monomorphism, whose
classifying map lands in the hypercomplete
$\Omega_X$, so left exactness gives
$\mathcal O(V)=\mathcal O(U)\times_{L_X\mathcal O(U)}L_X\mathcal O(V)$
and the unit $\mathcal O\to L_X\mathcal O$ is local. Thus $H$ sends
$\mathrm{Str}_T$ into $\mathrm{Str}_T^{\mathrm{vh}}$ and fixes
classical charts and the structure $F$ of
\eqref{eq:banach-square-zero}, whose values are already
hypercomplete.

\begin{thm}[Coreflection]
\label{thm:banach-coreflection}
The inclusion
$\mathrm{Str}_T^{\mathrm{vh}}\hookrightarrow\mathrm{Str}_T$ admits a
right adjoint $H$: for $Z\in\mathrm{Str}_T^{\mathrm{vh}}$ and
$X\in\mathrm{Str}_T$,
\begin{equation}
\label{eq:banach-coreflection}
 \Map_{\mathrm{Str}_T}(Z,X)
 \;\simeq\;
 \Map_{\mathrm{Str}_T^{\mathrm{vh}}}(Z,HX),
\end{equation}
with counit $HX\to X$ the unit $\mathcal O\to L_X\mathcal O$.
Moreover $H$ commutes with restriction to strict opens.
\end{thm}

\begin{proof}
Fix an inverse image $f^{*}\colon X\to Y$ underlying a morphism with
hypercomplete target structure $\mathcal P$ on $Y$. The maps
$f^{*}\mathcal O(U)\to f^{*}L_X\mathcal O(U)$ are
infinity-connective, so mapping into the hypercomplete
$\mathcal P(U)$ yields equivalences of mapping spaces, natural in
$U$; these assemble in the end computing natural transformations and
give a factorisation $\beta\colon f^{*}L_X\mathcal O\to\mathcal P$.
Its locality on $V\hookrightarrow U$ is the comparison
\[
 f^{*}L_X\mathcal O(V)\longrightarrow
 f^{*}L_X\mathcal O(U)\times_{\mathcal P(U)}\mathcal P(V),
\]
which becomes an equivalence after pulling back along the effective
epimorphism $f^{*}\mathcal O(U)\to f^{*}L_X\mathcal O(U)$, by
unit-locality and locality of the original map. Effective-epimorphic
pullback is conservative, so $\beta$ is local; conversely
composition with the local unit preserves locality. The equivalence
is natural over the space of underlying geometric morphisms, giving
\eqref{eq:banach-coreflection}. For restriction to strict opens, an object
$E\to a$ over a subterminal $a$ is hypercomplete in $X/a$ iff its
total object is hypercomplete in $X$, since slice mapping spaces are
fibres over maps into the hypercomplete $a$; left exactness then
identifies the slice hypercompletion with the restriction of
$L_X$.
\end{proof}

If $L=\lim_i X_i$ exists in $\mathrm{Str}_T$, the adjunction
\eqref{eq:banach-coreflection} identifies $HL$ with
$\lim_i HX_i$ in $\mathrm{Str}_T^{\mathrm{vh}}$.
For formal spectrum cones, existence follows from the spectrum
adjunction, so this observation applies. Finite limits of
geometric objects in $\mathrm{Str}_T^{\mathrm{vh}}$ are constructed
by strict-open gluing in Theorem~\ref{thm:banach-geometric-limits}.

\begin{lem}[Open restriction, gluing, monomorphic-cover locality]
\label{lem:banach-open-gluing}
Hypercompletion commutes with strict-open restriction, and effective
strict-open gluing of hypercomplete-valued structures stays in
$\mathrm{Str}_T^{\mathrm{vh}}$. Moreover, if subobjects
$A_i\hookrightarrow A$ in a topos form an effective epimorphic cover
and each $A_i$ is hypercomplete, then $A$ is hypercomplete.
\end{lem}

\begin{proof}
Commutation with restriction is the last statement of
Theorem~\ref{thm:banach-coreflection}. For the monomorphic-cover
fact, unit-locality on the monos $A_i\hookrightarrow A$ gives
$A_i=A\times_{L A}L A_i$. The coproduct $\coprod_i L A_i\to L A$ is
effective epi, and the unit $A\to L A$ becomes an equivalence after
this cover, hence is an equivalence. For gluing, apply this
sortwise: the structure values of the glued object, restricted to
the covering subterminals, are hypercomplete total objects by the
slice identification, so their effective-epi assembly is
hypercomplete.
\end{proof}

\subsection{The geometric envelope and hyp-affines}
\label{subsec:banach-envelope}

Let $j\colon T\to\E$ be the geometric envelope and let
$\E_{\mathrm{cell}}\subseteq\E$ be the full subcategory generated
from $j(T)$ by finite limits and equivalences. This specifies the
cell class $\E_{\mathrm{cell}}$ inside $\E$; retract closure is
a separate operation.

The absolute spectrum in physical variance is the right adjoint
\begin{equation}
\label{eq:banach-spec-adjunction}
 \Spec_{\E}\colon\operatorname{Pro}(\E)\longrightarrow\mathrm{Str}_T
\end{equation}
to physical global sections
\cite[Notation~2.2.1, p.~43]{LurieDAGV2009}; there is no additional
$\op$ on $\operatorname{Pro}(\E)$. It yields, for $u\in\E$ and
$Z\in\mathrm{Str}_T$,
\begin{equation}
\label{eq:banach-spec-formula}
 \Map_{\mathrm{Str}_T}(Z,\Spec_{\E}u)
 \simeq
 \Map_{\operatorname{Pro}(\E)}(\Gamma^{\mathrm{phys}}Z,u)
 \simeq\Gamma_Z\mathcal O_Z(u),
\end{equation}
where $\mathcal O_Z$ is the canonical extension of the structure to
$\E$. The underlying admissible site of $\Spec_{\E}u$ is
$\E^{\mathrm{ad}}_{/u}$ \cite[Definition~2.2.9,
p.~46]{LurieDAGV2009}, and its structure value at $v$ is the
\emph{ordinary sheafification} of
\begin{equation}
\label{eq:banach-eval-presheaf}
 (w\to u)\longmapsto\Map_{\E}(w,v)
\end{equation}
\cite[Proposition~2.2.11 and Theorem~2.2.12,
p.~47]{LurieDAGV2009}. Ordinary sheafification is part of the
spectrum construction, regardless of whether the admissible site
is subcanonical. Set
\begin{equation}
\label{eq:banach-hyp-affine}
 a(u)=H\,\Spec_{\E}(u)\qquad(u\in\E_{\mathrm{cell}}),
\end{equation}
the \emph{hyp-affines}.

\begin{lem}[Admissibles are monic and preserve the cell class]
\label{lem:banach-admissible-mono}
Every $\E$-admissible is a monomorphism, and every $\E$-admissible
with target in $\E_{\mathrm{cell}}$ has its source in
$\E_{\mathrm{cell}}$.
\end{lem}

\begin{proof}
For a $T$-open embedding $i\colon U\hookrightarrow V$ the diagonal
$U\to U\times_V U$ is an equivalence; this admissible pullback is
preserved by $j$, so $j(i)$ is mono. Let $A_0$ be the class obtained
from the maps $j(i)$ by base change, finite composition, and
equivalence; all its maps are monos, and it satisfies triangle
cancellation because if $g$ and $h=gf$ lie in $A_0$ then $f$ is the
pullback of the mono $h$ along $g$.

$A_0$ is closed under arrow retracts. Let $f\colon X\to Y$ be an
arrow retract of $g\colon U\to V$ with sections $i_X,i_Y$ and
retractions $r_X,r_Y$, and put $P=Y\times_V U$ along $i_Y$. The maps
$(f,i_X)\colon X\to P$ and $r_X\circ\mathrm{pr}_U\colon P\to X$ over
$Y$ (well defined since $f r_X=r_Y g$ and $r_Y i_Y=\id$) exhibit $X$
and $P$ as the same subobject of $Y$, because monos are closed under
retract and base change and subobjects of $Y$ form a poset. Thus the
retract is a base change of $g$ along the section $i_Y$, coherently.
Hence $A_0$ satisfies the admissibility axioms and, being the
smallest such class containing $j$ of the $T$-opens, is the
admissible class of the coarsest compatible geometry on $\E$.

Finally, a base change of a finite composite is the finite composite
of successive base changes, so an $A_0$-map into $u\in
\E_{\mathrm{cell}}$ is a finite tower of admissible base changes;
each intermediate source is a finite pullback of objects already in
the finite-limit-closed $\E_{\mathrm{cell}}$, hence stays in
$\E_{\mathrm{cell}}$.
\end{proof}

\begin{cor}[Strict-open basis of a hyp-affine]
\label{cor:banach-affine-basis}
For $u\in\E_{\mathrm{cell}}$ the underlying topos of $\Spec_{\E}u$
is generated by subterminals, with a basis of strict opens
$\Spec_{\E}w$ for admissibles $w\to u$ with $w\in\E_{\mathrm{cell}}$.
The same strict opens $a(w)$ form a basis for $a(u)$.
\end{cor}

\begin{proof}
The site $\E^{\mathrm{ad}}_{/u}$ is a full subcategory of the poset
of subobjects of $u$; its representables are subterminal presheaves
whose sheafifications are subterminal and generate the sheaf topos.
The slice at such a generator is $\Spec_{\E}w$
\cite[Example~2.3.8, p.~51, and Remark~2.3.20,
p.~57]{LurieDAGV2009}, and
$w\in\E_{\mathrm{cell}}$ by
Lemma~\ref{lem:banach-admissible-mono}. Sheafification may identify
site objects; this is harmless, as the generators need not embed
fully faithfully. Since $H$ keeps the topos and commutes with
strict-open restriction (Theorem~\ref{thm:banach-coreflection}), the
same subterminals give the basis for $a(u)$.
\end{proof}

\begin{lem}[Universal hyp-affines]
\label{lem:banach-hyp-affine-mapping}
For $u\in\E_{\mathrm{cell}}$ and
$Z\in\mathrm{Str}_T^{\mathrm{vh}}$,
\begin{equation}
\label{eq:banach-hyp-affine-mapping}
 \Map_{\mathrm{Str}_T^{\mathrm{vh}}}(Z,a(u))
 \simeq\Gamma_Z\mathcal O_Z(u).
\end{equation}
The functor $a$ sends finite $\E_{\mathrm{cell}}$ limit diagrams to
cones universal in $\mathrm{Str}_T^{\mathrm{vh}}$, and
$a(jU)=\operatorname{Chart}(U)$.
\end{lem}

\begin{proof}
Apply the coreflection \eqref{eq:banach-coreflection} to the
spectrum formula \eqref{eq:banach-spec-formula}; the extended
structure $\mathcal O_Z$ and global sections $\Gamma_Z$ are left
exact, so \eqref{eq:banach-hyp-affine-mapping} carries each finite
$\E_{\mathrm{cell}}$ limit cone to a limit of mapping spaces for
every $Z$, establishing the limit cone in
$\mathrm{Str}_T^{\mathrm{vh}}$. For $u=jU$ the right-hand
side is $\Gamma_Z\mathcal O_Z(U)$, so
Lemma~\ref{lem:banach-chart-eval} and Yoneda identify $a(jU)$ with
$\operatorname{Chart}(U)$. Extension of a structure to $\E$ commutes
with valuewise completion, since $L_X\mathcal O_{\E}$ is again an
$\E$-structure restricting to $L_X\mathcal O_T$, identified with the
completed extension by \cite[Proposition~3.4.5,
pp.~97--98]{LurieDAGV2009}.
\end{proof}

\begin{lem}[Local algebraic representatives]
\label{lem:banach-local-representative}
For $u,v\in\E_{\mathrm{cell}}$ and a map
$f\colon a(v)\to a(u)$ there are a strict-open cover of $a(v)$ by
$a(v_i)$ with $v_i\to v$ admissible and $v_i\in\E_{\mathrm{cell}}$,
and $\E$-maps $f_i\colon v_i\to u$, such that $f|_{a(v_i)}$ is
homotopic to $a(f_i)$. The representatives and homotopies
are local choices on this cover.
\end{lem}

\begin{proof}
There are two successive local lifting steps. First,
\eqref{eq:banach-hyp-affine-mapping} identifies $f$ with a section
of $L_{X_v}\mathcal O_v(u)$ on the unchanged topos $X_v$ of
$\Spec_{\E}v$. The unit
$\mathcal O_v(u)\to L_{X_v}\mathcal O_v(u)$ is
infinity-connective, hence effective epi; pulling it back along $f$
gives an effective epimorphism $E\to 1$ in $X_v$. Since $X_v$ is
generated by the subterminal basis of
Corollary~\ref{cor:banach-affine-basis}, the object $E$ is
effectively covered by \emph{maps} from those generators; these maps
into $E$ need not be monos, but their composites to $1$ cover $1$ by
composition of effective epimorphisms and thus give a strict-open
cover on which the section lifts to $\mathcal O_v(u)$ with a chosen
homotopy. Second, use the presheaf in
\eqref{eq:banach-eval-presheaf}, whose ordinary sheafification is
$\mathcal O_v(u)$. The canonical map from the coproduct of basis
generators labelled by $\E$-maps $w\to u$ to $\mathcal O_v(u)$
is effective epi. After a further open-basis refinement, the
section therefore lifts to actual $\E$-maps $v_i\to u$, again
up to chosen paths in the sheafification. Combining the refinements
gives admissibles $v_i\to v$ with sources in $\E_{\mathrm{cell}}$
(Lemma~\ref{lem:banach-admissible-mono}), and
\eqref{eq:banach-hyp-affine-mapping} transports the chosen homotopies
back to homotopies of structured maps. Thus the two local lifting
steps retain the chosen paths in the structured mapping spaces.
\end{proof}

Call $\mathcal G$ the full subcategory of
$\mathrm{Str}_T^{\mathrm{vh}}$ of objects admitting a strict-open
atlas by hyp-affines $a(u)$, $u\in\E_{\mathrm{cell}}$. Say a
monomorphism $h_P\to \Phi$ of presheaves on
$\mathrm{Str}_T^{\mathrm{vh}}$ is \emph{strict-open representable} if
its pullback along every $h_Z\to\Phi$ is the Yoneda image of a
strict open of $Z$.

\begin{lem}[Local representability]
\label{lem:banach-local-representability}
Let $\Phi$ be an ordinary sheaf for the strict-open topology on
$\mathrm{Str}_T^{\mathrm{vh}}$ equipped with a locally covering
family of strict-open-representable monos $h_{P_i}\to\Phi$ with each
$P_i\in\mathcal G$. Then $\Phi\simeq h_Y$ for the actual structured
gluing $Y\in\mathcal G$ of the $P_i$ along their open overlaps, and
the equivalence holds against every $\mathrm{Str}_T^{\mathrm{vh}}$
test.
\end{lem}

\begin{proof}
The pairwise and higher overlaps are strict opens of the $P_i$ by
representability, and their coherent identifications and cocycles are
extracted from the single sheaf $\Phi$. Effective strict-open gluing
\eqref{eq:banach-open-restriction} produces $Y$ with atlas $\{P_i\}$
and exactly those overlaps; the glued structure is valuewise
hypercomplete by Lemma~\ref{lem:banach-open-gluing}. Descent for
$\Phi$ on this effective atlas assembles compatible elements of
$\Phi(P_i)$ into $\Phi(Y)$ and yields $h_Y\to\Phi$ by Yoneda.

To see this is an equivalence, note $h_Y\times_\Phi h_{P_i}=h_{Y_i'}$
for a strict open $Y_i'\subseteq Y$ by representability; pulling back
to each atlas piece $P_j$ gives the specified overlap
$P_j\times_\Phi P_i$, so the strict opens $Y_i'$ and $P_i$ of $Y$
have equal restriction to every $P_j$ and hence coincide, subobjects
being determined locally. Thus $h_Y\to\Phi$ is an equivalence after
the effective epi $\coprod_i h_{P_i}\to\Phi$ in the ordinary sheaf
category; effective-epimorphic pullback is conservative, so
$h_Y=\Phi$ as sheaves and as presheaves, i.e.\ against every test.
Refining the $P_i$ by hyp-affine atlases exhibits $Y\in\mathcal G$.
\end{proof}

\subsection{Geometric finite limits}
\label{subsec:banach-g1}

Every represented $h_Y$ with $Y\in\mathcal G$ is already an ordinary
sheaf for the strict-open topology: an effective strict-open atlas
is a colimit in $\mathrm{Str}_T^{\mathrm{vh}}$, so mapping out of its
Cech diagram gives descent, using only
\eqref{eq:banach-open-restriction} and effective gluing.

\begin{thm}[Geometric limits over the Banach site]
\label{thm:banach-geometric-limits}
\ \begin{enumerate}
\item $\mathcal G$ contains the classical charts fully faithfully and
is closed under strict opens and coherent strict-open gluing.
\item Every finite diagram in $\mathcal G$ has a limit in
$\mathcal G$, and the limit cone is universal against every test
object of $\mathrm{Str}_T^{\mathrm{vh}}$.
\item Let $\mathcal C$ be the full finite-limit and strict-open
closure of the classical charts inside $\mathcal G$. Every
hyp-affine lies in $\mathcal C$, every $\mathcal C$-object has a
hyp-affine atlas, and $\mathcal G$ is exactly the strict-open
$\mathcal C$-gluing class.
\end{enumerate}
These are limits in the infinity-category of geometric structured
spaces, with the mapping-space universality in \textup{(2)}.
The construction uses hyp-affines and effective strict-open gluing.
\end{thm}

\begin{proof}
\emph{(1)} Strict-open restrictions of hyp-affines have hyp-affine
bases (Corollary~\ref{cor:banach-affine-basis}), so restricting and
refining atlases shows $\mathcal G$ is closed under strict opens;
effective strict-open gluing of $\mathcal G$-objects returns to
$\mathcal G$ by Lemma~\ref{lem:banach-local-representability}. Full
faithfulness of $\operatorname{Chart}$ is
Lemma~\ref{lem:banach-chart-eval}.

\emph{(2)} The terminal object is $a(j*)=\operatorname{Chart}(*)$,
terminal against all tests by
\eqref{eq:banach-hyp-affine-mapping}. For $X\to Z\leftarrow Y$ in
$\mathcal G$ form the presheaf fibre product
$\Phi=h_X\times_{h_Z}h_Y$, an ordinary sheaf by the remark above.
Choose a hyp-affine atlas $a(u_k)\to Z$; its strict-open preimages
in $X$ and $Y$ exist by \eqref{eq:banach-open-restriction} and lie
in $\mathcal G$, and we choose hyp-affine atlases of them. Applying
Lemma~\ref{lem:banach-local-representative} to their actual maps to
$a(u_k)$ yields, after refinement, hyp-affine pieces $a(v_{ki})$,
$a(w_{kj})$ strict open in $X$, $Y$ and $\E_{\mathrm{cell}}$-maps
$v_{ki}\to u_k\leftarrow w_{kj}$ representing the given maps up to
chosen homotopies. Set
\[
 P_{kij}=a\!\left(v_{ki}\times_{u_k}w_{kj}\right),
\]
which exists with the universal local-pullback property by
Lemma~\ref{lem:banach-hyp-affine-mapping}. The maps
$h_{P_{kij}}\to\Phi$ are strict-open representable: after any test
$h_Q\to\Phi$ the pullback is the intersection of the preimages in
$Q$ of the chosen source opens, computed by
\eqref{eq:banach-open-restriction}, and the compatibility over $Z$,
including its homotopy, is retained. These maps cover $\Phi$: given
$Q\to\Phi$, pull the atlas of $Z$ back along the common composite
using its homotopy, then pull the $X$- and $Y$-atlases back and
intersect; no finiteness of atlases is used.
Lemma~\ref{lem:banach-local-representability} represents $\Phi$ by an
actual $P\in\mathcal G$, whence for every
$Q\in\mathrm{Str}_T^{\mathrm{vh}}$
\[
 \Map(Q,P)\simeq
 \Map(Q,X)\times^h_{\Map(Q,Z)}\Map(Q,Y).
\]
This is the required geometric pullback; with the terminal object it
gives finite-limit closure and universality.

\emph{(3)} The closure $\mathcal C$ exists inside $\mathcal G$ by
parts (1)--(2) and is essentially small once the universe is fixed:
the initial charts, their mapping spaces, finite diagrams and
subterminal lattices are small, and iterating finitary operations
and equivalences preserves essential smallness. Every $a(u)$,
$u\in\E_{\mathrm{cell}}$, is built from $j(T)$ by finite limits;
applying $a$ and Lemma~\ref{lem:banach-hyp-affine-mapping} stepwise
places it in $\mathcal C$. Conversely each $\mathcal C$-object lies
in $\mathcal G$ and has a hyp-affine atlas, so $\mathcal G$ is the
strict-open $\mathcal C$-gluing class.
\end{proof}

Give $\mathcal C$ the topology generated by its intrinsic
strict-open covers. Effective structured gluing shows that
$h_Y(X)=\Map_{\mathrm{Str}_T^{\mathrm{vh}}}(X,Y)$ is an ordinary
sheaf on $\mathcal C$ for every
$Y\in\mathrm{Str}_T^{\mathrm{vh}}$; in particular the open topology
is subcanonical before hypercompletion. Subcanonicity here
follows from effective structured gluing.

\begin{samepage}
\subsection{Hyperdescent and representability}
We prove coherent split-open refinement and its
set-sized enriched localization consequence.
\label{subsec:banach-g2}
Together these assertions supply hyperdescent for the
represented functors. Their proofs are in
Appendix~\ref{app:split-open-hyperdescent}.
\end{samepage}

Let
\[
 \mathcal P_{\mathcal C}
 =\Fun(\mathcal C^{\op},\mathcal S)
\]
be the original enriched presheaf infinity-category.
All matching objects in the following refinement theorem
are augmented homotopy matching objects over the fixed
representable $h_X$, formed in $\mathcal P_{\mathcal C}$
before establishing hyperdescent of the represented functors.
Write
\[
 M_n^X(V)
 =
 h_X\times^h_{\{\partial\Delta[n],c h_X\}^h}
             \{\partial\Delta[n],V\}^h.
\]
Here $\{-,-\}^h$ denotes the homotopy weighted limit.
In particular,
\[
 M_0^X(V)=h_X,\qquad
 M_1^X(V)=V_0\times^h_{h_X}V_0.
\]

\begin{thm}[Split-open refinement and enriched localization]
Fix the small Kan-enriched presentation of $\mathcal C$
with its strict-open topology $\tau$.
\label{lem:banach-hypercover-refinement}
The following two assertions, denoted
$\mathbf{SO}(\mathcal C)$, hold.

First, for every test $X\in\mathcal C$, every
semi-representable augmented hypercover of $h_X$
in the original enriched presheaves admits
a coherent split refinement whose nondegenerate
components are represented by strict opens of $X$.
Their chosen full-component models give a strict
augmented diagram with all face and degeneracy data.

Second, choose an infinite small $\alpha$ bounding
the objects and all mapping-space simplices of the
fixed site and $\aleph_0$. Refinements can be chosen
with at most $\alpha$ total nondegenerate components
and total source data of size at most $\alpha$.
Normalized source augmentations form a set
$\mathscr S$ of at most $2^\alpha$ maps, including
$0\to h_\varnothing$, with
\[
 L_{\mathscr S}\mathbf P_{\mathrm{inj}}
       =\mathbf P_{\mathrm{inj},\tau}.
\]
Here $\mathbf P_{\mathrm{inj}}$ is the global injective
model on the original enriched presheaves, and
$\mathbf P_{\mathrm{inj},\tau}$ is its hypersheaf localization.
This is localization at all enriched semi-representable
hypercovers in the same universe. An auxiliary regular small
cardinal $\lambda_{\mathrm{SO}}>2^\alpha$ may be chosen
in the standing presheaf universe. These uniform bounds
concern augmentation sources. The sizes of target
replacements and refinement maps are constrained only
by the standing universe, and covers may be infinite.
\end{thm}

\begin{proof}
Strict opens are monomorphisms before sheafification
and form a covering basis. Their full-component
models $O_X(W)\subset h_X$ retain the entire enriched
action and compute intersections literally
(Lemma~\ref{lem:banach-open-monomorphic-models}).
For blockwise maps of indexed open models, finite
slice limits retain compatible index tuples and
their intersections, including labelled empty
intersections. These diagrams are Reedy fibrant
in the global projective slice, whose matching
computes the displayed augmented homotopy matching
(Lemma~\ref{lem:banach-open-slice-matching}).

Strictify $U$ and replace it Reedy-fibrantly in
the global injective slice over $h_X$. The
induction starts with an open cover of $X$
refining its degree-zero local lifting problem.
Given a split partial refinement, its
augmented matching components are finite intersections
of the chosen opens over $X$, with their component
indices retained. The required lifting map is
\begin{equation}
\label{eq:banach-relative-matching}
 M_n^X(V)\times^h_{M_n^X(U)}U_n
 \longrightarrow M_n^X(V).
\end{equation}
Its strict model is an injective fibration and
local cover over the constructed matching object.
Strict representable lifts on open covers extend
to the $O_X(W)$ models by the weak equivalence of
Kan slice mapping spaces in
Lemma~\ref{lem:banach-open-lift-extension}.
The Reedy latching extension retains these full
matching lifts and all target coherence, with the
degree-zero, -one and -two checks given in
Lemma~\ref{lem:banach-bounded-open-refinement}.
Lifts into $U_n$ remain choices even though
factorizations between opens are contractible.

Lemma~\ref{lem:banach-split-open-size} bounds and
normalizes only the augmentation sources.
Theorem~\ref{lem:banach-split-open-localization}
then identifies the localized model structures using
the framed augmented HAG pullback,
boundary lifting between local fibrant objects,
and the abstract model-comparison criterion. This
step turns refinement into a generating-set description
of enriched hypersheaf localization.
\end{proof}

\begin{thm}[Represented hyperdescent]
Every $h_Y$, $Y\in\mathcal G$, is hypercomplete.
\label{thm:banach-representability}
The functor
\[
 h:\mathcal G\longrightarrow\Sh(\mathcal C)^\wedge
\]
is fully faithful and preserves finite limits.
Effective strict-open descent gives ordinary-sheaf
full faithfulness, and geometric finite limits give
finite-limit preservation. Split-open refinement and
enriched localization supply hypersheaf membership.
\end{thm}

\begin{proof}
\emph{Chart targets and finite closure.}
For $Y=\operatorname{Chart}(U)$, chart evaluation
identifies $h_Y(X/a)$ with
$\Map_X(a,\mathcal O_X(U))$.
Lemma~\ref{lem:banach-internal-open-hypercover}
replaces all indexed open components by internal
subterminals and proves the internal hypercover
condition. Its natural equivalence of the entire
open-poset diagram identifies the cosimplicial
descent diagram, not just its terms. Hypercompleteness
of $\mathcal O_X(U)$, not of the underlying topos,
therefore gives descent. The generating-set clause
of Theorem~\ref{lem:banach-hypercover-refinement}
makes $h_Y$ a hypersheaf.

Restricted Yoneda preserves limits in its target.
All these functors are already ordinary sheaves
by effective structured gluing. We therefore take
subobjects in $\Sh(\mathcal C)$, whose $0$-truncated
subobject classifier $\Omega$ is hypercomplete.
For a strict open $W\to C$, its mono is classified
there and
$h_W=h_C\times_\Omega1$. Thus if $h_C$ is
hypercomplete, so is $h_W$, by finite-limit closure.
Induction over finite limits and strict opens
gives the assertion for every $C\in\mathcal C$.

\emph{Locally affine targets.} We first work in the \emph{ordinary}
topos $\Sh(\mathcal C)$, before membership in the hypercomplete
subcategory is known. For a strict-open atlas $C_i\to Y$, $h_Y$ is
already a sheaf by
Subsection~\ref{subsec:banach-g1}. The $h_{C_i}\to h_Y$ are monos
and form an effective epi cover, since pulling the atlas back along
any test map yields an open cover of the test. Their sources are
hypercomplete by the affine case, so the monomorphic-cover half of
Lemma~\ref{lem:banach-open-gluing} applies. Explicitly,
for hypercompletion $L$ and a mono $A_i\to A$,
classification by $\Omega$ makes the unit square
cartesian:
\[
 A_i\simeq A\times_{LA}LA_i.
\]
The hypercomplete atlas members have equivalence
units. Their images cover $LA$, and effective-epi
descent detects that $A\to LA$ is an equivalence.
Apply this to $A=h_Y$ and $A_i=h_{C_i}$.
The atlas index set is any permitted small set,
not necessarily bounded by the generator cardinal
$\alpha$. Thus $h_Y$ lies in $\Sh(\mathcal C)^\wedge$.

\emph{Full faithfulness.} Let $C_i\to X$ be a hyp-affine atlas with
arbitrary small index set $I$; no finite subcover is assumed. Every
intersection $C_{i_0\cdots i_n}$ is a strict open of a hyp-affine,
hence in $\mathcal C$, and effective structured descent gives
\[
 \Map_{\mathcal G}(X,Y)=
 \operatorname*{Tot}_{[n]}\ \prod_{(i_0,\dots,i_n)\in I^{n+1}}
 \Map_{\mathcal G}(C_{i_0\cdots i_n},Y).
\]
In $\Sh(\mathcal C)$ the Cech nerve of the effective cover
$\coprod_i h_{C_i}\to h_X$ has degree-$n$ term
\[
 \coprod_{i_0,\dots,i_n}h_{C_{i_0\cdots i_n}}.
\]
These coproducts need not be single affines. Mapping the nerve into $h_Y$ produces a
totalisation of \emph{products} of mapping spaces, and Yoneda
applies componentwise to each intersection
$C_{i_0\cdots i_n}\in\mathcal C$, identifying those factors with
the display above. This proves full mapping-space
faithfulness in $\Sh(\mathcal C)$, hence in its full hypercomplete
subcategory; finite-limit preservation follows from the universal
property of represented mapping spaces.
\end{proof}

The full subcategory of
hypersheaves equivalent to some $h_X$ is equivalent to
$\mathcal G$. For every finite indexing shape $K$, full
faithfulness then identifies whole coherent diagrams:
\[
 \Fun(K,\mathcal G)\simeq\Fun(K,h(\mathcal G)).
\]
Together with
Theorem~\ref{thm:banach-geometric-limits}, this supplies
the finite weighted matching and prism objects needed
for the hypercomplete groupoid construction.
The same diagram-lifting statement already holds in
the ordinary fully faithful represented image.

\subsection{A genuinely derived Banach intersection}
\label{subsec:banach-intersection}

\begin{prop}[Derived Banach self-intersection]
\label{prop:banach-derived-intersection}
For a Banach space $H$ whose whole-space domain belongs to $T$,
the finite limit
\begin{equation}
\label{eq:banach-DH}
 D_H=\operatorname{Chart}(*)
 \times^h_{\operatorname{Chart}(H)}\operatorname{Chart}(*)
\end{equation}
at the origin exists in $\mathcal C$, is universal against
$\mathrm{Str}_T^{\mathrm{vh}}$, and satisfies
\[
 \Map(F,D_H)\simeq H_{\mathrm{add,disc}},\qquad
 \Map(\operatorname{Chart}(S),D_H)\simeq *
\]
for the structured test $F$ of \eqref{eq:banach-square-zero} and
every ordinary chart with $S\in T$. For $H\neq0$ the first space
is noncontractible,
so $D_H$ is not the ordinary point despite being invisible to all
ordinary chart tests; this holds for infinite-dimensional $H$ as
well as for $H=\R$.
\end{prop}

\begin{proof}
Existence and universality are
Theorem~\ref{thm:banach-geometric-limits}(2)--(3). Chart evaluation
(Lemma~\ref{lem:banach-chart-eval}) and the coreflection
(Theorem~\ref{thm:banach-coreflection}) compute
$\Map(F,D_H)\simeq *\times^h_{F(H)}*\simeq H_{\mathrm{add,disc}}$
from \eqref{eq:banach-square-zero}: the component of
$F(H)=\coprod_{u\in H}K(H_{\mathrm{add}},1)$ selected by the origin
is $K(H_{\mathrm{add}},1)$. For an ordinary chart $S$ the structure
values are discrete, so the same computation gives
$\Map(\operatorname{Chart}(S),D_H)=*$. The test $F$ need not itself
lie in $\mathcal G$: it is an object of the larger test category
$\mathrm{Str}_T^{\mathrm{vh}}$ against which the displayed limit is
universal. This is genuinely more than adjoining the ordinary
classifying prestack $S\mapsto K(C^{\infty}(S,H),1)$ to a Banach
sheaf topos, which only adds ordinary stackiness and leaves the
self-intersection of two points over the ordinary representable $\y H$ a
point. The computation applies to every admitted $H$.
\end{proof}

\subsection{Represented Banach groupoids and the hypercomplete CFO}
\label{subsec:banach-cfo}

Define the generated smooth and cover classes on $\mathcal G$
and its represented image. Let $\mathrm{Smooth}_{\mathrm{gen}}$
be the smallest class containing equivalences, strict opens and
the projections $Y\times\operatorname{Chart}(V)\to Y$ for
$V\in T$, closed under composition and homotopy base change,
with membership local on both source and target for strict-open
covers. Explicitly, a map belongs to the class if its restrictions
to a strict-open cover of its source do; it also belongs to the
class if its base changes to a strict-open cover of its target do.
A map $f:X\to Y$ belongs to $\mathrm{Cov}_{\mathrm{gen}}$ if it
belongs to $\mathrm{Smooth}_{\mathrm{gen}}$ and there is a
strict-open cover $\{Y_i\to Y\}$ with maps $s_i:Y_i\to X$ and
homotopies $f s_i\simeq(Y_i\to Y)$. For the empty target use
the empty covering family. These covers are distinct from the
pregeometry's admissibles. The closure definition of
$\mathrm{Smooth}_{\mathrm{gen}}$ applies to derived objects through
homotopy base change and strict-open locality; ordinary
local-product charts provide examples below.

\begin{lem}[Cover interface]
\label{lem:banach-cover-interface}
$\mathrm{Cov}_{\mathrm{gen}}$ is stable under composition and
homotopy base change and is target-open local.
Its represented maps are effective epimorphisms in
$\Sh(\mathcal C)$, and their hypercompletions are
effective epimorphisms in $\Sh(\mathcal C)^\wedge$.
By Theorem~\ref{thm:banach-representability}, these
are already the hypersheaf represented maps.
\end{lem}

\begin{proof}
For $f\colon X\to Y$ and $g\colon Y\to Z$ in
$\mathrm{Cov}_{\mathrm{gen}}$ choose sections $t_j\colon Z_j\to Y$ of
$g$ over a target-open cover, and $s_i\colon Y_i\to X$ of $f$ over a
cover of $Y$. Pulling $Y_i$ back along each $t_j$, the composites
$s_i t_j$ are sections of $gf$ over the resulting opens, and $gf$
lies in $\mathrm{Smooth}_{\mathrm{gen}}$ by closure; a pullback of a
section is a section of the base-changed map, giving base-change
stability, and composing covering families gives target locality.
Finite products of cover maps are composites of base changes of
those cover maps, hence covers. For representables, the effective
open cover $\coprod_i h_{Y_i}\to h_Y$ factors through $h_X$ via the
sections, so $h_X\to h_Y$ is effective epi; hypercompletion
preserves effective epimorphisms.
\end{proof}

The target-section condition may not be weakened to source-local
smoothness: $\emptyset\to Y$ is a cover only when $Y$ is initial. An
ordinary split Banach submersion lies in
$\mathrm{Smooth}_{\mathrm{gen}}$ through its local product charts and
open gluing \cite[pp.~28--30]{Schmeding2021}; if surjective, its
target charts cover, and choosing \emph{any} point of each nonempty
fibre domain (not necessarily the origin) supplies the required
local sections. This use of projection coordinates concerns
ordinary split Banach submersions with the stipulated local charts.
A bounded linear quotient with no bounded linear right inverse
has no local smooth section, since differentiating such a section
would supply one. Such a quotient is therefore not a cover.

Theorem~\ref{thm:banach-representability} supplies
the fully faithful hypersheaf functor of points
for the groupoid construction.
The realization comparison is
Theorem~\ref{lem:derived-hypercomplete-base-change}
on this same fixed Banach test site.
Fix the
Kan-enriched local hypersheaf presentation
$\CM_{\mathrm{hyp}}$ of $\Sh(\mathcal C)^{\wedge}$ from
Lemma~\ref{lem:derived-model-dictionary}, with $\mathcal C$ as its
enriched site. Applying the represented-horn criterion,
Proposition~\ref{prop:geometric-realization-criterion}, we take the
representability data of
Definition~\ref{def:representability-data} to be
\[
 \begin{gathered}
 \CM=\CM_{\mathrm{hyp}},\qquad
 \mathcal D=\mathcal G,\\
 h\colon\mathcal G\hookrightarrow\Sh(\mathcal C)^{\wedge}
 \ \text{of Theorem~\ref{thm:banach-representability}},\qquad
 \mathsf C=\mathrm{Cov}_{\mathrm{gen}},
 \end{gathered}
\]
and the realization functor
$\widehat{\mathsf{Real}}(X)=L^{\wedge}|X|$. Here $|X|$ denotes
realization in $\Sh(\mathcal C)$ after viewing the hypercomplete
level objects in that ordinary infinity-sheaf category. The
model dictionary identifies $L^{\wedge}|X|$ with their colimit
in $\Sh(\mathcal C)^\wedge$. The weak equivalences are
$W_{\mathrm{hyp}}=\{f:\ L^{\wedge}|f|\text{ is an equivalence}\}$.
Thus $\mathcal G$ has finite limits with $h$ fully faithful and
finite-limit preserving
(Theorem~\ref{thm:banach-representability}),
and $\mathrm{Cov}_{\mathrm{gen}}$ contains equivalences and is stable
under composition and homotopy base change
(Lemma~\ref{lem:banach-cover-interface}). Let
$\mathscr R_{\mathrm{Cov}}$ (resp.\ $\mathscr R_{\mathrm{Cov},n}$)
be the strict outer Reedy-fibrant diagrams in $\CM_{\mathrm{hyp}}$
with levels weakly represented by $\mathcal G$ and positive
homotopy horns in $\mathrm{Cov}_{\mathrm{gen}}$ (resp.\ with those
horns equivalences above $n$); fibrations are Reedy model fibrations
whose positive relative homotopy horns lie in
$\mathrm{Cov}_{\mathrm{gen}}$. The geometric requirement concerns
positive relative horns; degree zero retains only the Reedy
model-fibration requirement.

\begin{thm}[Represented hypercomplete Banach CFO]
For\label{thm:banach-cfo} the fixed Kan-enriched Banach test site above,
$(\mathscr R_{\mathrm{Cov}},W_{\mathrm{hyp}},F)$
is a full Brown category of fibrant objects, without
a bound on an object's outer amplitude.
Each finite-bound category
$\mathscr R_{\mathrm{Cov},n}$ and the union over all
finite bounds,
\[
 \bigcup_{n<\infty}\mathscr R_{\mathrm{Cov},n},
\]
have the same Brown structure with the same
hypercompleted realization weak equivalences
$W_{\mathrm{hyp}}$.
\end{thm}

\begin{proof}
The geometric finite-limit and represented hyperdescent
theorems supply the representability
data for
Proposition~\ref{prop:geometric-realization-criterion}.
Lemma~\ref{lem:banach-cover-interface} supplies the
stable cover class and its effective-epimorphism
property.

Hypercompleted realization inverts levelwise model
equivalences and the constant outer-path maps by
the diagonal-realization dictionary and the
explicit outer simplicial contraction.
For a fibration, its positive relative homotopy
horns become local effective epimorphisms.
Theorem~\ref{lem:derived-hypercomplete-base-change},
specialized to $\mathcal C$, therefore identifies
the realization of the ordinary Reedy pullback
with the homotopy pullback of realizations.
These are the two realization hypotheses of the
represented-horn criterion.

The criterion gives the Brown structure.
Its finite-matching and prism arguments preserve
each finite outer bound. Every finite Brown
construction in the union over finite bounds has
a common maximum bound. The weak-equivalence
class remains $W_{\mathrm{hyp}}$ throughout.
\end{proof}

\begin{rem}[Geometry and realization]
\label{rem:banach-nonclaims}
Theorem~\ref{thm:banach-cfo} combines the geometric
finite-limit and open-gluing construction with enriched
realization and hyperdescent. Its charts are hyp-affines,
and its geometric horns use the generated-smooth,
target-open-local-section class.

The fully faithful hypersheaf functor of points represents
geometric spaces. Realization of unbounded higher-groupoid
presentations defines a further functor on the relative
localization; determining its mapping spaces and essential
image is a separate comparison problem. Every outer bound
uses hypercompleted realization.

The finite-limit construction uses the spectrum adjunction
and open-local algebraic representatives in place of the
sufficient precanonicity criterion of
\cite[Proposition~2.3.21, p.~58]{LurieDAGV2009}.
Valuewise hypercompletion acts on structure values and
keeps the underlying topoi fixed, so hyp-affines supply
the charts directly. Comparisons with dg Banach,
Fredholm, Frechet or convenient-calculus models require
functors compatible with these structures and with
realization. The Fredholm problem is formulated in
Section~\ref{sec:comparisons}.
\end{rem}

%% file: noncommutative.tex
\section{Associative noncommutative geometric groupoids}
\label{sec:noncommutative-geometry}

This section carries the represented-groupoid programme of the
preceding sections into associative noncommutative geometry.
The represented part of the development is affine: it builds
categories of fibrant objects whose objects are outer simplicial
diagrams of \emph{affine} associative pieces, and whose weak
equivalences are realization equivalences of those diagrams.
For these represented groupoids, the ambient categories
are presheaf categories on associative affines.
Pridham's geometric classes and finite effectivity
results \cite{PridhamNC2023,Pridham2013}, together with the
represented-horn criterion of
Section~\ref{sec:represented-criterion}, give the geometric
categories of fibrant objects below.

The section contains three related constructions.
\begin{itemize}
\item The \emph{opposite affine} category of fibrant objects
      (Proposition~\ref{prop:nc-affine-cfo}) uses cofibrant
      associative dg algebras with the all-degree fibration
      convention; its weak equivalences are opposite
      quasi-isomorphisms.
\item The \emph{compact split-horn prestack} category of fibrant
      objects (Theorem~\ref{thm:nc-prestack-cfo}) uses represented
      outer diagrams whose positive horn maps have homotopy
      sections.
\item The \emph{all-connective geometric} Artin and DM
      prestack categories of fibrant objects
      (Theorem~\ref{thm:nc-geometric-cfo}) use Pridham's
      homotopy-submersive and homotopy-etale classes. Their
      horn maps satisfy his geometric lifting, finiteness
      and split classical truncation conditions.
\end{itemize}
Throughout, weak algebra maps are \emph{quasi-isomorphisms}, not
Morita equivalences. The associative theory
distinguishes $k$ from a matrix algebra $M_n(k)$, and testing is
by algebra maps, not by scalar commutative points. The following
remark records why.

\begin{rem}[A scalar-versus-matrix boundary]
\label{rem:nc-scalar-matrix}
Fix a characteristic-zero field $k$ and let $E_{ij}$ be the
matrix units of $M_2(k)$. Then
\[
 [E_{12},E_{21}]^2=(E_{11}-E_{22})^2=1 .
\]
A unital map to a nonzero commutative algebra would send the
commutator to zero and its square to one, a contradiction; by
simplicity of $M_n(k)$ the same nonexistence holds for every
$n\ge2$. In particular, for $k=\R$ there is no unital graded
algebra map from degree-zero $M_2(\R)$ to
$\Omega^\bullet(M;\R)$ for a nonempty smooth manifold $M$.
These scalar tests are empty on $M_2(\R)$ and nonempty on $\R$,
although the two algebras are Morita equivalent. Thus these
scalar tests do not descend to Morita equivalence. Associative
algebra maps distinguish the two affines and are the tests
used below.
\end{rem}

\subsection{Associative affine homotopy theory}

For this subsection use \emph{unbounded homological} dg
associative unital $k$-algebras, with $k$ of characteristic zero.
The homological grading may be unbounded in both directions.
The associative operad carries the transferred model
structure on these algebras, in which weak equivalences are
quasi-isomorphisms and fibrations are surjective in every degree;
this is the associative case of the characteristic-zero transfer
theorem for operadic algebras \cite{Hinich2015}.

\begin{prop}[The opposite affine CFO]\label{prop:nc-affine-cfo}
The opposite of the full category of cofibrant associative
dg algebras is a full Brown category of fibrant objects. Its weak
equivalences are opposite quasi-isomorphisms, its fibrations are
opposite algebra cofibrations, and its localization is the
opposite associative-algebra infinity-category.
\end{prop}

\begin{proof}
This is the category of fibrant objects in the opposite model
category. Explicitly its terminal object is $\Spec k$. If
$A\to B$ is a cofibration and $C$ is cofibrant, then
$C\to B\amalg_A C$ is a cofibration; the pushout is cofibrant.
If $A\to B$ is acyclic, so is its pushout. These are precisely
the ordinary pullback and acyclic-pullback axioms in the opposite
category. Factoring the fold map
\[
 A\amalg_k A\hookrightarrow\operatorname{Cyl}(A)
                  \overset{\sim}{\twoheadrightarrow}A
\]
gives its path object, with cofibrant cylinder. The remaining
axioms and the localization follow from the model structure and
cofibrant replacement.
\end{proof}

The algebra pushout here is an amalgamated \emph{free} product,
not a commutative tensor product. Within the stated CFO the
objects are cofibrant and the pushout has a cofibration leg, so
the cofibrant gluing lemma makes that strict pushout a homotopy
pushout. For arbitrary diagrams outside these hypotheses one
instead uses homotopy pushouts or cofibrant diagram replacements.

For a connective derived-function geometry, instead use
connective associative algebras: nonnegative homological degrees,
equivalently nonpositive cohomological degrees. Let
$\mathcal A_{\ge0}$ denote their infinity-category and put
\begin{equation}\label{eq:nc-compact-affines}
 \mathcal T=(\mathcal A_{\ge0}^{\omega})^{\op},
\end{equation}
where compactness means preservation of filtered colimits by the
derived mapping-space functor.

\begin{lem}\label{lem:nc-affine-finite-limits}
After a universe choice, $\mathcal T$ has a small presentation,
finite limits, and terminal object $\Spec k$.
\end{lem}

\begin{proof}
The connective associative algebra category is presentable.
Free algebras on finite perfect connective complexes are compact,
by the free--forgetful adjunction and preservation of filtered
colimits by the forgetful functor. Free resolutions give compact
generation, so the compact objects form an essentially small
subcategory. Finite colimits of compact objects remain compact:
mapping out converts them to finite limits of spaces, which
commute with filtered colimits. Thus the opposite has finite
limits. Connective algebras are closed under these finite
homotopy colimits in the unbounded algebra infinity-category, so
the two computations agree. The initial algebra $k$ is compact
and gives the stated terminal affine object.
\end{proof}

Write $\mathscr A_{\mathrm{cof}}$ for the full category of
cofibrant unbounded unital associative dg $k$-algebras in
Proposition~\ref{prop:nc-affine-cfo}, with $W$ its
quasi-isomorphisms. That proposition gives
\[
 \mathscr A_{\mathrm{cof}}^{\op}[W^{-1}]
       \simeq(\operatorname{Alg}^{E_1}_k)^{\op}.
\]
The category in \eqref{eq:nc-compact-affines} is obtained by
selecting the compact objects of the connective full subcategory
before taking the opposite. A Kan-enriched presentation may be
selected from a simplicial localization of the algebra model.
The next construction uses finite limits in this affine
infinity-category, with geometric horn fibrations defined
through homotopy sections. The opposite affine CFO of
Proposition~\ref{prop:nc-affine-cfo} instead uses opposite
algebra cofibrations. These differ from algebra-model
fibrations: the latter are surjective in every degree
in the unbounded model, and in positive degrees in a
connective model.

\subsection{A represented noncommutative prestack CFO}

Use the presheaf infinity-topos
\[
 \mathcal X_{\mathrm{pre}}=\PSh(\mathcal T)
             =\Fun(\mathcal T^{\op},\mathcal S).
\]
Choose a small Kan-enriched presentation of $\mathcal T$ and the
global injective model of enriched simplicial presheaves. Its
weak equivalences are objectwise. Limits and colimits in
$\mathcal X_{\mathrm{pre}}$ are pointwise. Evaluation is a jointly
conservative family of geometric inverse-image functors; an
infinitely connected map is therefore an equivalence because its
evaluations are infinitely connected maps of spaces. Thus this
particular presheaf infinity-topos is hypercomplete.

\begin{lem}\label{lem:nc-split-cover}
A map of representables $\y U\to\y V$ is an effective epimorphism
in $\mathcal X_{\mathrm{pre}}$ exactly when $U\to V$ has a
homotopy section in $\mathcal T$. These maps are stable under
finite composition and homotopy pullback.
\end{lem}

\begin{proof}
A homotopy section gives a section up to homotopy on every
evaluated mapping space, hence surjectivity on $\pi_0$. Effective
epimorphisms are detected pointwise. Conversely, evaluate at $V$
and lift the identity component in $\Map(V,V)$; this is a homotopy
section. Sections up to homotopy compose and pull back along the
finite limits in $\mathcal T$.
\end{proof}

Define an ordinary category $\mathscr R_{\mathrm{nc,split}}$ of
outer simplicial enriched presheaves as follows. Its objects are
outer Reedy-fibrant diagrams with every level weakly represented
by an object of $\mathcal T$, and with every positive homotopy
horn map represented by a homotopy split epimorphism in
$\mathcal T$. Its fibrations are outer Reedy model fibrations with
the same condition on positive relative horn maps. The geometric
condition concerns positive degrees; the Reedy requirement
applies in every degree, including zero. Weak
equivalences are equivalences of outer realizations in
$\mathcal X_{\mathrm{pre}}$.

\begin{thm}[Split-horn NC prestack CFO]\label{thm:nc-prestack-cfo}
With these weak equivalences and fibrations,
$\mathscr R_{\mathrm{nc,split}}$ is a full Brown category of
fibrant objects, without a bound on an object's outer amplitude.
\end{thm}

\begin{proof}
We use finite limits and split covers in the presheaf category.
Finite homotopy weighted limits remain
representable by Lemma~\ref{lem:nc-affine-finite-limits}. For
Reedy-fibrant models, their strict weighted limits compute those
homotopy limits, by the finite-cell argument of
Lemma~\ref{lem:derived-weighted-matching}. Ordinary pullback
along a Reedy fibration is therefore levelwise a homotopy
pullback and remains weakly represented. Relative horn maps pull
back, and their composition has the relative-horn factorization
of Lemma~\ref{lem:derived-reedy-calculus}.
Lemma~\ref{lem:nc-split-cover} preserves the geometric
conditions. For $p\ge1$ and $0\le i\le p$, put
$\Lambda=\Lambda^i[p]$ and write $Q_\Lambda=\{\Lambda,Q\}$
for any diagram $Q$. With $P=X\times_Y Z$, its absolute
horn map factors as
\[
 P_p\longrightarrow
      Z_p\times_{Z_\Lambda}P_\Lambda
          \longrightarrow P_\Lambda.
\]
The first map is a base change of the relative horn of $X\to Y$,
and the second of the absolute horn of $Z$. Both split up to
homotopy, so $P$ is itself an object. The terminal presheaf is
represented by $\Spec k$, and every object is fibrant by
definition.

Use the outer path object
\[
 (PX)_p=\{\Delta[p]\times\Delta[1],X\}.
\]
The finite prism filtrations of
Lemma~\ref{lem:derived-prism-filtration} and
Proposition~\ref{prop:derived-outer-path} use only finite
pullbacks and composites of positive horn maps. They therefore
preserve the split condition and weak representability here as
well. The endpoints give a fibration with the prescribed positive
relative horns. Absolute horns of $PX$ satisfy the same factorization
over $X\times X$, whose absolute horns split by finite-product
stability. These absolute horns make $PX$ an object.
Constant paths are a realization weak equivalence by the
simplicial contraction of $\Delta[1]$.

For acyclic pullback, evaluate a fibration $f:X\to Y$ at
$T\in\mathcal T$. The relative outer horn maps are inner Kan
fibrations, by the Reedy model fibration condition. Their
represented homotopy sections make them surjective on $\pi_0$,
hence on vertices by path lifting. Lifting the anodyne inclusion
$\Delta[0]\to\Delta[q]$ then makes them surjective in every inner
degree. Together with the evaluated Reedy fibration,
vertex-surjectivity is exactly the hypothesis of
Lemma~\ref{lem:enriched-pointwise-diagonal}; its explicit
two-step horn factorization makes the actual diagonal
$f(T)$ a Kan fibration. This is the pointwise form of
the realization mechanism in
\cite[Proposition~1.19]{Pridham2013}, with the horn
factorization given in Appendix~\ref{app:enriched-realization}.
The general local comparison is
Theorem~\ref{lem:derived-hypercomplete-base-change}.
Applying the pointwise lemma separately to the
object-to-terminal maps, using their positive
absolute split horns, also makes the diagonal objects
Kan. Thus the absolute horns ensure object Kan-ness,
while the Reedy condition and relative-vertex-surjectivity
together give the diagonal fibration.

If $f$ is also a weak equivalence, its diagonal is a trivial Kan
fibration, and every ordinary pullback remains one. Diagonal
commutes with strict pullbacks. The diagonal computes outer
realization because the global injective cofibrations are
monomorphisms and outer simplicial objects are Reedy cofibrant.
Here $\mathcal T$ carries the trivial topology, and the global
injective model presents $\mathcal X_{\mathrm{pre}}$ with
objectwise weak equivalences. Sheafification and hypercompletion
in \eqref{eq:derived-diagonal-realization} are therefore identities,
and the formula reduces in the underlying infinity-category to
$\operatorname{diag}Z_\bullet\simeq
\operatorname*{colim}_{\Delta^{\op}}Z_k$. Thus the
diagonal-realization formula is evaluated in the associative
presheaf model. Evaluation is jointly conservative, proving the
acyclic-pullback axiom. The other Brown axioms follow from the
fibration calculus and two-out-of-three for
realization.
\end{proof}

\begin{rem}[Instance of the general criterion]
\label{rem:nc-split-criterion-instance}
The verification above is exactly an instance of the
representability data of
Proposition~\ref{prop:geometric-realization-criterion}: take
$\mathcal D=\mathcal T$, $h$ the derived Yoneda embedding,
$\mathsf C$ the class of homotopy split epimorphisms, and
$\mathsf{Real}$ the outer realization in
$\mathcal X_{\mathrm{pre}}$. Its first realization hypothesis is
the simplicial contraction of $\Delta[1]$, and its second is the
pointwise diagonal-fibration computation just given. The direct
proof records how both absolute and relative split horns enter
these conclusions.
\end{rem}

The cover condition in this prestack construction is precisely
the existence of a homotopy section. The Artin/DM construction below uses Pridham's formal
lifting and finiteness conditions on all connective
affines; formal smoothness and splitness play different roles.

\begin{example}[An associative affine additive group]
\label{ex:nc-additive-nerve}
Let $T_k$ be the free unital associative algebra functor. The
compact affine object $G_{\mathrm{add}}=\Spec T_k(k[0])$ is a
group object: contravariantly, addition sends $t$ to $t_1+t_2$ in
the free product, inversion sends $t$ to $-t$, and the unit sends
$t$ to zero. The free algebra adjunction identifies its derived
points with the additive underlying mapping space of a connective
dg algebra. Its outer nerve has levels
\[
 (NG_{\mathrm{add}})_p
       =\Spec T_k(k^p[0]).
\]
For $p\ge2$ these coordinate algebras have noncommuting
generators; replacing them by commutative polynomial algebras
would give a different example. For $p=1$ the one-generator
algebra is $k[t]$; noncommutativity here begins with its NC affine
products. The free-algebra cofaces and codegeneracies form a
strict cosimplicial algebra diagram, since the cogroup identities
hold on its generators. Its localized nerve is a coherent
simplicial prestack; Lemma~\ref{lem:derived-strictification},
applied in the global presheaf model, supplies an ordinary outer
model before Reedy-fibrant replacement. The one-dimensional horn
maps have the zero section, and the higher homotopy horn maps of
the groupoid object are equivalences. For example the degree-two
additive coordinate changes are $(u,v)\mapsto(u,u+v)$ and
$(u,v)\mapsto(u+v,v)$, with inverse changes using subtraction, not
inverses of free-algebra generators. The matching objects are
represented by finite derived affine limits. After Reedy-fibrant
replacement, weak representability and these homotopy conditions
persist. It is therefore an object of
$\mathscr R_{\mathrm{nc,split}}$, giving a split-horn
presentation of associative addition.
\end{example}

\subsection{A geometric subclass of the split-cover criterion}

The split-prestack argument has a useful parameterized form.
For a geometric subclass, the closure and path properties
must hold for its chosen horn class.

\begin{prop}[Represented prestack criterion]
\label{prop:nc-prestack-criterion}
Let $\mathcal T_0$ be a small infinity-category with finite
limits, and let $\mathsf C$ be an equivalence-invariant class of
maps containing all equivalences, stable under composition and
homotopy pullback. Suppose every map in $\mathsf C$ has a homotopy
section. Equivalently, its derived Yoneda map is an effective
epimorphism, or its evaluation on $\Map(T,-)$ is surjective on
$\pi_0$ for every $T\in\mathcal T_0$, as in
Lemma~\ref{lem:nc-split-cover}.

In the global injective enriched-presheaf model for
$\PSh(\mathcal T_0)$, let $\mathscr R_{\mathsf C}$ consist of
outer Reedy-fibrant diagrams with weakly represented levels and
positive absolute homotopy horn maps in $\mathsf C$. Equivalence
invariance makes this condition independent of the chosen
representing objects and maps. Take fibrations to be Reedy model
fibrations with positive relative homotopy horns in $\mathsf C$,
and weak equivalences to be objectwise diagonal equivalences. Then
$\mathscr R_{\mathsf C}$ is a full Brown category of fibrant
objects. Its geometric fibration conditions are confined
to positive degrees.
\end{prop}

\begin{proof}
The proof of Theorem~\ref{thm:nc-prestack-cfo} uses exactly these
geometric hypotheses. Finite homotopy weighted limits remain
representable by Lemma~\ref{lem:derived-weighted-matching}.
Relative horns of pullbacks and composites remain in $\mathsf C$
by the relative-horn factorization of
Lemma~\ref{lem:derived-reedy-calculus}. For $m\ge1$, choose
$\Lambda=\Lambda^i[m]$. For a pullback $P=X\times_Y Z$,
its absolute horn factors as
\[
 P_m\longrightarrow Z_m\times_{Z_\Lambda}P_\Lambda
       \longrightarrow P_\Lambda,
 \qquad P_\Lambda=\{\Lambda,P\}.
\]
The two maps are base changes of a relative horn of $X\to Y$ and
an absolute horn of $Z$, respectively. Thus the pullback is an
object of the specified category.

The finite outer-prism filtrations of
Lemma~\ref{lem:derived-prism-filtration} use only pullbacks and
composites of positive horns. As in
Proposition~\ref{prop:derived-outer-path}, they give the endpoint
fibration for $PX$ and, by the same absolute factorization over
$X\times X$, its object membership. Constant paths are diagonal
weak equivalences. For acyclic pullback, a represented
$\mathsf C$-map evaluates to a map surjective on components. The
Reedy condition makes the relative horn maps inner Kan fibrations,
hence surjective on vertices and then on every simplex degree.
Together with the evaluated vertical Reedy fibration this is the
pointwise hypothesis of
Lemma~\ref{lem:enriched-pointwise-diagonal}.
The objects' separate absolute $\mathsf C$-horns
give their Kan diagonals by its terminal-map case.
A weak equivalence among the resulting diagonal
fibrations is a trivial Kan fibration, preserved by
strict pullback. This gives the acyclic-pullback axiom. Terminal
objects, composition and two-out-of-three are as in that proof.
All limits and diagonal weak equivalences here are computed
in the presheaf model.
\end{proof}

For the Artin/DM classes below, Pridham's finite effectivity
theorem supplies homotopy sections in
Lemma~\ref{lem:nc-geometric-sections}. The pointwise
diagonal lemma then applies to those sections.

This proposition specializes the realization
hypothesis of
Proposition~\ref{prop:geometric-realization-criterion}: geometric
covers have homotopy sections, giving pointwise surjectivity and a
pointwise diagonal-fibration calculation. The construction is
compatible with the represented-hypergroupoid method
\cite{Pridham2013}, with homotopy sections supplying
the required pointwise cover condition.

\subsection{Pridham's geometric classes on all NC affines}
\label{subsec:nc-geometric-classes}

Pridham develops geometric NC \emph{prestacks} using
homotopy-submersive and homotopy-etale horn conditions
\cite[Introduction and Definitions~1.19--1.23]{PridhamNC2023}.
We recall his geometric classes on all connective affines.

Fix the characteristic-zero field $k$. Choose compatible universes
$\mathcal U\in\mathcal V$ and a $\mathcal V$-small presentation of
\[
 \mathcal T_{\mathrm{all}}
   =\bigl(\operatorname{Alg}_{k,\ge0}^{E_1,\mathcal U}\bigr)^{\op},
 \qquad
 \mathcal P_{\mathrm{all}}=\PSh(\mathcal T_{\mathrm{all}}).
\]
Here all $\mathcal U$-small connective unital associative algebras
are allowed, not merely the compact objects of
\eqref{eq:nc-compact-affines}. Finite affine homotopy limits are
opposite derived amalgamated free products and remain in this
universe. Their mapping spaces are derived algebra mapping spaces;
weak equivalences of algebras are quasi-isomorphisms, not Morita
equivalences.

Pridham's homotopy-presheaf model localizes ordinary
simplicial presheaves at algebra quasi-isomorphisms. Its fibrant
functors are objectwise fibrant and invert those maps
\cite[Definition~1.19 and the following paragraph,
p.~9]{PridhamNC2023}. By the universal property of simplicial
localization, its underlying infinity-category is
$\Fun(\mathcal T_{\mathrm{all}}^{\op},\mathcal S)$. Thus its
associated prestacks are interpreted in the plain presheaf
category $\mathcal P_{\mathrm{all}}$. The corresponding
ambient model is the global enriched-presheaf model on
$\mathcal T_{\mathrm{all}}$.

Recall the following definitions
\cite[Definitions~1.14, 1.16 and 2.31, pp.~7--8 and
23]{PridhamNC2023}. For an algebra $A$ put $A_R=\Map(A,R)$. A map
$A\to B$ is \emph{homotopy locally of finite presentation} if, for
every filtered system of \emph{unbounded} test algebras with
colimit $R$, the canonical comparison
\begin{equation}\label{eq:nc-hlfp}
 \operatorname*{colim}_i B_{R_i}
 \longrightarrow
 \left(\operatorname*{colim}_i A_{R_i}\right)
       \times^h_{A_R}B_R
\end{equation}
is an equivalence. Here a \emph{strict nilpotent surjection}
$R\to S$ means a map of connective dg associative algebras
which is surjective in every homological degree, including
degree zero, and whose actual two-sided dg kernel $I$
satisfies $I^N=0$ for some finite $N\ge1$ depending on the
map. The nilpotence exponent and homological amplitude
may vary with the test map
\cite[Definition~1.16, p.~8]{PridhamNC2023}.
The word ``strict'' distinguishes this literal surjection
from the weaker positive-degree surjectivity condition for
a connective-model fibration. For such a surjection, consider
\begin{equation}\label{eq:nc-formal-lifting}
 \lambda_{B/A}(R,S):
 B_R\longrightarrow B_S\times^h_{A_S}A_R.
\end{equation}
Homotopy formal submersiveness requires surjectivity on $\pi_0$ in
\eqref{eq:nc-formal-lifting}; homotopy formal etaleness requires an
equivalence. Adding \eqref{eq:nc-hlfp} gives homotopy
submersiveness and homotopy etaleness, respectively. The
mapping-space comparison of \eqref{eq:nc-hlfp} includes all
homotopy groups, rather than only $\pi_0$. Its filtered tests
are unbounded, whereas the formal-lifting
tests in \eqref{eq:nc-formal-lifting} are connective.

Let $\mathsf C_{\mathrm{Art}}$ consist of affine maps
$\Spec B\to\Spec A$ for which $A\to B$ is homotopy submersive
and there is an algebra retraction $H_0(B)\to H_0(A)$ of
$H_0(A)\to H_0(B)$. Let $\mathsf C_{\mathrm{DM}}$ use homotopy
etaleness instead. The retraction is the split
\emph{classical affine} truncation condition in Pridham's
definition. It is not
surjectivity of the algebra map on $H_0$, and does not by itself
supply a derived section. For example, the augmentation
$T_k(k[1])\to k$, with zero differential and homological generator
degree one, is an isomorphism on $H_0$. It has no homotopy left
inverse, since it kills the nonzero $H_1$ of its source.

\begin{lem}[Stability of the actual geometric classes]
\label{lem:nc-geometric-stability}
The classes $\mathsf C_{\mathrm{Art}}$ and
$\mathsf C_{\mathrm{DM}}$ are invariant under equivalence, contain
all equivalences, and are stable under composition and arbitrary
derived affine base change. Moreover
$\mathsf C_{\mathrm{DM}}\subseteq\mathsf C_{\mathrm{Art}}$.
\end{lem}

\begin{proof}
For a fixed filtered system of unbounded tests, write
$\overline A=\operatorname*{colim}_i A_{R_i}$, and similarly for
other algebras. For homotopy l.f.p.\ maps $A\to B\to C$, the
actual comparison for their composite is identified with
\[
 \overline C
   \simeq\overline B\times^h_{B_R}C_R
   \simeq
     (\overline A\times^h_{A_R}B_R)\times^h_{B_R}C_R
   \simeq\overline A\times^h_{A_R}C_R.
\]
For base change put $B'=B*^{\mathbf L}_A A'$. The inclusion of the
connective associative model into the unbounded model is left
Quillen \cite[Definition~1.2, p.~4]{PridhamNC2023}. It preserves
this derived pushout, so the pushout mapping property holds also
against the unbounded tests in \eqref{eq:nc-hlfp}. Filtered
colimits of spaces commute with finite homotopy limits, giving
\[
 \overline {B'}
   \simeq\overline B\times^h_{\overline A}\overline {A'}
   \simeq\overline {A'}\times^h_{A_R}B_R
   \simeq\overline {A'}\times^h_{A'_R}B'_R.
\]
These are identifications of the canonical maps in
\eqref{eq:nc-hlfp}, so they establish its full mapping-space
condition, including the higher homotopy data.

For a nilpotent connective test $R\to S$, the lifting map of a
composite factors as
\[
 C_R\longrightarrow C_S\times^h_{B_S}B_R
       \longrightarrow C_S\times^h_{A_S}A_R.
\]
The first map is $\lambda_{C/B}$, and the second is a homotopy
pullback of $\lambda_{B/A}$. The lifting map after $A\to A'$ base
change is likewise a homotopy pullback of
\eqref{eq:nc-formal-lifting}, by the mapping property of
$B*^{\mathbf L}_A A'$. Both equivalences and surjectivity on
components are stable under these homotopy pullbacks and
compositions. For the latter pullback assertion, if $X\to Y$ is
surjective on components and $z\in Z$ maps to $Y$, choose $x\in X$
and a path from its image to that of $z$. The resulting triple is
a point of $X\times_Y^h Z$ over $z$. This is a homotopy-pullback
assertion, not a strict-pullback assertion for arbitrary maps of
spaces. This proves the two formal assertions.

Finally, classical retractions compose. Choose a cofibrant
connective span $A_c\to B_c$, $A_c\to A'_c$ representing the given
span, with both legs cofibrations. Its strict pushout computes
$B*^{\mathbf L}_A A'$. On connective dg algebras, $H_0$ is left
adjoint to inclusion of discrete associative algebras and
therefore preserves this strict pushout. The replacements are
quasi-isomorphisms, giving the natural isomorphism
\[
 H_0(B*^{\mathbf L}_A A')
       \cong H_0(B)*_{H_0(A)}H_0(A').
\]
A retraction $H_0(B)\to H_0(A)$, followed by
$H_0(A)\to H_0(A')$, together with the identity of $H_0(A')$,
defines a retraction of the base-changed map by this
amalgamated-free-product universal property. This computation
takes $H_0$ after forming the cofibrant pushout. Equivalence invariance
and inclusion of equivalences follow directly from the
definitions. Formal etaleness implies formal submersiveness,
proving the last assertion.
\end{proof}

\begin{lem}[Derived sections from the finite Cech theorem]
\label{lem:nc-geometric-sections}
Every map in $\mathsf C_{\mathrm{Art}}$, and hence in
$\mathsf C_{\mathrm{DM}}$, has a homotopy section in
$\mathcal T_{\mathrm{all}}$. Its derived Yoneda map is a pointwise
effective epimorphism.
\end{lem}

\begin{proof}
Let $u:U=\Spec B\to V=\Spec A$ belong to
$\mathsf C_{\mathrm{Art}}$. First represent it by a fibration
between fibrant objects of the opposite connective algebra model:
take $A$ cofibrant and replace $A\to B$ by a cofibration with
cofibrant target. Form the strict Cech nerve of this
representative. Its pullbacks compute the derived ones, so it
represents
\[
 X_m=U^{\times^h_V(m+1)},\qquad X_\bullet\longrightarrow cV.
\]
All levels are affine. This is the relative homotopy
$0$-coskeleton of $u$. For a finite simplicial set $L$, right Kan
extension from outer degree zero gives
\begin{equation}\label{eq:nc-cech-weighted}
 \{L,X_\bullet\}\simeq
 U^{L_0}\times^h_{V^{L_0}}\{L,cV\}.
\end{equation}
Here the powers indexed by $L_0$ are finite products. For
$L=\partial\Delta[m]$ and $m\ge1$, all $m+1$ vertices are present.
Pulling \eqref{eq:nc-cech-weighted} back along $V\to\{L,cV\}$
therefore gives $X_m$. In degree one this is explicitly
$V\times^h_{V\times V}(U\times U)=U\times^h_V U$. For higher
boundaries we do not replace $\{L,cV\}$ by $V$: the relative
pullback, not connectedness of $L$, gives the asserted
equivalence. Its relative boundary map in degree zero is $u$; the
relative boundary maps in all degrees $m\ge1$ are equivalences.
Thus the augmentation is a trivial derived homotopy NC Artin
$1$-hypergroupoid, with Pridham's boundary convention
\cite[Definitions~3.1--3.2, p.~16]{Pridham2013}. Its absolute
degree-one horns are base changes of $u$. Higher horns contain all
vertices and are contractible, so \eqref{eq:nc-cech-weighted}
makes their absolute maps equivalences. The constant target is a
homotopy $0$-hypergroupoid.

Choose a Reedy fibrant replacement $cV\overset{\sim}{\to}Y'$ in
simplicial objects of the affine model category. Factor
$X_\bullet\to Y'$ as a Reedy trivial cofibration
$X_\bullet\overset{\sim}{\to}X'_\bullet$ followed by a Reedy
fibration $X'_\bullet\to Y'$. The replacement levels remain
affine. Homotopy matching data are unchanged by these levelwise
weak equivalences, by their definition through Reedy replacement
\cite[Definition~2.30, p.~23]{PridhamNC2023}; the geometric classes
are equivalence-invariant. Both diagrams are now Reedy fibrant and
the arrow is a Reedy fibration. Reedy replacement is important
even for a constant diagram, whose first matching map is the diagonal
$V\to V\times V$.

Pridham's finite effectivity result now applies:
trivial finite derived homotopy NC Artin hypergroupoids give
equivalences of associated prestacks on \emph{all} connective dg
tests \cite[Remarks~2.34, pp.~24--25]{PridhamNC2023}.
The passage from bounded-homology to all connective tests uses
the nilpotent-test fibration property along the connective
Postnikov tower. Consequently $|X_\bullet|\to\y V$ is an equivalence.

Evaluation at a connective algebra $R$ gives the Cech nerve of
$B_R\to A_R$, since derived Yoneda preserves the finite affine
homotopy pullbacks and evaluation preserves colimits in a presheaf
infinity-category. Its realization is the $(-1)$-truncated image of
that map of spaces \cite[Proposition~6.2.3.4]{LurieHTT2009}. The
augmentation being an equivalence therefore makes $B_R\to A_R$ an
effective epimorphism, equivalently a map surjective on components
\cite[Corollary~7.2.1.15]{LurieHTT2009}. At $R=A$, lift the
component of $1_A$ to obtain $r:B\to A$ with
$r\circ(A\to B)\simeq1_A$. This is the required homotopy section.
It is specified up to homotopy rather than as a strict dg
algebra retraction.
\end{proof}

\subsection{Represented NC Artin and DM categories of fibrant objects}

\begin{thm}[Geometric NC prestack CFOs]
\label{thm:nc-geometric-cfo}
In the global injective model for $\mathcal P_{\mathrm{all}}$, take
ordinary outer Reedy-fibrant diagrams with weakly affine levels and
positive absolute homotopy horns in $\mathsf C_{\mathrm{Art}}$,
respectively $\mathsf C_{\mathrm{DM}}$. With objectwise
diagonal/realization weak equivalences and Reedy model fibrations
satisfying the corresponding positive relative geometric horns,
they form full Brown categories of fibrant objects
\[
 \mathscr R_{\mathrm{nc,Art}},\qquad
 \mathscr R_{\mathrm{nc,DM}}.
\]
Outer amplitude is unrestricted, and geometric horn
conditions are imposed only in positive degrees.
\end{thm}

\begin{proof}
The affine infinity-category has the required finite limits.
Lemmas~\ref{lem:nc-geometric-stability} and
\ref{lem:nc-geometric-sections} supply exactly the geometric
hypotheses of Proposition~\ref{prop:nc-prestack-criterion}. Apply
that proposition to the two classes.

The objects here are ordinary presheaf models. Derived Yoneda is
fully faithful and preserves the finite limits used by matching, so
a coherent affine homotopy diagram gives the corresponding
coherent represented presheaf diagram. Conversely a coherent
diagram with representable values factors through this full Yoneda
image. Strictification in the global presheaf model, followed by
outer Reedy-fibrant replacement, supplies ordinary models of these
diagrams and preserves their homotopy horn conditions as in
Lemma~\ref{lem:derived-strictification}. This comparison
localizes strict diagrams at levelwise equivalences.
\end{proof}

\begin{cor}[Finite outer bounds]\label{cor:nc-geometric-finite}
For each finite $n\ge0$, impose in either class that the absolute
horns in outer degrees $m>n$ are equivalences. The resulting full
subcategories, and their unions over finite $n$, are also full
Brown categories of fibrant objects.
\end{cor}

\begin{proof}
Above $n$, a map between two such objects automatically has
equivalent relative horns: both absolute horn maps are
equivalences, and their relative comparison is obtained by
homotopy pullback and two-out-of-three. The absolute pullback
factorization of Lemma~\ref{lem:derived-reedy-calculus} therefore
preserves the bound. For paths, the prism filtration of
Lemma~\ref{lem:derived-prism-filtration} for a degree-$m$ horn
uses cells in dimensions $m$ and $m+1$. If $m>n$, all
corresponding horn maps are equivalences. Thus paths also preserve
the bound, as in Proposition~\ref{prop:derived-outer-path}. In
the union over all finite bounds, each Brown construction involves
only finitely many objects; take the maximum of their bounds.
\end{proof}

In particular, Pridham's strict finite NC hypergroupoids give
objects of these represented models, by derived Yoneda and the
preceding strictification. Their realizations are the
strongly quasi-compact geometric prestacks of
\cite[Definitions~2.30--2.33]{PridhamNC2023}. The finite
CFO construction appears in
\cite[Remark~3.28, p.~21]{Pridham2013}.
Three constructions remain distinct: these finite geometric
realizations, Pridham's varying-finite-degree filtered-colimit
``infinity-geometric'' objects, and the category allowing
individually unbounded presentations in
Theorem~\ref{thm:nc-geometric-cfo}.
Determining the essential image and mapping spaces of realization
for the last category is a further comparison problem.
The weak maps used here are defined as realization equivalences;
refinements carry additional geometric conditions.

For the following external example, specialize to $k=\R$.
Let $X$ be an ordinary real manifold and let $\Omega^\bullet(X)$
be its ordinary commutative de Rham algebra. In
\cite[Example~3.15, pp.~34--35]{PridhamNC2023}, Pridham considers
\[
 B\longmapsto D_*\mathsf{Proj}
          \bigl(\Omega^\bullet(X)\otimes_\R B\bigr)
\]
for connective unital associative real dg algebras $B$, together
with the perfect-complex variant. Here $\mathsf{Proj}$ is the
finite-projective-module moduli functor used by Pridham and $D_*$
its stacky extension.
The stacky extension keeps the cohomological de Rham degree
separate from the homological degree of $B$.
This gives a moduli prestack over connective
associative affines. A version in a noncommutative sheaf
geometry would require a specified topology and descent
theorem; a Morita-invariant version would require compatible
weak equivalences and a comparison of the algebraic mapping spaces.

%% file: comparisons.tex
\section{Comparisons and further applications}
\label{sec:comparisons}

Comparisons between these homotopy theories involve functors
of points, choices of realization, and changes of geometry.
We summarize the comparisons furnished by the preceding
constructions and formulate further geometric applications.

\subsection{Functors of points and realization}

Levelwise ordinary smooth Yoneda preserves the specified
classes and existing categorical fibration pullbacks
(Proposition~\ref{prop:geometric-yoneda}).
The constant-test and represented-pullback arguments give
exactness for the iCFO of Theorem~\ref{thm:geometric-icfo}.
This exactness concerns the underlying categories;
localized mapping spaces and essential images form
a further comparison problem.

The ordinary split-Banach iCFO and its finite unique-horn
versions use the fixed-site chart compatibility of
Theorem~\ref{thm:ordinary-banach-icfo}. The ordinary
Banach sheaf CFO instead applies to any plot site with
the stated opens and products, with germ tests taken
on that fixed site. The point-only-site counterexample
shows why the geometric theorem needs additional
chart compatibility.

Internal Kan replacement identifies the localization of
locally Kan sheaves with the hypercomplete local homotopy
theory (Theorem~\ref{thm:smooth-localization}).
This is a sheaf-theoretic equivalence: Kan replacement
supplies local fillers, whereas geometric representability
and hyperdescent-fibrant replacement require further
constructions.

Using Nuiten's finite realization theorem,
Theorem~\ref{thm:derived-reedy-finite} identifies the
localization with the corresponding finite-geometric stack
category. Theorem~\ref{lem:derived-hypercomplete-base-change}
proves $\mathbf{PB}_\tau(\mathcal T)$ and gives the
unbounded hypercomplete Brown and base-change results.
Theorem~\ref{lem:banach-hypercover-refinement} proves
both clauses of $\mathbf{SO}(\mathcal C)$; together
with PB on the same fixed site, they give Banach
hypersheaf representation and the full represented CFO.
Every Banach outer bound retains the same hypercompleted
weak-equivalence class. Ordinary $\mathbf{AP}_\infty$
is the question posed in
Question~\ref{question:derived-unbounded-ordinary}.
The mapping spaces and geometric essential image of
unbounded realization are further localization questions.

For the associative models, strictification compares strict
and coherent diagrams through levelwise equivalences.
The pointwise Brown deductions use
Lemma~\ref{lem:enriched-pointwise-diagonal} with the
Reedy and relative-vertex-surjectivity hypotheses.
Their separate absolute horns give Kan diagonal
objects. The Artin/DM applications use Pridham's
finite effectivity theorem to obtain the derived
sections needed for this pointwise argument.
Weak maps are realization equivalences; describing
them by refinements is a further geometric comparison.
Affine tests remain localized at quasi-isomorphisms.

\subsection{Geometric tests}

The examples expose different failure modes.
The crossing axes show that ordinary manifold
representability can obstruct a full Brown
structure even when the horn conditions hold.
For a smaller chart class this conclusion requires
admission of the crossing-axes cospan.
The separate cubic/product example shows that insufficient
ordinary plot tests can make an acyclic geometric fibration
whose mandatory categorical pullback does not exist.
The cubic is not weak on the full Euclidean site.
The derived-smooth self-intersection
$\R\otimes_{\R[\epsilon_{-1}]}^{\mathbf L}\R$
has a degree-$-2$ generator, showing why a
quasi-smooth restriction is not the required
finite-limit closure.

For a Banach space $H$ admitted by the chosen chart category,
the self-intersection $D_H$ is distinguished from the ordinary
point by the structured square-zero test $F$, although ordinary
chart tests do not distinguish them. This includes
infinite-dimensional choices of $H$. The conclusion is
nonordinariness. Finite-dimensional derived presentations,
and the tangent and cotangent objects of this intersection,
require separate analysis.

Finally, associative affine products use
derived amalgamated free products, not
commutative tensor products. The explicit
associative affine group and geometric
cover calculations in
Section~\ref{sec:noncommutative-geometry}
keep the noncommutative test category visible.
These examples isolate the representability and testing
issues that further moduli or deformation applications
must address.

\subsection{Changes of geometry}

\begin{construction}[Compatibility of geometric realizations]
\label{problem:change-of-geometry}
For a specified change of geometry, construct
the induced comparison of represented higher
groupoids and identify the hypotheses under
which it preserves the chosen weak equivalences.
Determine its effect on mapping spaces and
geometric essential images.
\end{construction}

Finite-limit and cover preservation are
important inputs, but the realization
compatibility must also be supplied.
In particular an inclusion of classical
charts does not establish an equivalence
of complete derived geometries.
A comparison of the finite-dimensional
derived-smooth and Banach constructions
must retain their local structures
and their distinct realization functors.
A commutative/associative comparison must
retain the algebraic mapping spaces,
rather than replace them by scalar points.

Ordinary algebraic base change is a separate issue.
For finite-dimensional Hausdorff second-countable manifolds
without boundary, the smooth-function pullback of a nonempty
submersion with positive-dimensional fibres and target
dimension at least three is nonflat as a map of ordinary rings
\cite[Corollary~7.7]{ZengSubmersion2026}.
This obstruction concerns ordinary algebraic tensor products.
The structured derived pullbacks above and completed
topological tensor products use different constructions;
applications passing between them must specify the
extension-of-scalars functor and its coefficient class.

\subsection{Fredholm and noncommutative applications}

\begin{construction}[Derived Fredholm local models]
\label{problem:fredholm}
Specify a class of smooth Fredholm sections
of Banach bundles whose derived zero loci
belong to the constructed geometry.
Establish finite-dimensional local
reduction, independence of reduction data,
and compatibility with the chosen
geometric higher-groupoid presentations.
\end{construction}

Here Fredholm means that the relevant bounded
linearized operator has closed range and
finite-dimensional kernel and cokernel.
A comparison of these derived zero loci with
finite-dimensional obstruction models would connect
the geometry exemplified by $D_H$ to Fredholm
moduli problems. Such a comparison should specify
its relation to Kuranishi, polyfold or
convenient-calculus models.

The categorical Penrose--Ward comparison of Block and Zeng
\cite[Theorem~1.1]{BlockZengWard2026} supplies a concrete test problem.
On a closed connected oriented spin self-dual four-manifold,
it compares the ordered-pair deformation complexes of
self-dual bundles with Dolbeault complexes on twistor space.
An application to the geometry here would require compatible
Banach completions, Fredholm local presentations, and a
comparison of the resulting derived moduli objects.

\begin{construction}[An associative geometric comparison]
\label{problem:nc-comparison}
For a specified class of associative
moduli or deformation prestacks, compute
the realization and mapping spaces of
their geometric hypergroupoid models.
Alternatively, prove a refinement or
essential-image theorem beyond the
finite geometric results cited above.
\end{construction}

Any Morita-invariant extension needs its
own weak-equivalence class and comparison
argument. The smooth-kernel example of
\cite[Theorem~5.4]{SmoothKernelKTheory2026} makes this distinction
concrete: the full tangent foliation on the circle and
the point foliation are connected-fibre Morita equivalent,
but the raw integral nonconnective algebraic $K$-theory
spectra of their AS convolution algebras differ in
degree minus two. Geometric foliation-Morita equivalence
and the quasi-isomorphism-based algebra tests here
therefore address different invariance questions.

A noncommutative Banach theory
would additionally require analytic
algebra objects, specified completed operations,
and compatible descent and weak-equivalence data.

The transport-compatible chart quotients for locally finitely
geometrically resolvable singular foliations form simplicial
groupoid objects in smooth stacks, with outer horn equivalences
in degrees at least two \cite[Theorem~E]{ZengHigherHolonomy}.
Their relation levels provide a concrete representability problem
for the smooth and Banach geometries above. Applications to
coefficient integration, Riemann--Hilbert correspondences and
Tannakian duality further require functors relating these
geometric realizations to algebroids and their representations.

%% file: enriched-realization.tex
\section{Enriched realization base change}
\label{app:enriched-realization}

This appendix proves the two lifting mechanisms used in
Theorem~\ref{lem:derived-hypercomplete-base-change}, including
the pointwise lemma needed by the associative applications.
Fix a small Kan-enriched site $(\CT,\tau)$ and write
\[
 \mathbf P=\Fun_{\sSet}(\CT^\op,\sSet),\qquad
 h_t(s)=\Map_{\CT}(s,t).
\]
The topology is on $\operatorname{Ho}(\CT)$; presheaves and
their actions remain enriched. Use the global injective model
or its ordinary-sheaf or hypersheaf localization.
Localized fibrations are global injective fibrations, since
global trivial cofibrations remain locally trivial.
Global injective fibrations are objectwise Kan fibrations:
the maps $h_t\otimes\Lambda^i[q]\to h_t\otimes\Delta[q]$
are objectwise trivial cofibrations, and enriched Yoneda
detects the corresponding lifts.

A square $h_t\otimes K\to E$, $h_t\otimes L\to D$ for
$K\hookrightarrow L$ has \emph{local strict lifts} if a
covering sieve of $t$ makes its restriction along every
actual arrow representing a class in that sieve strictly
liftable. Fillers need not agree, and each finite problem
may use its own sieve.

\subsection{Relative matching and diagonal lifting}

\begin{lem}[Relative matching and strict local lifts]
Let $f:X\to Y$ be a Reedy model fibration between
Reedy-fibrant strict outer diagrams. For a finite inclusion
$A\hookrightarrow B$, the map
\[
 \{B,X\}\longrightarrow
 \{A,X\}\times_{\{A,Y\}}\{B,Y\}
\]
is a model fibration and its target computes the indicated
homotopy relative target.
\label{lem:enriched-relative-matching}
If a positive relative homotopy horn of $f$ is locally
surjective on components, every inner simplex of its
actual relative horn target has local strict lifts.
\end{lem}

\begin{proof}
Attach the finitely many nondegenerate simplices of $B$
relative to $A$ along their boundaries. Each extension
map is a pullback of a Reedy matching fibration, so their
finite composite is a fibration. Applying the same
induction to a Reedy-fibrant object's terminal map shows
that finite weighted limits are fibrant and compute
homotopy weighted limits. In particular
$Y_n\to\{\Lambda^i[n],Y\}$ is a fibration between
fibrant objects, so the strict target
\[
 H_{n,i}(f)=
 \{\Lambda^i[n],X\}
 \times_{\{\Lambda^i[n],Y\}}Y_n
\]
represents precisely the specified homotopy horn target.
The strict map $q:X_n\to H_{n,i}(f)$ is induced by the
same face restrictions and $f_n$ as the homotopy
horn comparison.

These are also global fibrations and are objectwise
Kan fibrations between Kan complexes. For an inner
simplex $b:\Delta[q]\to H_{n,i}(f)(t)$, local component
surjectivity gives a covering sieve on which a vertex
$b_0$ lifts up to homotopy
\cite[Proposition~3.1.4 and Remark~3.1.5,
PDF leaf~19]{HAGI2005}. At each actual restriction
$v:s\to t$, path lifting corrects the chosen vertex
to lie exactly over $v^*b_0$. Lifting the anodyne
inclusion $\{0\}\hookrightarrow\Delta[q]$ then lifts
the whole simplex, with all source horn data fixed.
Changing the representative of the test-arrow class
gives a homotopy by the enriched action, and the same
path correction supplies a strict lift for that
representative as well.
\end{proof}

\begin{lem}[Pointwise diagonal fibration]
Let $f:X\to Y$ be a Reedy fibration of bisimplicial sets,
with the first simplicial direction outer.
\label{lem:enriched-pointwise-diagonal}
If every positive horizontal relative horn map
\[
 X_n\longrightarrow
 \{\Lambda^i[n],X\}
 \times_{\{\Lambda^i[n],Y\}}Y_n
\]
is surjective on inner vertices, then
$\operatorname{diag}f$ is a Kan fibration.
The assertion also applies objectwise to enriched
presheaves. For an object, the terminal-map case requires
the corresponding positive absolute horn conditions
and yields a Kan diagonal.
\end{lem}

\begin{proof}
Let $\delta_!$ be left adjoint to the diagonal.
For a diagonal horn in degree $n\ge1$, let $F_j$
be the face of $\Delta[n]$ omitting $j$. Adjunction on
representables and preservation of their face-intersection
colimits give
\[
 D=\delta_!\Delta[n]=\Delta[n]\boxtimes\Delta[n],
 \qquad
 L=\delta_!\Lambda^i[n]=\bigcup_{j\ne i}F_j\boxtimes F_j.
\]
Put $A=L\cup(\Delta[n]\boxtimes\{i\})$. Since
\[
 L\cap(\Delta[n]\boxtimes\{i\})
       =\Lambda^i[n]\boxtimes\{i\},
\]
$L\to A$ is one horizontal horn attachment at inner
vertex $i$, and the assumed vertex-surjectivity lifts it.

Over an outer simplex $\alpha:[r]\to[n]$ with image $S$,
the inner component of $A$ is
\[
 C_S=\{i\}\cup
       \bigcup_{\substack{j\notin S\\j\ne i}}F_j
       \ \subseteq\ \Delta[n].
\]
Adjoining $i$ to any simplex stays in $C_S$; hence it
is a nonempty cone at $i$, including the case
$C_S=\{i\}$. Its inclusion is a monomorphic weak
equivalence. The relative outer latching map supplies
every entire degenerate target component; on a
nondegenerate component it is $C_S\hookrightarrow\Delta[n]$.
Indeed, if a source simplex is degenerate in the target,
a face recovers its antecedent in the source. Thus the
relative latching object is exactly the union of the
source with the target's degenerate simplices.
It follows that $A\to D$ is a Reedy trivial cofibration.
The Reedy fibration $f$ lifts this second step.
Adjunction gives the required diagonal filler.

The description includes all outer degeneracies: their
images $S$ do not change, and for $r>n$ every target
component is degenerate. For $n=1$, the two choices of
$i$ give respectively the vertices $(0,0)$ and $(1,1)$
and the horizontal intervals at inner endpoints $0$
and $1$, so both endpoint horns are included.

The same factorization can be tensored with an enriched
representable. Writing $(h_t\odot B)_r=h_t\otimes B_{r,*}$,
evaluation multiplies each inner component by $h_t(s)$.
Products of simplicial sets preserve these monomorphisms
and weak equivalences, and latching colimits are
objectwise. Thus $h_t\odot A\to h_t\odot D$ is a
global injective Reedy trivial cofibration.
\end{proof}

\begin{lem}[Enriched finite local anodyne lifting]
Under the Reedy and local-component horn hypotheses of
Theorem~\ref{lem:derived-hypercomplete-base-change},
$\operatorname{diag}f$ has local strict lifts against
every finite anodyne inclusion.
\label{lem:enriched-diagonal-local-lifting}
\end{lem}

\begin{proof}
For a single diagonal horn, use the factorization
$L\subset A\subset D$ just proved.
Lemma~\ref{lem:enriched-relative-matching} lifts the
first horizontal horn locally. The second step lifts
against the global Reedy fibration. This constructs
the filler directly, before testing Kan-ness of
either diagonal.

We must retain homotopies of test arrows when composing
covering sieves. The functor $\delta_!$ preserves
monomorphisms: it takes a boundary inclusion to the
union of the corresponding external-square faces, and
every monomorphism is built by boundary attachments.
Consequently $h_s\odot\delta_!K\to h_s\odot\delta_!L$
is a Reedy cofibration for a finite inclusion $K\to L$.
A path between test arrows $v_0,v_1:s\to t$ induces
an inner homotopy of their entire restricted squares.
If a lift at $v_0$ is given, extend it across
\[
 \bigl((h_s\odot\delta_!L)\otimes\{0\}\bigr)
 \mathop{\cup}_{(h_s\odot\delta_!K)\otimes\{0\}}
 \bigl((h_s\odot\delta_!K)\otimes\Delta[1]\bigr)
 \longrightarrow (h_s\odot\delta_!L)\otimes\Delta[1].
\]
The last tensor is inner. This is the pushout product
of a Reedy cofibration with an endpoint trivial
cofibration, hence is Reedy trivial. Lifting against
$f$ transports the filler to $v_1$. The other endpoint
and finite chains of paths give invariance under the
test-arrow class. Together with strict precomposition,
this makes \emph{existence} of a lift a sieve on
$\operatorname{Ho}(\CT)$. Only existence is required
locally; compatibility of the chosen fillers is
unnecessary.

For a finite composite of horn attachments, choose
successive local fillers. The composites of their
covering sieves generate a covering sieve by transitivity.
Homotopy invariance handles actual representatives of
composite classes that are not the selected strict
composites.

Finally a finite anodyne inclusion $K\hookrightarrow L$
is a retract of a finite relative horn-cell complex.
To see the finite support, factor it by the countable
horn small-object construction as
$K\to W\to L$, with the second map a Kan fibration.
Its lifting property gives a section $s:L\to W$
relative to $K$. Finiteness of $L$ puts $s(L)$ in
one finite stage and finitely many cells. Include the
finitely many predecessor cells supporting their finite
attaching horns. Recursing strictly decreases the
stage number, so it terminates with a finite relative
horn-cell subcomplex containing $s(L)$. Restrict the
retraction to that subcomplex. Lifting there proves
the assertion for $K\to L$.
Each problem uses only finitely many refinements,
and its covering sieve may depend on that problem.
\end{proof}

\subsection{Enriched replacement and the local path comparison}

\begin{lem}[Enriched finite-limit-preserving Kan replacement]
The functor $R=\Ex^\infty$ acts on enriched presheaves,
with enriched natural unit, objectwise Kan values, and
enriched finite-limit comparisons.
\label{lem:enriched-ex-replacement}
It preserves the property of having local strict lifts
against every finite anodyne inclusion.
\end{lem}

\begin{proof}
At stage $R_m=\Ex^m$, with product-compatible unit
$\eta_m$, define the action by
\[
\begin{split}
 \Map_{\CT}(s,t)\times R_mF(t)
 &\longrightarrow
 R_m\Map_{\CT}(s,t)\times R_mF(t)\\
 &\cong R_m(\Map_{\CT}(s,t)\times F(t))
 \longrightarrow R_mF(s).
\end{split}
\]
The first arrow is $\eta_m\times1$. Both iterated
actions on composable arrows factor through
$R_m(\Map(r,s)\times\Map(s,t)\times F(t))$.
Naturality of $\eta_m$ with respect to composition,
product compatibility, and the original associativity
identify them as simplicial maps. Units follow from
$R_m(*)=*$. The same naturality proves that maps of
enriched presheaves and $F\to R_mF$ are enriched natural.

The standard transitions $\tau_m:R_m\to R_{m+1}$
are natural, product compatible, and satisfy
$\tau_m\eta_m=\eta_{m+1}$. They therefore intertwine
these actions. Their filtered colimit is an enriched
$RF$. Restriction along an actual arrow $v:s\to t$
is exactly $R_m(F(v))$ at a stage and $R(F(v))$ in
the colimit. Each $\Ex$ preserves finite limits,
and filtered colimits of simplicial sets commute with
finite limits. The comparisons respect the displayed
actions, giving in particular an enriched isomorphism
\[
 R(E\times_D C)\cong RE\times_{RD}RC.
\]
The ordinary simplicial Kan-replacement theorem makes
the unit objectwise weak and $RF$ objectwise Kan.

A finite lifting problem for $R_mp$ is adjoint to
one for $p$ along $\Sd^mK\hookrightarrow\Sd^mL$.
Subdivision preserves finite simplicial sets,
monomorphisms and weak equivalences, so a finite
anodyne map stays finite anodyne. Transpose its local
lifts; the actual-arrow formula verifies restriction
compatibility. A finite problem in $Rp$ factors through
one finite stage. Enlarge the stage if necessary until
its finitely many commutativity equations hold.
Thus each finite problem uses a cover chosen at a
finite replacement stage.
\end{proof}

\begin{lem}[Enriched boundary criterion]
Let $u:S\to H$ be a map of objectwise Kan enriched
presheaves.
\label{lem:enriched-boundary-criterion}
Suppose every representable-test boundary square for
$\partial\Delta[m]\hookrightarrow\Delta[m]$, $m\ge0$,
has local homotopy lifts that agree with its prescribed
$S$-boundary exactly and whose images are homotopic to
the given $H$-simplex relative to that entire boundary.
Then $u$ is an enriched local $\pi_*$-equivalence.
\end{lem}

\begin{proof}
The object
\[
 B_m=S^{\partial\Delta[m]}
       \times_{H^{\partial\Delta[m]}}H^{\Delta[m]}
\]
is an objectwise Kan model of the global derived
boundary target, since the cotensor restriction of
$H$ is an objectwise fibration. A vertex
$b\in B_m(t)_0$ is precisely a boundary square.
After restriction along each arrow $v:s\to t$ in a
suitable covering sieve, the local homotopy is a path
in $B_m(s)$ with the restricted
$S^{\partial\Delta[m]}$-coordinate fixed, from $v^*b$
to the image of a vertex of $S^{\Delta[m]}(s)$.
Thus $S^{\Delta[m]}\to B_m$ is a covering by
\cite[Proposition~3.1.4]{HAGI2005}. These are the
derived boundary-cover conditions of
\cite[Lemma~3.3.3, PDF leaf~24]{HAGI2005};
constant extension $S\to S^{\Delta[m]}$ is
objectwise weak and gives the same condition if the
criterion is expressed with source $S$.

Objectwise Kan is not injective fibrancy. These
derived cotensors and pullbacks are computed in the
global \emph{projective} enriched-presheaf model,
which has the same objectwise weak equivalences and
presents the same presheaf infinity-category.
The HAG criterion uses the original enriched comma
sites $\operatorname{Ho}(\CT/t)$ and their base-point
transport. Our tests include every boundary and base
point, and the homotopies fix the whole boundary.
Covering sieves pull back to those comma sites.
This applies the enriched criterion on those comma
sites, retaining the mapping spaces and base-point
transport.
\end{proof}

\begin{lem}[Canonical local pullback comparison]
Let $p:E\to D$ have local strict lifts against every
finite anodyne inclusion of simplicial sets.
\label{lem:enriched-local-homotopy-pullback}
For every $C\to D$, the canonical comparison
\[
 a_{\mathrm{hyp}}(E\times_D C)
 \longrightarrow
 a_{\mathrm{hyp}}E
 \times_{a_{\mathrm{hyp}}D}a_{\mathrm{hyp}}C
\]
is an equivalence.
\end{lem}

\begin{proof}
Put $E'=RE$, $D'=RD$, $C'=RC$, $p'=Rp$, and
$S=R(E\times_D C)=E'\times_{D'}C'$.
Since $S$ is an $R$-image, it is objectwise Kan.
The endpoint fibration of $D'$ gives
the objectwise Kan global homotopy-pullback model
\[
 H=E'\times_{D',\ev_0}(D')^{\Delta[1]}
                         \times_{D',\ev_1}C'.
\]
Let $j:S\to H$ insert constant paths.

A boundary problem for $j$ consists of simplices
$e:\Delta[m]\to E'$, $c:\Delta[m]\to C'$ and
a path $\gamma:\Delta[m]\times\Delta[1]\to D'$
from $p'e$ to the image of $c$. On the prescribed
boundary, $\gamma$ is constant and $(e,c)$ is the
given map to $S$. The inclusion
\[
 (\Delta[m]\times\{0\})
 \cup(\partial\Delta[m]\times\Delta[1])
 \hookrightarrow\Delta[m]\times\Delta[1]
\]
is finite anodyne, being the boundary--endpoint
pushout product. Prescribe $e$ at time zero and the
constant boundary lift. By
Lemma~\ref{lem:enriched-ex-replacement}, $Rp$ lifts
this problem locally to $\widetilde\gamma$.
Its endpoint gives
$(\widetilde\gamma(-,1),c):\Delta[m]\to S$
with the exact given boundary.

The remaining paths give an actual simplicial
homotopy in $H$:
\[
 t\longmapsto
 \bigl(\widetilde\gamma(-,t),c,\,
           [u\mapsto\gamma(-,\max(t,u))]\bigr).
\]
The maximum map $[1]\times[1]\to[1]$ is monotone.
At parameter $t$ the path begins at
$p'\widetilde\gamma(-,t)$ and ends at the image of
$c$. At $t=0$ it is the original simplex; at $t=1$
it is the constant-path image of the endpoint lift.
Everything is fixed on the entire prescribed boundary.
For $m=0$ the lifting inclusion is
$\{0\}\hookrightarrow\Delta[1]$: a point of $H$
already includes its starting point in $E'$, which
supplies the initial lift.

Lemma~\ref{lem:enriched-boundary-criterion} makes
this fixed $j$ a local $\pi_*$-equivalence.
Hypersheaf localization inverts it and the natural
units into $R$. Its established left exactness,
\cite[Proposition~3.4.10, PDF leaf~34]{HAGI2005},
identifies $a_{\mathrm{hyp}}H$ with the pullback
of the localized cospan. The cone projections and
the natural finite-limit comparison for $R$ identify
this map with the original canonical comparison,
naturally in the entire cospan. The local fillers
establish that this canonical comparison is an
equivalence.
\end{proof}

If $E$ or $C$ is the initial presheaf, both models in
the last proof are objectwise empty. Emptiness of a
strict pullback at an individual test, however, need
not imply emptiness of its path model. The result is
local, not an objectwise strengthening. An allowed
empty covering sieve has its usual vacuous local
meaning. The final Whitehead criterion is
hypersheaf-local; ordinary $\mathbf{AP}_\infty$
remains the separate question in
Section~\ref{sec:derived}.

References to HAG in this appendix use
arXiv:math/0207028v4, with PDF page locators.

%% file: split-open-hyperdescent.tex
\section{Split-open hyperdescent}
\label{app:split-open-hyperdescent}

We prove the refinement and localization assertions used in
Theorem~\ref{lem:banach-hypercover-refinement}, and the
chart-target descent used in
Theorem~\ref{thm:banach-representability}.
Fix the small Kan-enriched presentation $\mathcal C$
of the geometric test category in Section~\ref{sec:banach}.
Write
\[
 \mathbf P=\Fun_{\sSet}(\mathcal C^\op,\sSet),
 \qquad \mathcal P=\PSh(\mathcal C),\qquad
 h_X(T)=\Map_{\mathcal C}(T,X).
\]
All constructions initially take place in these original
enriched presheaves; hypersheaf membership will follow
from localization.
The site and all permitted index sets are small in the
standing compatible presheaf universe $\mathbb V$.
Global realization is denoted $|V|_{\mathcal P}$;
for the strict models below it is computed by
$\operatorname{diag}V$.

For a test $X\in\mathcal C$, a hypercover of $h_X$ is
\emph{semi-representable} when every level is equivalent
in $\mathcal P$ to a small coproduct of representables
\cite[Definition~3.4.8(1), PDF leaf~33]{HAGI2005}.
Semi-representability is imposed up to equivalence
in the global infinity-category. Strict indexed
coproduct models are constructed below.

\subsection{Monomorphic opens and exact slice matching}

\begin{lem}[Full-component models for strict opens]
Strict opens of a test $X$ give monomorphisms of
represented presheaves, and the strict-open families
form a basis for the specified topology.
\label{lem:banach-open-monomorphic-models}
For each strict open $W\to X$ there is a literal enriched
subpresheaf $O_X(W)\subset h_X$ with an objectwise weak
equivalence $h_W\to O_X(W)$ over $h_X$.
These models depend only on the open subobject,
respect inclusions, and satisfy the literal equality
\[
 O_X(W)\cap O_X(W')=O_X(W\cap_X W').
\]
Their inclusions into $h_X$ are objectwise Kan fibrations.
\end{lem}

\begin{proof}
Strict-open base change gives $W\times_X W\simeq W$
with its diagonal equivalence. Mapping out of any
test proves that $h_W\to h_X$ is $(-1)$-truncated.
Finite intersections are represented by the strict-open
intersections already in $\mathcal C$.
For a strict-open cover represented by subterminals
$a_i$, covering is $\bigvee_i a_i=1$.
Inverse image preserves this join and finite intersections.
Composition follows from
$\bigvee_{i,j}a_{ij}=\bigvee_i a_i=1$.
Together with identities and the empty cover of the
empty object, these are the pretopology axioms.
The specified topology is the generated topology,
so every covering sieve contains such a family.

Choose an actual representative $W\to X$ and let
$O_X(W)(T)$ be the union of the complete components
of the Kan complex $h_X(T)$ meeting $h_W(T)$.
A monomorphism of spaces has contractible nonempty
homotopy fibres, so $h_W(T)\to O_X(W)(T)$ is a
weak equivalence. A union of components of a Kan
complex is a Kan subcomplex whose inclusion is a
Kan fibration.

The full action
$\Map_{\mathcal C}(S,T)\times h_X(T)\to h_X(S)$
restricts to these subcomplexes. At each vertex,
a map connected to one factoring through $W$
remains so after composition, by the enriched action
on its connecting path. Every simplex in the product
therefore lands in a selected complete component.
This proves compatibility with every mapping-space
simplex.
Associativity and units are inherited from $h_X$,
and $h_W\to O_X(W)$ is enriched natural.

Equivalent open representatives select the same
components, and open inclusions give inclusions of
the models. A map to $X$ factors up to homotopy
through both opens exactly when its two factorizations,
with their paths to that fixed map, give a point of
the homotopy intersection. This identifies the
selected components and proves the literal intersection
formula. In particular
$O_X(\varnothing)\simeq h_\varnothing$ over $h_X$.
It is not the initial presheaf: at the empty test
it has a contractible nonempty value.
\end{proof}

An \emph{indexed open model} over $h_X$ is
$\coprod_{a\in I}O_X(W_a)$.
A \emph{blockwise} map specifies one target index
for each source index and uses the corresponding
literal open inclusion. The limits below use these
specified blockwise maps.

\begin{lem}[Blockwise limits and augmented matching]
A finite strict limit in $\mathbf P/h_X$ of indexed
open models and blockwise maps is
\[
 \coprod_{\mathbf a\in J}O_X(W_{\mathbf a}),
\]
where $J$ consists of compatible index tuples and
$W_{\mathbf a}$ is their finite open intersection.
\label{lem:banach-open-slice-matching}
Blockwise maps are objectwise Kan fibrations.
A simplicial blockwise diagram $V$ is Reedy fibrant
in the global projective slice, and its strict
slice matching computes
\[
 M_n^X(V)=
 h_X\times^h_{\{\partial\Delta[n],c h_X\}^{h}}
                    \{\partial\Delta[n],V\}^{h}.
\]
In particular $M_0^X(V)=h_X$ and
$M_1^X(V)=V_0\times^h_{h_X}V_0$.
\end{lem}

\begin{proof}
At a test $T$, decompose $h_X(T)$ into connected
Kan components $B_c$. An indexed open model has value
\[
 \coprod_c B_c\times I(c)
\]
for discrete label sets $I(c)$. A blockwise map is
the identity on $B_c$ times a map of sets.
Such a map is a Kan fibration even when its index
map is not surjective. A finite slice limit takes
compatible tuples of labels over each $B_c$; the
intersection formula of the preceding lemma gives
the claimed model. Incompatible tuples contribute
no summand. Compatible tuples with empty geometric
intersection retain their labelled
$O_X(\varnothing)$-summand.

In the global projective model, $h_X$ is cofibrant
by enriched Yoneda and fibrant by the Kan enrichment.
Its projective slice presents $\mathcal P/h_X$.
The global injective slice has the same underlying
objects, maps and objectwise weak equivalences as a
relative category, so presents the same infinity-slice.
The comparison uses the common objectwise weak
equivalences, while the fibrancy assertion is in
the projective slice.

The matching maps of $V$ are blockwise maps of the
type just calculated. Thus they are fibrations in
the projective slice, proving its Reedy fibrancy.
Ordinary matching there computes derived slice
matching. The displayed formula is the formula
for that limit in an infinity-slice. The empty
matching limit is its terminal object $h_X$, rather
than the terminal presheaf. The argument uses
slice Reedy fibrancy.
\end{proof}

\subsection{Coherent augmented refinement}

\begin{lem}[Extending a representable lift]
Let $q:P\to M$ be a global injective fibration and
let $a:h_W\to O_X(W)$ be an objectwise equivalence
over a prescribed matching target $M$.
\label{lem:banach-open-lift-extension}
A strict lift $\ell:h_W\to P$ over $M$ extends to
a strict map $\bar\ell:O_X(W)\to P$ over $M$,
with a homotopy $\bar\ell a\simeq\ell$ in that slice.
\end{lem}

\begin{proof}
In the simplicial model category $\mathbf P_{\mathrm{inj}}/M$
both source objects are cofibrant and $P\to M$ is
fibrant. Therefore precomposition is a weak equivalence
of Kan mapping spaces
\[
 \Map_{/M}(O_X(W),P)\longrightarrow\Map_{/M}(h_W,P).
\]
Choose a preimage of the component of $\ell$.
It supplies the asserted map and homotopy, with all
matching equations strict over $M$.
The map $a$ need not be monic: the mapping-space
equivalence, rather than a trivial-cofibration
lifting property, gives the chosen extension.
\end{proof}

\begin{lem}[Augmented split-open refinement]
Every semi-representable augmented hypercover
$U\to h_X$, $X\in\mathcal C$, has a coherent
refinement by a strict split diagram
\[
 V_n=\coprod_{[n]\twoheadrightarrow[k]}N_k,\qquad
 N_k=\coprod_{a\in I_k}O_X(W_{k,a}),
\]
with blockwise faces and strict opens $W_{k,a}\to X$.
\label{lem:banach-bounded-open-refinement}
The augmentation is a hypercover. The refinement
$V\to U$ need not be levelwise covering, and the
augmentation need not admit a section.
\end{lem}

\begin{proof}
Strictify the coherent augmented target in
$\mathcal P/h_X$ and choose a Reedy-fibrant replacement
$U^{\mathrm f}$ in the global injective \emph{slice
over $h_X$}. It is levelwise globally equivalent
to $U$ there, preserving the coherent augmentation,
semi-representability up to equivalence, and the
augmented hypercover condition. Its strict relative
matching maps are injective fibrations. Their targets
are fibrant over the objectwise Kan $h_X$ and hence
are objectwise Kan, as are the source levels.
Fibrancy is taken in the slice over $h_X$, rather
than imposed on the constant diagram $c h_X$.

Suppose the split diagram and its strict map to
$U^{\mathrm f}$ are constructed below degree $n$.
By Lemma~\ref{lem:banach-open-slice-matching},
\[
 M=M_n^{\mathrm{str},X}(V)
   =\coprod_{a\in J_n}O_X(W_a).
\]
The lower map supplies $M\to
M_U=M_n^{\mathrm{str},X}(U^{\mathrm f})$. Form
\[
 P=M\times_{M_U}U^{\mathrm f}_n\longrightarrow M.
\]
This is an injective fibration. Since the pulled-back
leg is an objectwise fibration between objectwise
Kan objects, this strict pullback computes
exactly the homotopy relative lifting target in
\eqref{eq:banach-relative-matching}.
It is a local cover by homotopy-base-change stability
of coverings \cite[Corollary~3.1.6]{HAGI2005}.

For each component choose $h_{W_a}\to O_X(W_a)\to M$.
The local-cover criterion and path lifting give
strict representable lifts on a covering sieve
\cite[Proposition~3.1.4 and Remark~3.1.5]{HAGI2005}.
Choose a strict-open covering family
$\{W_{a,b}\to W_a\}_b$ inside that sieve.
An actual composite $W_{a,b}\to W_a\to X$ gives
an equivalence $h_{W_{a,b}}\to O_X(W_{a,b})$
strictly over the fixed matching component.
Changing an arrow representative only changes
the restriction by an enriched homotopy, which
the objectwise fibration corrects by path lifting.
Lemma~\ref{lem:banach-open-lift-extension} now gives
the lift from $O_X(W_{a,b})$ over $M$.

Set $N_n=\coprod_{a,b}O_X(W_{a,b})$ and
\[
 L_nV=\coprod_{\substack{[n]\twoheadrightarrow[k]\\k<n}}N_k,
 \qquad V_n=L_nV\amalg N_n.
\]
On $L_nV$ the matching map and map to $U^{\mathrm f}_n$
are forced by lower degrees and target degeneracies.
On $N_n$ use the selected matching map and lift to
$P$. Their composite agrees with the canonical
latching-to-matching map. The Reedy extension
construction therefore supplies all faces and
degeneracies and a strict map to $U^{\mathrm f}$.
Faces of a new summand select its matching indices
and are open inclusions; faces of degenerate
summands are blockwise by the simplicial identities.

The restriction of $V_n\to M$ to $N_n$ already
covers each matching open. Thus the degree-$n$
hypercover condition holds. The strict slice
matching is derived matching by
Lemma~\ref{lem:banach-open-slice-matching}.
Countable induction gives the augmented hypercover
and a strict map $V\to U^{\mathrm f}$.
Compose in the augmented diagram infinity-category
with $U^{\mathrm f}\simeq U$. This retains all the
original augmented target coherence.

For clarity, the initial degrees are explicit.
At zero, $M_0=h_X$ and $V_0=N_0$ is an open cover;
the split decomposition refers to degeneracies,
rather than a section $h_X\to V_0$. At one,
\[
 M_1=\coprod_{a_0,a_1}
 O_X(W_{a_0}\cap_XW_{a_1}),\qquad
 V_1=s_0N_0\amalg N_1,
\]
with $d_1$ selecting $a_0$ and $d_0$ selecting $a_1$.
At two, matching indices are compatible triples
$(e_{01},e_{02},e_{12})$, including degenerate edges,
with the threefold open intersection. The latching
object is
\[
 L_2V=s_0N_1\amalg s_1N_1\amalg s_0s_0N_0,
 \qquad s_0s_0N_0=s_1s_0N_0,
\]
where the last summand is counted once. If
$d_1e=a$ and $d_0e=b$, the faces are
\[
 \begin{array}{c|ccc}
       &d_0&d_1&d_2\\ \hline
 s_0e  &e&e&s_0a\\
 s_1e  &s_0b&e&e
 \end{array}.
\]
The totally degenerate vertex has three faces $s_0a$.
New degree-two faces and their target coherence
come from the single matching lift. The Reedy
extension proves the higher identities as well.
Empty-intersection labels remain throughout this
presheaf construction; only incompatible tuples
are omitted.
\end{proof}

\subsection{Uniform source size and enriched localization}

\begin{lem}[A bounded set of augmentations]
Let an infinite $\mathbb V$-small cardinal $\alpha$
bound the objects and all mapping-space simplices
of the fixed site, including $\aleph_0$.
\label{lem:banach-split-open-size}
The refinement above can be chosen with
$|\coprod_n I_n|\le\alpha$ and total source
presheaf data of size at most $\alpha$.
After normalizing indices, its source augmentations
belong to a set $\mathscr S$ of at most $2^\alpha$
maps $|V|_{\mathcal P}\to h_X$.
\end{lem}

\begin{proof}
Open subobjects of a test $X$ are a quotient of a
subset of the actual arrows $W\to X$, so there
are at most $\alpha$. Each $O_X(W)$ is a
subpresheaf of $h_X$ of total size at most $\alpha$.
A covering family can be reduced to at most
$\alpha$ distinct opens by choosing one lift for
each retained member.

Inductively, in a fixed degree there are finitely
many degeneracy surjections and a matching index
is a finite compatible tuple. If the lower
nondegenerate index sets have size at most $\alpha$,
so does $J_n$. Choosing at most $\alpha$ opens
over each matching component gives $|I_n|\le\alpha$,
retaining separate copies at different matching
indices. Countably many degrees still have size
at most $\alpha$. All source actions are inherited
from $h_X$, and the blockwise maps are their fixed
inclusions.

Normalize each $I_n$ to a subset of a fixed set
of size $\alpha$. The test $X$, countably many
index subsets and open labels, and finitely many
face-index functions per degree have at most
$2^\alpha$ codes. Impose the simplicial identities
and the condition that source opens cover every
augmented matching open. These restrictions define
the set $\mathscr S$. A bounded refinement is
reindexed into it by an actual isomorphism over
$h_X$. The diagonal sources still have size
at most $\alpha$.

Generator data consist of source augmentations.
Extending the chosen lifts adds no source cells.
Target replacements and maps into target hypercovers
remain in the original universe but need not be
uniformly bounded. An auxiliary regular small cardinal
$\lambda_{\mathrm{SO}}>2^\alpha$ is available in that same
Grothendieck universe.

The latching maps of $V$ are summand inclusions,
so it is Reedy cofibrant in the global injective
model; its diagonal computes $|V|_{\mathcal P}$.
Each generator is a local $\pi_*$-equivalence by
\cite[Lemma~3.4.2, PDF leaves~27--28]{HAGI2005}.
Include $0\to h_\varnothing$: the zero augmented
diagram has $M_0=h_\varnothing$ and $M_n=0$ for
$n>0$. Its zero-degree cover is the permitted
empty family. This generator forces a local
object's empty-test value to be contractible.
\end{proof}

\begin{thm}[Enriched split-open localization]
The set $\mathscr S$ satisfies
\[
 L_{\mathscr S}\mathbf P_{\mathrm{inj}}
       =\mathbf P_{\mathrm{inj},\tau}
\]
as model localizations under the original enriched
presheaf category.
\label{lem:banach-split-open-localization}
It therefore generates the hypersheaf localization,
equivalently localization at all semi-representable
augmented hypercovers.
\end{thm}

\begin{proof}
The global injective model is left proper and
combinatorial in the standing universe. The diagram
model and small-set localization use precisely
\cite[Proposition~A.1.3, Theorems~A.2.2 and~A.2.4,
and Proposition~A.2.5, PDF leaves~68--69]{HAGI2005}.
Every $\mathscr S$-map is locally weak, so the
already established hypersheaf model is a further
localization of $\mathbf P_{\mathscr S}
=L_{\mathscr S}\mathbf P_{\mathrm{inj}}$.

Let $f:F\to G$ be a $\mathbf P_{\mathscr S}$-fibration
between $\mathbf P_{\mathscr S}$-fibrant objects
and also a local $\pi_*$-equivalence. It is a
global injective fibration with globally fibrant
source and target. For a boundary problem at $h_X$
put
\[
 A=F^{\Delta[m]},\qquad
 B=F^{\partial\Delta[m]}
       \times_{G^{\partial\Delta[m]}}G^{\Delta[m]},
 \qquad p:A\to B.
\]
The problem is a vertex $b:h_X\to B$.
The simplicial model axiom makes $p$ an injective
fibration. Both $A$ and $B$ are globally fibrant:
the defining pullback uses the fibration
$G^{\Delta[m]}\to G^{\partial\Delta[m]}$.
They are also $\mathscr S$-local. Cotensors
preserve locality by the mapping-space adjunction,
and local objects are closed under homotopy limits.

To see that $p$ is locally weak, apply the left-exact
hypersheaf localization. Cotensors by a finite
simplicial set are built from finitely many
homotopy pullbacks: its finite cell attachments
become pullbacks, and each simplex is contractible.
Thus $a_{\mathrm{hyp}}$ preserves these cotensors.
Since $a_{\mathrm{hyp}}f$ is an equivalence,
so is $a_{\mathrm{hyp}}p$.
At $m=0$, these formulas are $A=F$, $B=G$, $p=f$.

We use the framed version of
\cite[Lemma~3.4.3, PDF leaves~29--32]{HAGI2005}.
Choose a projective fibrant replacement
$h_X\hookrightarrow R(h_X)$ that is a projective
trivial cofibration; here $R$ denotes this choice,
not the $\Ex^\infty$ functor of the preceding appendix.
Extend $b$ across it using projective fibrancy of $B$.
The lemma supplies
\[
 H_\bullet\longrightarrow R(h_X)^{\Delta[\bullet]}
\]
and a lift $|H_\bullet|_{\mathcal P}\to A$ over $B$.
The chosen $H_n$ are globally equivalent to
coproducts of representables, as specified in the
paragraph preceding Definition~3.4.8 of that source
(PDF leaf~33). This semi-representability statement
concerns the levels $H_n$. In the original augmented
diagram infinity-category set
\[
 U_\bullet=
 H_\bullet\times^h_{R(h_X)^{\Delta[\bullet]}}c h_X.
\]
The constant-extension leg from $c h_X$ is a
levelwise global equivalence, because
$h_X\to R(h_X)$ is such an equivalence and the
simplices are contractible. Hence $U_n\simeq H_n$,
so each $U_n$ inherits semi-representability from
$H_n$. The augmented hypercover condition also
survives, and realization gives the coherent
lift over the original $b$.

Choose a bounded refinement $V\to U$ by
Lemma~\ref{lem:banach-bounded-open-refinement}
and normalize it by
Lemma~\ref{lem:banach-split-open-size}.
The coherent lift over $b|_V$ can be made strict:
$|V|_{\mathcal P}$ is injective-cofibrant, and
the map
\[
 \Map(|V|_{\mathcal P},A)\longrightarrow
 \Map(|V|_{\mathcal P},B)
\]
is a Kan fibration,
so lift its specified comparison path.
Now consider
\[
 \begin{tikzcd}
 \Map(h_X,A)\arrow[r]\arrow[d]
   &\Map(|V|_{\mathcal P},A)\arrow[d]\\
 \Map(h_X,B)\arrow[r]
   &\Map(|V|_{\mathcal P},B).
 \end{tikzcd}
\]
All four spaces are Kan, the vertical maps are
fibrations, and the horizontal maps are weak
equivalences by $\mathscr S$-locality. The strict
fibres over $b$ and $b|_V$ therefore model
equivalent homotopy fibres. The right fibre is
nonempty, so the left is nonempty. This gives
an actual lift of $b$ and, by enriched Yoneda,
a strict filler of the original boundary problem.
Thus every $f(X)$ is a trivial Kan fibration.

The abstract comparison criterion
\cite[Lemma~6.3, preprint pp.~16--17]{DHI2004}
now applies. Explicitly, replace a locally weak
map by a map between $\mathscr S$-fibrant objects
and factor it as an $\mathscr S$-trivial
cofibration followed by an $\mathscr S$-fibration.
The latter is locally weak, hence objectwise
weak by the preceding argument. Two-out-of-three
makes the original map $\mathscr S$-weak.
The two models have the same weak equivalences
and cofibrations, and therefore the same fibrations.
The hypercover-localization identification of
\cite[Theorem~3.4.1, Corollaries~3.4.5 and~3.4.7,
and Proposition~3.6.1]{HAGI2005} gives the stated
final description. The DHI comparison criterion
is used at the model-category level; HAG supplies
the enriched hypercover localization.
\end{proof}

\subsection{Coherent internal hypercovers and chart descent}

\begin{lem}[Internal hypercovers and chart targets]
For each generator over a test $X$, replace its
open labels by their subterminals in the underlying
infinity-topos $\mathcal X_X$.
\label{lem:banach-internal-open-hypercover}
This gives an internal hypercover of $1_X$.
The whole coherent chart-evaluation diagram
identifies its descent with descent of $h_Y$
along the generator, for a chart target
$Y=\operatorname{Chart}(D)$.
Consequently every such $h_Y$ is a hypersheaf.
\end{lem}

\begin{proof}
Let $a_W\hookrightarrow1_X$ correspond to $W\to X$.
Retain all component indices and their face and
degeneracy functions, and form
\[
 A_n=\coprod_{a\in I(V_n)}a_{W_a}.
\]
These coproducts are permitted by the standing
universe convention. They are $0$-truncated.
Disjointness and universality of coproducts
compute finite matching by compatible index
tuples and intersections of subterminals.
Thus the augmented matching indices are exactly
those already constructed. Each chosen open
family covers its matching open, giving an
effective epimorphism onto each matching summand.
This proves the internal hypercover condition,
including degree zero. An empty open becomes
initial only in this internal interpretation;
its labelled $O_X(\varnothing)$ was retained
in the presheaf matching. For the empty base
topos, initial and terminal coincide, so the
empty augmentation also has the asserted property.

The strict-open subobjects over $X$ form a poset
infinity-category: their factorization spaces
are empty or contractible. The representables
$h_W$ and full-component models $O_X(W)$ define
naturally equivalent diagrams over this open
poset in $\mathcal P/h_X$. Extending by the
recorded indexed coproducts identifies the
\emph{whole} simplicial diagram, including its
face and degeneracy maps. The contractible
factorization spaces of opens supply this coherence.

The natural chart-evaluation formula gives
\[
 h_Y(W)\simeq
 \Map_{\mathcal X_X}(a_W,\mathcal O_X(D)).
\]
Therefore the entire cosimplicial comparison is
\[
\begin{split}
 \Map_{\mathcal P}(|V|_{\mathcal P},h_Y)
 &\simeq \Tot_{[n]}
       \prod_{a\in I(V_n)}h_Y(W_a)\\
 &\simeq \Tot_{[n]}
       \Map_{\mathcal X_X}(A_n,\mathcal O_X(D))\\
 &\simeq \Map_{\mathcal X_X}(1_X,\mathcal O_X(D))
 \simeq h_Y(X).
\end{split}
\]
The penultimate equivalence is internal hypercover
descent for the hypercomplete \emph{structure value}
$\mathcal O_X(D)$. The underlying topos need not
be hypercomplete. Hence $h_Y$ is $\mathscr S$-local,
and Theorem~\ref{lem:banach-split-open-localization}
makes it a hypersheaf.
\end{proof}

The uniform bound concerns augmentations based on
tests in $\mathcal C$, not atlases of arbitrary
larger geometric targets. Their ordinary-sheaf
subobject, atlas and full mapping-space arguments
are supplied in the proof of
Theorem~\ref{thm:banach-representability}.
All HAG locators in this appendix refer to
arXiv:math/0207028v4; the DHI locator refers to
arXiv:math/0205027v2. Page locators follow these
preprint editions.